\RequirePackage{fix-cm}
\documentclass[11pt,openany,twoside,letterpaper]{report}

\usepackage[utf8]{inputenc}
\usepackage[T1]{fontenc}
\usepackage{amsfonts}
\usepackage{amsmath}
\usepackage{amsbsy}
\usepackage{amstext}
\usepackage{amsxtra}
\usepackage{amsthm}
\let\origlrcorner\lrcorner 
\usepackage{mathabx}
\let\lrcorner\origlrcorner
\usepackage{amssymb}
\usepackage[heavycircles]{stmaryrd}
\usepackage{esdiff}
\usepackage{titlesec}
\usepackage[all]{xy}
\usepackage{graphicx}
\usepackage[colorlinks=true,pdftex,bookmarksopen,colorlinks]{hyperref}
\usepackage{enumitem}

\usepackage[margin=1in]{geometry}

\newcommand\sect[2][\empty]{\Gamma_{#1}\left({#2}\right)}
\def\Z{\mathbb{Z}}
\def\R{\mathbb{R}}

\def\ins{\lrcorner}
\def\z{\Z_2}

\def\im{\operatorname{Im}}
\def\ev{\mathrm{ev}}
\def\tr{\operatorname{tr}}
\def\sgn{\operatorname{sgn}}
\def\End{\operatorname{End}}

\def\Hom{\operatorname{Hom}}

\def\Conn{\operatorname{Conn}}
\def\Inc{\operatorname{Inc}}
\def\Sym{\operatorname{Sym}}
\def\der{\operatorname{der}}
\def\sder{\operatorname{sder}}
\def\Der{\operatorname{Der}}
\newcommand{\twistcom}[3]{\left[{#1}\empty_{#2}{#3}\right]}
\newcommand{\abs}[1]{\left|{#1}\right|}
\newcommand{\parity}[1]{\lceil{#1}\rfloor}
\def\Sym{\operatorname{Sym}}
\def\Jet{\operatorname{Jet}}
\def\jet{\operatorname{jet}}

\def\Aut{\operatorname{Aut}}
\def\id{\operatorname{id}}
\def\ad{\operatorname{ad}}
\def\lsem{\llbracket}
\def\rsem{\rrbracket}
\def\pr{\operatorname{pr}}
\def\rank{\operatorname{rank}}
\def\Pol{\operatorname{Pol}}
\def\ker{\operatorname{ker}}
\def\coker{\operatorname{coker}}

\newcommand\superfold[1]{({#1}|\mathcal{R}{#1})}
\newcommand\superfoldi[1]{({#1}|\mathbf{S}{#1})}
\newcommand\C{\mathcal{C}^{\infty}}
\newcommand{\set}[2][\empty]{\left\{{#2}\right\}_{#1}}
\newcommand{\tupla}[2][n]{\left({#2}_1,\ldots{#2}_{#1}\right)}
\newcommand{\tuplao}[2][n]{\left({#2}_0,\ldots{#2}_{#1}\right)}
\newcommand{\listuplao}[2][n]{{#2}_0,\ldots{#2}_{#1}}
\newcommand{\setupla}[2][n]{\left\{{#2}_1,\ldots{#2}_{#1}\right\}}
\newcommand{\listupla}[2][n]{{#2}_1,\ldots,{#2}_{#1}}
\newcommand{\prodtupla}[2][n]{{#2}_1\dotsm{#2}_{#1}}
\newcommand{\inv}[1]{{#1}^{-1}}

\newtheoremstyle{mit}%
{3pt}
{3pt}
{}
{}
{\itshape}
{:}
{1em}
{}

\swapnumbers
\theoremstyle{definition}
\newtheorem{defn}{Definition}[chapter]
\newtheorem{nota}[defn]{Remark}
\newtheorem{ejemplo}[defn]{Example}

\theoremstyle{plain}
\newtheorem{teo}[defn]{Theorem}
\newtheorem{thm}[defn]{Theorem}
\newtheorem{lema}[defn]{Lemma}
\newtheorem{lem}[defn]{Lemma}
\newtheorem{cor}[defn]{Corollary}
\newtheorem{prop}[defn]{Proposition}

\newtheorem*{cartan-poincare}{Cartan-Poincaré Lemma}
\newtheorem*{ccr}{Cannonical Commutiation and Anticommutation relations}

\theoremstyle{mit}
\newtheorem*{claim}{Claim}

\begin{document}

\begin{titlepage}

\begin{minipage}{\textwidth}
\includegraphics[scale=0.38]{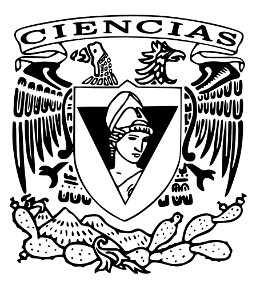}
\hfill
\includegraphics[scale=1.1]{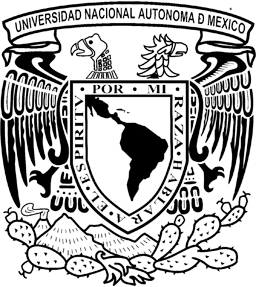}
\vfill
\rule{\textwidth}{2pt}
\end{minipage}
\vfill
\begin{minipage}[][0.56\textheight][s]{\textwidth}
\begin{center}
{\fontsize{25}{15}\selectfont%
\textbf{A classical approach to smooth supermanifolds}}\\
\vfill
{\fontsize{20}{30}\selectfont%
Óscar Guajardo}\\
\vfill
{\fontsize{18}{30}\selectfont {A thesis submitted to the Faculty of Sciences of the National Autonomous University of Mexico in partial fulfillment of the requirements for the degree of}}\\
\vfill
{\fontsize{19}{30}\selectfont \textbf{Master in Science}}\\
\vfill
{\fontsize{18}{30}\selectfont Thesis supervisor:\\ Dr. Gregor Weingart\\ Institute of Mathematics at Cuernavaca, UNAM}
\vfill
\end{center}
\end{minipage}
\vfill
\rule{\textwidth}{2pt}
\vfill
\end{titlepage}

\tableofcontents

\begin{abstract}
This is an extension of the author's Master's thesis written under the supervision of Dr. Gregor Weingart at the National Autonomous University of Mexico. 
By ``extension'' we mean that some of the results have been rewritten and some others have been added to the original work. The purpose of this study is to rewrite differential supergeometry in terms of classical differential geometry. This rewriting from ``first principles'' has two main motivations:
\begin{enumerate}
\item avoid using local (and usually not very well-defined) odd coordinates;
\item use both the language and the tools (both highly developed) of classical differential geometry to state and prove results in supergeometry.
\end{enumerate}

Although there is work in this direction (for instance \cite{adolfo-1}) this work's point of view might be useful to translate from the sheaf-theoretic language to one that is better suited for explicit calculations. We now give a summary of the contents. 

Chapter \ref{chap:ext} is about superalgebras. In particular, we give explicit isomorphisms for the space of superderivations of an exterior algebra.  The final section of this chapter is devoted to the twisted action of the symmetric group on the supervector space of tensors with a fixed rank; this allows us to construct the supersymmetric and superexterior algebras of a finite-dimensional supervector space.

In chapter \ref{chap:superfolds} we begin our study of smooth finite-dimensional supermanifolds. As the title of this monograph indicates, our approach will be rather classical in the following sense: our definition of a supermanifold is not given in terms of local charts nor in terms of sheaves of superalgebras; we rather study superalgebra bundles over smooth manifolds. That is, supermanifolds in our sense are vector bundles such that the fibre at each point is a free supercommutative superalgebra of finite rank. This approach allows the use of tools from differential geometry which in many cases (e.g. Batchelor's theorem, which we prove as corollary \ref{cor:batchelors}) simplifies the proofs; furthermore, we prove (theorem \ref{thm:sheaf=bundle}) that both approaches are equivalent. A noteworthy feature of our approach is that supersmooth maps turn out to be generalizations of linear differential operators (proposition \ref{prop:supermaps-are-diffops}). With our approach we prove special splittings of the tangent bundle (theorem \ref{thm:exact-sequence-derivations}) and study the tangential maps of a supersmooth map between two supermanifolds which give conditions for the supermanifold to be split (proposition \ref{prop:split-iff}). We also prove that the de Rham cohomology of a supermanifold is isomorphic to the de Rham cohomology of the underlying smooth manifold (theorem \ref{thm:deRham}) by writing down the exterior derivative in a different way (\ref{thm:super-exterior-derivative}).

We include three appendices. The first one concerns linear differential operators; there we develop the necessary tools to understand supersmooth maps with our approach. 

The two other appendices deal with algebraic results of independent interest.

The second appendix is the statement and proof of the Cartan--Poincaré lemma, a result on the (co)homology of a complex that arises in our study. 

The third appendix deals with the proof of an algebraic fact (lemma \ref{lemma:derivations-isomorphism}) we use in order to prove a ``flowbox coordinates'' theorem for supermanifolds (theorem \ref{thm:flowbox}) which is the foundation of the main results of chapter \ref{chap:superfolds}.

\end{abstract}

\chapter{Superlinear algebra}\label{chap:ext}
This first chapter is aimed at establishing the definitions and properties needed for doing superlinear algebra a.k.a supergeometry of a one--point manifold. 
First we study supervector spaces and superalgebras, later turning our attention to exterior algebras of finite-dimensional vector space. An explicit calculation of the space of derivations of an exterior algebra is deduced on section \ref{sec:derivations}; we then study graded (or super) derivations of these algebras. Section \ref{sec:lie-superalgebras} is on Lie superalgebras. The last section is devoted to the study of supersymmetric and superexterior algebras; we do this using a special action of the symmetric group on the supertensor algebra.

All vector spaces are finite-dimensional and defined over the real numbers unless stated otherwise.

\section{$\z$-graded spaces and supervector spaces}

\begin{defn}
A \textbf{$\z$-graded vector space} is a vector space \(W\) with a direct sum decomposition \(W=W_+\oplus W_-\).
\end{defn}

To obtain a supervector space an additional structure is needed for the algebra of endomorphisms of \(W\):

\begin{defn}
A \textbf{supervector space} is a \(\z\)-graded vector space \(W\) such that the space of endomorphisms is also endowed with a \(\z\)-grading:
\[
\End(W)=\End_+(W)\oplus\End_-(W)
\]
where
\[
\End_+(W)=\End(W_+)\oplus\End(W_-)\qquad\End_-(W)=\Hom(W_+,W_-)\oplus\Hom(W_-,W_+)
\]
in which the commutator operation is replaced by the \textbf{supercommutator}:
\begin{equation}\label{eq:supercommutator}
\lsem A,B\rsem\colon =A\circ B-(-1)^{\parity{A}\parity{B}}B\circ A;
\end{equation}
here $\parity{\cdot}$ denotes the \textbf{parity} of an element of the set $(W_+\cup W_-) -\set{0}$, defined as 
\begin{equation}\label{eq:parity}
\parity{w}=
\begin{cases}
0,\quad\text{if\ }w\in W_+;\\
1,\quad\text{if\ }w\in W_-
\end{cases}
\end{equation}
and likewise for endomorphisms of $W$. We denote a supervector space by \(W=(W_+|W_-)\) and the whole space of endomorphisms by $\End(W_+|W_-)$. The elements of $W_+$ (resp. $W_-$) are called \textbf{even} (resp. \textbf{odd}). The set $(W_+\cup W_-) -\set{0}$ is the set of \textbf{homogeneous elements} of $W$. For endomorphisms of supervector spaces we use the terms \textbf{even endomorphism} and \textbf{odd endomorphism} accordingly.
\end{defn}

\begin{nota}\label{remark:notation}
For notational purposes it will sometimes be convenient to denote a supervector space $(W_+|W_-)$ as $(W_0|W_1)$. This latter notation is consistent with the usual notation for gradings by a group, being in this case $\z$. On some other occasions (from chapter \ref{chap:superfolds} onwards) we'll denote a supervector space with decomposition $V\oplus U$ by $(V|U)$.
\end{nota}

A fundamental property of supervector spaces is given in the following

\begin{prop}\label{prop:superstructure-morphism}
Let $W$ be a vector space. There is a bijection between the set of supervector space structures on $W$ and automorphisms $\gamma$ of $W$ such that $\gamma^2=\id$. In this way, we get 
\begin{equation}\label{eq:superspace-grading}
W_\pm=\set{v\in W|\gamma(v)=\pm v}.
\end{equation}
\end{prop}
\begin{proof}
If  $W=(W_+|W_-)$ is a supervector space then the automorphism $\gamma$ is defined by setting it as in \eqref{eq:superspace-grading} and extending it linearly. Note that this induces a similar automorphism $\gamma^*$ on $\End(W_+|W_-)$ by precomposition:
\begin{equation}\label{eq:gamma-extended}
\gamma^* A=A\circ\gamma
\end{equation}

Conversely, let $\gamma$ be an automorphism of a vector space $W$ such that $\gamma^2=\id_W$. Then the decomposition \eqref{eq:superspace-grading} is a $\z$-grading of $W$. The supervector space structure is given by twisting the action of $\End W$ by $\gamma$ as in \eqref{eq:gamma-extended}. Then, for all endomorphisms $A$ and $B$ of $W$ we define
\begin{equation}
\lsem A,B\rsem=AB-\gamma^* A\gamma^* B
\end{equation}\label{eq:superbracket-gamma}
and thus get a decomposition
\[
\End(W)_+=\set{A\mid \gamma^*A=A},\quad \End(W)_-=\set{A\mid \gamma^*A=-A}
\]
for which the definition of the parity of a superlinear map makes sense is behaves as expected. Thus we get a unique supervector space structure on $W$ given by $\gamma$.
\end{proof}

\subsection{Superalgebras}
Of course the gradings are useless without a multiplicative structure. For instance, the algebra $\End(W_+|W_-)$ of endomorphisms of a supervector space satisfies the following: if $T^{\pm}\in\End_{\pm}(W)$ then the following identities hold:
\[
T^+\circ T^+\in\End_+(W),\quad T^-\circ T^-\in\End_+(W),\quad T^{\pm}\circ T^{\mp}\in\End_-(W)
\]
This is a very important example of the following concept.

\begin{defn}
Let $\mathcal{A}=(\mathcal{A}_0|\mathcal{A}_1)$ (see note \ref{remark:notation}) be a supervector space and $m\colon\mathcal{A}\otimes\mathcal{A}\to\mathcal{A}$ be a bilinear multiplication turning it into an algebra. The space $\mathcal{A}$ is a \textbf{superalgebra} if the multiplication is $\z$-graded, that is:
\[
m(\mathcal{A}_\mu\otimes\mathcal{A}_\nu)\subseteq\mathcal{A}_{\mu+\nu}
\]
where the indices are elements of the group $\z$. The algebra $\mathcal{A}$ is called \textbf{supercommutative} if for all $a$ and $b$ homogeneous elements of $\mathcal{A}$ the identity
\[
a\cdot b=(-1)^{\parity{a}\parity{b}}ba
\]
\end{defn}

Note that there is a priori no essential distinction between a superalgebra and a $\z$-graded algebra. A distinction will be fundamental when discussing Lie superalgebras. Notice that proposition \ref{prop:superstructure-morphism} still holds for superalgebras, in this case $\gamma$ being an algebra morphism. This has an important consequence:

\begin{lema}[The rule of signs]\label{rule-of-signs}
If $\mathcal{B}$ and $\mathcal{B}'$ are (left) supermodules over the superalgebra $\mathcal{A}$ (that is, $\mathcal{B}$ is a supervector space and a superlinear action $\mathcal{A}\otimes\mathcal{B}\to\mathcal{B}$ is defined, and likewise for $\mathcal{B}'$) then for any homogeneous homomorphism $T\colon\mathcal{B}\to\mathcal{B}'$ we have
\[
T(ab)=(-1)^{\parity{a}\parity{T}}aT(b)
\]
for all homogeneous $a\in\mathcal{A}$.
\end{lema}
\begin{proof}
One only has to observe that $M_a=a\cdot$ (left multiplication by $a$) is a module endomorphism and apply \eqref{eq:superbracket-gamma} to $T$ and $M_a$.
\end{proof}

The following example is the fundamental one for all our subsequent work:

\begin{ejemplo}
Let $V$ be a vector space of dimension $n$. The \textbf{exterior algebra} of $V$ is the quotient
\[
\Lambda V=\bigotimes V/\langle x\otimes x|x\in V\rangle
\]
The product on this algebra is denoted by $X\wedge Y$, where $X$ and $Y$ are classes in the quotient. If $x$ is an element of $V$ and we denote by the same symbol its class both in $\bigotimes V$ and $\Lambda V$ then the above definition forces the identity 
\begin{equation}\label{eq:wedge-skew-symmetric}
x\wedge x=0. 
\end{equation}
Let $T$ be a homogeneous tensor of rank $k$ expresible as $T=x_1\otimes\cdots\otimes x_k$, with each $x_j$ in $V$. Then its class in the exterior algebra is $x_1\wedge\cdots\wedge x_k$. These elements are called \textbf{decomposible $\mathbf{k}$-multivectors} and their span is denoted $\Lambda^k V$ ; note that, if $\pi\colon\bigotimes V\to\Lambda V$ is the projection, then $\pi(\bigotimes^k V)$ is generated by the latter classes. Because the algebra $\bigotimes V$ is $\Z$ graded the exterior algebra inherits this grading, that is:
\[
\Lambda^k V\wedge\Lambda^r V\subseteq\Lambda^{k+r}V
\] 
If $\setupla{v}$ is a basis of $V$ then the set $\set{v_{j_1}\wedge\cdots\wedge v_{j_r}|j_1\leq\cdots\leq j_r}$ is a basis for $\Lambda^r V$. It is also important noting that because of identity \eqref{eq:wedge-skew-symmetric} we get the following identity for decomposable multivectors:
\begin{equation}\label{eq:wedge-product-permutation}
x_{j_1}\wedge\cdots\wedge x_{j_r}=\sgn\sigma\cdot x_{\sigma(j_1)}\wedge\cdots\wedge x_{\sigma(j_r)}
\end{equation}
where $\sigma\in S_r$ is a permutation of the indices and $\sgn\sigma$ denotes the signature of $\sigma$ understood as taking the values $\pm 1$. Another important consequence of \eqref{eq:wedge-skew-symmetric} is that if $S=\setupla[r]{v}$ is a set of vectors in $V$ then $v_1\wedge\cdots\wedge v_r\neq 0$ if and only if the set $S$ is linearly independent; this immediately implies $\Lambda^r V=\set{0}$ if $r>n$. The $\z$-grading of $\Lambda V$ is given by
\[
\Lambda_+ V=\bigoplus_{k\geq 0}\Lambda^{2k} V,\quad\Lambda_- V=\bigoplus_{k\geq 0}\Lambda^{2k+1} V
\]
The multiplication satisfies a very important identity. Let $X\in\Lambda^q V$ and $Y\in\Lambda^p V$, then their product lies on $\Lambda^{p+q} V$ and it satisfies
\begin{equation}\label{eq:wedge-supercommutative}
X\wedge Y=(-1)^{p+q}Y\wedge X.
\end{equation}
Note that this identity is a rather direct consequence of \eqref{eq:wedge-product-permutation}.

When considering the exterior algebra $\Lambda V^*$ of the dual space we get a structure that codifies multiniear skew-symmetric mappings $\lambda\colon V\otimes\cdots\otimes V\to\R$; in particular there is a one--to--one correspondence between sets of oriented bases of $V$ and non-zero elements of $\Lambda^n V^*$.

Note that $\Lambda V$ can be presented as
\begin{equation}\label{eq:exterior-algebra-free}
\Lambda V=\langle \listupla{x} | x_\mu\cdot x_\nu=-x_\nu\cdot x_\mu\rangle
\end{equation}
where each $x$ represents a class of an element in $V$ and the set $x_1,\ldots,x_n$ is a basis for $V$.
\end{ejemplo}

Identity \eqref{eq:wedge-supercommutative} above is graded commutativity. Note that if $p$ is even and $q$ is odd in \eqref{eq:wedge-supercommutative} then the degree of their product is odd, whereas if both are of the same parity then their product is necessarily even. Thus an exterior algebra is a freely generated supercommutative algebra due to equality \eqref{eq:exterior-algebra-free}. So it's important to ask whether there are other possibilities for a free supercommutative algebra. The following result and its corollary guarantee that there are none.

\begin{prop}\label{prop:free-superalgebra}
A supercommutative free superalgebra $\mathcal{A}$ of even rank $m$ and odd rank $n$ (i.e. generated by $m$ even elements and $n$ odd elements) is isomorphic to $\Sym W\otimes\Lambda V$, where $\dim W=m$ and $\dim V=n$. 
\end{prop}
\begin{proof}
Let $\set{x_1,\ldots,x_m,\xi_1,\ldots,\xi_n}$ be a set of generators of $\mathcal{A}$, where the $x_\nu$ are even and the $\xi_\mu$ are odd, and fix bases $\setupla[m]{w}$ and $\setupla{v}$ of $W$ and $V$ respectively. The isomorphism is given by the linear extension of the map $v_\mu\mapsto\xi_\mu$, $w_\nu\mapsto x_\nu$, since the $\xi$'s and $x$'s freely generate $\mathcal{A}$.
\end{proof}

\begin{cor}\label{cor:free-superalgebra}
A completely odd finite-dimensional superalgebra $\mathcal{A}$ (i.e., with notation as above, $m=0$) is isomorphic to the exterior algebra of a finite-dimensional vector space.
\end{cor}

Since our main concern will be finitely and purely odd generated supercommutative algebras, we now turn our attention to exterior algebras and their structure. 

Let $V$ be a vector space of dimension $n$. We denote by $\Lambda^{\geq 1} V$ the subset of $\Lambda V$ of multivectors of degree greater than or equal to one, and note that this is an ideal. Since $\Lambda V$ is $\Z$-graded it comes with a natural filtration
\begin{equation}\label{eq:filtration}
\Lambda V:=\Lambda^{\geq 0}V\supset\Lambda^{\geq 1} V\supset\Lambda^{\geq 2}V\supset\cdots\supset\Lambda^{\geq n}V\supset\set{0},
\end{equation}
where we use the notation $\Lambda^{\geq k} V:=\left( \Lambda^{\geq 1}V\right)^k$; this is just the set (and ideal) of multivectors of degree greater than or equal to $k$. 

\begin{defn}
The projection $\varepsilon:\Lambda V\to\Lambda V/\Lambda^{\geq 1} V$ is called the \textbf{augmentation map}.
\end{defn}

\begin{lema}
The ideal $\Lambda^{\geq 1} V$ is the unique maximal ideal of $\Lambda V$.
\end{lema}
\begin{proof}
Observe that a multivector $\alpha\neq 0$ is invertible if and only if $\varepsilon(\alpha)\neq 0$; thus $\ker\varepsilon$ is maximal, but this is exactly $\Lambda^{\geq 1}V$.
\end{proof}

\begin{cor}\label{cor:filtration-preservation}
Let $\psi:\Lambda V\to\Lambda W$ be a unital morphism of supercommutative algebras. Then the filtration \eqref{eq:filtration} is preserved under $\psi$, that is $\psi\left(\Lambda^{\geq k}V\right)\subseteq\Lambda^{\geq k} W$.
\end{cor}

For an abstract exterior algebra (that is, a purely-odd and finitely generated supercommutative algebra) $\mathcal{A}$ we use the notation $\mathcal{A}^{\geq 1}$ for the (unique maximal) nilpotent ideal, and thus denote by $\mathcal{A}^{\geq k}$ its powers. The augmentation map will be denoted by $\varepsilon$ also.

\begin{nota}\label{remark:epsilon-unital}
The lemma above implies that $\Lambda V/\Lambda^{\geq 1} V$ is isomorphic to $\R$. Furthermore $\varepsilon|_{\Lambda^0 V}$ is an isomorphism and $\varepsilon$ is thus a unital homomorphism of superalgebras (that is, it takes $1$ to $1$). It is therefore the identity map on $\Lambda^0 V\cong\R$.
\end{nota}

\section{Derivations of exterior algebras}\label{sec:derivations}
\noindent In this section we'll be concerned exclusively with the exterior algebra of a finite-dimensional vector space $V^*$. The reason for taking ``duals'' is because in this way our results are conveniently expressed for the work to be done in chapter \ref{chap:superfolds}. Recall that $\Lambda V^*$ is a $\Z$-graded algebra where the grading is given by the degree of a form; that is, a form is said to be \textbf{homogeneous of degree} $k$ if it is a linear combination of monomials of the form $\alpha_1\wedge\cdots\wedge \alpha_k$, where $\setupla[k]{\alpha}$ is a linearly independent set in $V^*$. 

Let $A$ be an algebra. Recall that a linear map $D\in\End(A)$ is called a \textbf{derivation of $A$} if the identity
\[
D(ab)=D(a)b+aD(b)
\]
(the ``Leibnitz rule'') holds for every $a,b\in A$.

\begin{lema}\label{lemma:derivation-generators}
Let $D$ and $\widetilde{D}$ be derivations of the exterior algebra $\Lambda V^*$. If $D|_{V^*}=\widetilde{D}|_{V^*}$ then $D=\widetilde{D}$.
\end{lema}
\begin{proof}
Let us consider the map $F=D-\widetilde{D}$, which is also a derivation. If $\alpha_1,\ldots,\alpha_k$ are elements of $V^*$ then by a straightforward induction we get
\begin{equation}
\begin{split}
F(\alpha_1\wedge\cdots\wedge\alpha_k)&=\sum_{\mu=1}^k\alpha_1\wedge\cdots\wedge F(v_\mu)\wedge\cdots\wedge\alpha_k\\
									&=\sum_{\mu=1}^k\alpha_1\wedge\cdots\wedge \big(D(v_\mu)-\widetilde{D}(v_\mu)\big)\wedge\cdots\wedge\alpha_k
\end{split}
\end{equation}
which equals zero by the hypothesis on $D$ and $\widetilde{D}$. Since the exterior algebra is the linear span of decomposable forms the result follows.
\end{proof}
The above lemma is tantamount to saying that a derivation on an exterior algebra $\Lambda V^*$ is completely determined by its action on the generating vector space $V^*$.

Let us recall that, for $v\in V$, the operator $v\lrcorner$ is the map
\begin{eqnarray*}
\Lambda^k V^*&\to & \Lambda^{k-1}V^*\\
				\omega	&\mapsto & \omega(v,\cdot),
\end{eqnarray*}
i.e. $v\lrcorner\omega$ is the $(k-1)$-form obtained from $\omega$ by inserting $v$ as its first argument. A well known fact about this operator is that it is a \emph{derivation of degree $-1$}, which means
\begin{equation}\label{eq:inserting-antiderivation}
v\lrcorner(\alpha\wedge\beta)=(v\lrcorner\alpha)\wedge\beta+(-1)^{\parity{\alpha}}\alpha\wedge(v\lrcorner\beta)
\end{equation}
for all forms $\alpha$ and $\beta$. Identity \eqref{eq:inserting-antiderivation} is a consequence of the fact that $v\lrcorner$ is the dual operator of the exterior multiplication by a vector; if $\alpha$ is an exterior form then for all multivectors $X$
\[
<v\lrcorner\alpha,X>=<\alpha,v\wedge X>,
\] 
where $<\cdot,\cdot>$ denotes the evaluation pairing between $\Lambda V^*$ and $\Lambda V$. This operator allows us to construct many derivations of the exterior algebra.

\begin{lema}\label{lemma:derivations}
Let $F:V^*\to\Lambda_-{V^*}$ be a linear map and let $\setupla{v}$ and $\setupla{{dv}}$ be dual bases for $V$ and $V^*$ respectively. The map
\[
D_F:=\sum_{\mu=1}^n F(dv_\mu)\wedge\circ (v_\mu\lrcorner)
\]
is a derivation of $\Lambda V^*$. 
\end{lema}
\begin{proof}
Let $\alpha$ and $\beta$ be alternating forms, and let $\alpha$ and $F(dv_\mu)$ be of degree $a$ and $d_\mu$ respectively; observe that $d_\mu$ is always odd for all $\mu$. Using \eqref{eq:inserting-antiderivation} we compute
\[
\begin{split}
D_F(\alpha\wedge\beta)	&=\sum_{\mu=1}^m F(dv_\mu)\wedge\circ v_\mu\lrcorner(\alpha\wedge\beta)\\
						&=\sum_{\mu=1}^m F(dv_\mu)\wedge\big( (v_\mu\lrcorner\alpha)\wedge\beta) +(-1)^{a}\alpha\wedge (v_\mu\lrcorner\beta)\big)\\
						&=\sum_{\mu=1}^m F(dv_\mu)\wedge(v_\mu\lrcorner\alpha)\wedge\beta)+(-1)^{a}F(dv_\mu)\wedge\alpha\wedge (v_\mu\lrcorner\beta)\\
						&=\sum_{\mu=1}^m F(dv_\mu)\wedge(v_\mu\lrcorner\alpha)\wedge\beta)+\sum_{\mu=1}^m (-1)^{a}F(dv_\mu)\wedge\alpha\wedge(v\lrcorner\beta)\\
						&=\sum_{\mu=1}^m F(dv_\mu)\wedge(v\lrcorner\alpha)\wedge\beta)+\sum_{\mu=1}^m (-1)^{a}(-1)^ {ad_\mu}\alpha\wedge F(dv_\mu)\wedge(v_\mu\lrcorner\beta)\\
						&=\sum_{\mu=1}^m F(dv_\mu)\wedge(v\lrcorner\alpha)\wedge\beta)+\sum_{\mu=1}^m (-1)^ {a(d_\mu+1)}\alpha\wedge F(dv_\mu)\wedge(v_\mu\lrcorner\beta)\\
						&=D_F(\alpha)\wedge\beta+\alpha\wedge D_F(\beta)
\end{split}				
\]
the sign in the last equality cancelling because $d_\mu+1$ is always even. Thus $D_F$ is a derivation.
\end{proof}

\begin{cor}
If $F=\id_{V^*}$ then $D_F$ is the operator $k\id$ on the subspace $\Lambda^k V^*$.
\end{cor}

\begin{defn}\label{def:op-numbers}
The operator $N:=D_{\id}$ is called the \textbf{operator of numbers.}
\end{defn}

We can now characterize the space of derivations of $\Lambda V^*$.

\begin{teo}\label{thm:derivations}
Let $V$ be a vector space of dimension $n$. We have a natural isomorphism
\begin{equation}\label{eq:space-derivations}
\der \left( \Lambda V^* \right) \cong V\otimes\Lambda_- V^*\oplus\Lambda_- V^*\big{/}\left( \Lambda_- V^*\cap\Lambda^n V^* \right)
\end{equation}
\end{teo}
\begin{proof}
Let $D$ be a derivation of $\Lambda V^*$ and set $F_D:=D|_{V^*}$. As a linear map, $F_D$ decomposes in $F^+\oplus F^-$, where $F^\pm:V^*\to\Lambda^\pm V^*$. Lemma \ref{lemma:derivations} accounts for the map $F^-$ and the factor $V\otimes\Lambda^- V^*$ in \eqref{eq:space-derivations}. 

As for $F^+$, let $\alpha$ and $\beta$ be in $V^*$; we first observe the identity $0=F^+(\alpha\wedge\alpha)=2\alpha\wedge F^+(\alpha)$, which implies, by polarization, the following identity:
\[
\begin{split}
0	&=F^+(\alpha+\beta)\wedge(\alpha+\beta)+F^+(\alpha-\beta)\wedge(\alpha-\beta)\\
	&=(F^+(\alpha)+F(\beta))\wedge(\alpha+\beta)+(F^+(\alpha)-F^+(\beta))\wedge(\alpha\wedge\beta)\\
	&=F^+(\alpha)\wedge\beta +F^+(\beta)\wedge\alpha -F^+(\alpha)\wedge\beta-\alpha\wedge F^+(\beta)
\end{split}
\]
which yields
\begin{equation}\label{eq:polariz}
F^+(\alpha)\wedge\beta=\alpha\wedge F^+(\beta)
\end{equation}
Using the operators of numbers $N$ (definition \ref{def:op-numbers}) we compute:
\begin{eqnarray*}
(m-N)F^+(\alpha)&=&\sum_{\mu =1}^{n}v_\mu\ins \bigl(dv_\mu \wedge F^+(\alpha)\bigr)\\
				&=&\sum_{\mu =1}^{n}v_\mu\ins \bigl(F^+(dv_\mu) \wedge \alpha\bigr)\qquad\qquad\text{(identity \eqref{eq:polariz})}\\
				&=&\sum_{\mu =1}^n v_\mu\ins\bigl(-\alpha\wedge F^+(dv_\mu)\bigr)\\
				&=&\sum_{\mu =1}^n -\alpha(v_\mu)\wedge F^+(dv_\mu)+\alpha\wedge (v_\mu\ins F^+(dv_\mu))\\
				&=&-F^+(\alpha)+\alpha\wedge\sum_{\mu =1}^n  (v_\mu\ins F^+(dv_\mu)).
\end{eqnarray*}
Defining $\eta:=-\sum_{\mu =1}^n  (v_\mu\ins F^+(dv_\mu))$, we see that $\eta$ is an odd form. The final equation is
\begin{equation}
F^+(\alpha)=(n-N+1)^{-1}\eta\wedge\alpha .
\end{equation}
That is, $F^+$ is an operator equivalent with the wedge product with an odd form. This means that $D_{F^+}$ is the extension as a derivation of the operator $\eta\wedge$, which is the operator $[\eta,\cdot]$ (the Lie bracket being given by the commutator, because we are working in an associative algebra). By definition, $\eta$  cannot be of maximal degree, because $v\ins \phi$ is always of degree strictly less than $n=\dim V$ for all $v\in V$ and all $\phi\in A$. This accounts for the second summand $\Lambda^- V^*\big{/}\left( \Lambda^- V^*\cap\Lambda^n V^* \right)$ in \eqref{eq:space-derivations}. Conversely, given an odd form $\beta$, the operator $\frac{1}{2}[\beta,\cdot]$ is a derivation that agrees on generators with $\beta\wedge$ and is $0$ on even forms.
\end{proof}

\begin{cor}
The space $\der_{\z}(\Lambda V^*)$ of derivations that preserve the $\z$-grading satisfies
\begin{equation}\label{eq:z2-graded-derivations}
\der_{\z}\cong V\otimes\Lambda^- V^*
\end{equation}
\end{cor}
\begin{proof}
We know that the product with an odd form changes parity, so from \eqref{eq:space-derivations} we see that the factor $\Lambda^- V^*\big{/}\left( \Lambda^- V^*\cap\Lambda^n V^* \right)$ does not appear in this case.
\end{proof}

Also, a straightforward computation with basic forms of any degree (i.e. monomials of the form $\omega_{\mu_1}\wedge\cdots\wedge\omega_{\mu_p}$ for a basis $\setupla{\omega}$ of $V^*$ and $p\leq n$ a non-negative integer) leads to the following
\begin{cor}\label{cor:z-graded-derivations}
The space $\der_{\Z} \Lambda V^*$ of derivations preserving the $\Z$-graduation is isomorphic to $V\otimes V^*$.
\end{cor}

\subsection{Superderivations}
\noindent Recall that if $W=(W_0|W_1)$ is a supervector space then the space of linear operators is also a supervector space with $\z$-grading given by
\[
\begin{split}
\End_+(W)&=\set{T|T(W_\mu)\subseteq W_\mu}\\
\End_-(W)&=\set{T|T(W_\mu)\subseteq W_{\mu+1}}
\end{split}
\]
and thus
\[
\End(W_0|W_1)=\End_+(W)\oplus\End_-(W)
\]
is the (super)space of superendomorphisms of $W$. Likewise, it is important to distinguish between even and odd superderivations on a superalgebra. 

\begin{defn}
Let $\mathcal{A}$ be a superalgebra. A homogeneous endomorphism $D$ of $\mathcal{A}$ is a \textbf{superderivation} if the following identity holds
\begin{equation}\label{eq:super-leibnitz}
D(ab)=D(a)b+(-1)^{\parity{a}\parity{D}}aD(b)
\end{equation}
for all homogenous elements $a$ and $b$ of $\mathcal{A}$.
\end{defn}

\begin{nota}\label{remark:notation-for-derivations}
If $\mathcal{A}=A_0\oplus A_1$ is a superalgebra then the superspace of superderivations will be denoted $\sder(\mathcal{A})$ or by $\sder(A_0|A_1)$ for emphazising the $\z$-grading. Equation \eqref{eq:super-leibnitz} is the graded Leibnitz identity. The space of superderivations is graded by
\[
\begin{split}
\sder_+(\mathcal{A})&=\set{D|D(\mathcal{A}_\bullet)\subseteq\mathcal{A}_\bullet}\subset\End_+(\mathcal{A})\\
\sder_-(\mathcal{A})&=\set{D|D(\mathcal{A}_\bullet)\subseteq\mathcal{A}_{-\bullet}}\subset\End_-(\mathcal{A})
\end{split}
\]
Elements of these sets are called \textbf{even} and \textbf{odd} superderivations, respectively. 
\end{nota}

\begin{teo}\label{thm:superderivations}
The superspace of all superderivations $\sder\left( \Lambda V^*\right)$ is isomorphic to $V\otimes\Lambda V^*$, the $\z$-grading being given by
\begin{equation}\label{eq:superderivations}
\sder\left( \Lambda V^*\right)_\bullet=V\otimes\Lambda_{-\bullet} V^*
\end{equation}
where $-\bullet$ means a change of parity.
\end{teo}
\begin{proof}
Let us first observe that lemma \ref{lemma:derivation-generators} also applies: a superderivation is completely determined by its action on generators; the proof is essentially the same, mutatis mutandis.

The set of even superderivations is exactly the set $\der_{\z}(\Lambda V^*)$ of derivations that preserve the $\z$-grading, so \eqref{eq:superderivations} is just the isomorphism \eqref{eq:z2-graded-derivations}.

As in lemma \ref{lemma:derivations} we construct an odd superderivation $D_F$ from a map $F\colon V^*\to\Lambda_+ V^*$. If $D$ is an odd superderivation then its restriction $D|_{V^*}$ is a map $F_D\colon V^*\to\Lambda_+ V^*$. We need to prove that $D\mapsto F_D$ and $F\mapsto D_F$ are inverse to each other. So let $D$ be an odd derivation, $F=D|_{V^*}$ its restriction to the space of generators and denote by $\widetilde{D}$ the derivation arising from $F$, i.e. $\widetilde{D}=D_F$. Our first remark implies $\widetilde{D}=D$ and a computation similar to the one made in the proof of lemma \ref{lemma:derivations} implies $F_{\widetilde{D}}=F$, so the result follows.
\end{proof}

We now prove an extension of corollary \ref{cor:free-superalgebra} for the space of superderivations.

Let us first digress on the algebra of polynomials $\R[x_1,\ldots,x_n]$ and its derivations. With the fixed bases $\setupla{e}$ of $\R^n$ and $\setupla{x}$ of the polynomial algebra we have an associated isomorphism with the symmetric algebra $\Sym(\R^n)^*$. When considering the derivations of this algebra with respect to the standard basis, we can make the identification
\begin{equation}
\diffp{\empty}{{x_\mu}}\longleftrightarrow e_\mu\ins,
\end{equation}
where $\setupla{e}$ is the standard basis of $\R^n$. The derivation $\diffp{\empty}{{x_\mu}}$ is by definition the \emph{directional derivative}
\begin{equation}\label{eq:directional-derivative}
\left.\diff{\empty}{t}\right|_{t=0}\left(p(x+te_\mu)\right)
\end{equation}
on a polynomial $p$; it is this equation we want to carry over to the superalgebra setting. This is done in order to give a precise definition of the ``odd derivatives'' of a smooth superfunction in the next chapter.

\begin{prop}\label{prop:odd-derivations-ins}
Given an isomorphism $\phi:\Lambda V^*\to\mathcal{A}$ the derivation $\diffp{\empty}{{\xi_\mu}}$ corresponds under $\inv{\phi}$ to the operator $v_\mu\ins$.
\end{prop}
\begin{proof}
The linear extension of corollary \ref{cor:free-superalgebra} allows us to make the same identification of derivations and inner product operators as in equation \eqref{eq:directional-derivative}.
\end{proof}

\begin{cor}\label{cor:odd-derivations-ins}
The space of odd superderivations is generated by $V$ as a left $\Lambda V^*$-module.
\end{cor}

Turning to an abstract free supercommutative superalgebra $\mathcal{A}$ we denote by $\mathbf{S}$ the space of generators of $\mathcal{A}$ and by $\pr:\sder_-\mathcal{A}\to\mathbf{S}$ the projection to its space of generators.

The proof of our result depends on the Cartan-Poincaré Lemma, which we state and prove on appendix \ref{app:cartan-poincare}. 

\begin{lema}\label{lemma:derivations-isomorphism}
Let $\mathcal{A}$ be a free supercommutative finite-dimensional superalgebra and denote by $\mathbf{S}^*$ its space of generators. Let $D:V\to\sder_-\mathcal{A}$ be a linear map such that the composition
\[
\xymatrix@1{
f:V\ar[r]^{D}&\,\sder_{-}\mathcal{A}\ar[r]^-{\pr}&\mathbf{S}&
}
\]
is injective and such that if $v,\widetilde{v}$ are in $V$ then $[D_v,D_{\widetilde v}]=0$. Then there exists an isomorphism $G:\Lambda\mathbf{S}^*\to\mathcal{A}$ of $\z$-graded algebras with unit such that $G$ induces the identity 
\[
\overline{G}:\mathbf{S}^*\to\mathcal{A}^{\geq 1}\big{/}\mathcal{A}^{\geq 2}=:\mathbf{S}^*
\]
and such that, for all $v\in V$ and all $\sigma\in\Lambda\mathbf{S}^*$
\[
D_v(G\sigma)=D(f(v)\ins \sigma).
\]
Furthermore, up to the ideal generated by $\Lambda^3\ker(f^*)$ in $\Lambda^3\mathbf{S}^*$ the isomorphism $G$ is unique in the sense that if $G'$ is any other such isomorphism then 
\[
\inv{G}\circ G':\Lambda\mathbf{S}^*\to\Lambda\mathbf{S}^*:\sigma\mapsto\sigma+\langle \Lambda^3\ker(f^*) \rangle
\]
\end{lema}

We defer the proof to appendix \ref{app:composition}.

\section{Lie superalgebras}\label{sec:lie-superalgebras}
\noindent Let $(V_0|V_1)$ be a supervector space. Notice that the supercommutator \eqref{eq:supercommutator} is restricted by associativity to satisfy:
\begin{equation}\label{eq:jacobi-supercommutator}
(-1)^{\parity{A}\parity{C}}\lsem A,\lsem B, C\rsem\rsem +(-1)^{\parity{B}\parity{C}}\lsem C,\lsem A, B\rsem\rsem +(-1)^{\parity{A}\parity{B}}\lsem B,\lsem C, A\rsem\rsem =0
\end{equation}
called, in analogy with the classical case, the \textbf{super-Jacobi identity}. Note that the equation above is equivalent to
\begin{equation}\label{eq:jacobi-superderivation}
\lsem A,\lsem B,C\rsem\rsem=\lsem\lsem A,B\rsem\,C\rsem +(-1)^{\parity{A}\parity{B}}\lsem B,\lsem A,C\rsem\rsem
\end{equation}
which is, as expected, quite similar to equation \eqref{eq:super-leibnitz}. What this means is that the operator
\[
\begin{split}
\ad_A\colon\End(V_0|V_1)&\to\End(V_0|V_1)\\
				B		&\mapsto \lsem A,B\rsem
\end{split}
\]
is a superderivation of $\End(V_0|V_1)$ with respect to the superbracket operation. All of the above is the main motivation for the following

\begin{defn}\label{def:Lie}
A supervector space $\mathcal{L}=(\mathfrak{g}|\mathbf{S})$ (cf. remark \ref{remark:notation}) is a \textbf{Lie superalgebra} if it is endowed with a superbilinear map $\lsem\cdot,\cdot\rsem\colon\mathcal{L}\times\mathcal{L}\to\mathcal{L}$ such that
\begin{itemize}
\item $\lsem X,Y\rsem=(-1)^{\parity{X}\parity{Y}}\lsem Y,X\rsem$ (superalternating);
\item $\lsem X,\lsem Y,Z\rsem\rsem =\lsem\lsem X,Y\rsem ,Z\rsem +(-1)^{\parity{X}\parity{Y}}\lsem Y,\lsem X,Z\rsem\rsem$ (super-Jacobi identity).
\end{itemize}
for all $X,Y,Z$ homogeneous elements of $\mathcal{L}$.
\end{defn}

A very simple and straightforward fact  about Lie superalgebras is the following:

\begin{lema}\label{lemma:lie-superalgebras}
Let $\mathcal{L}=(\mathfrak{g}|\mathbf{S})$  be a Lie superalgebra. Then 
\begin{enumerate}
\item $\mathfrak{g}$ is a Lie algebra when restricting the superbracket to $\mathfrak{g}\times\mathfrak{g}$.
\item\label{repre} $\mathbf{S}$ is a representation of $\mathfrak{g}$ with the action given by the superbracket; that is $\lsem\mathfrak{g},\mathbf{S}\rsem\subseteq\mathbf{S}$.
\item\label{B} The superbracket restricts to a symmetric bilinear map $\lsem\cdot,\cdot\rsem\colon\Sym^2\mathbf{S}\to\mathfrak{g}$.
\end{enumerate}
\end{lema}
\begin{proof}
Since all elements of $\mathfrak{g}$ are even by definition, the usual bracket operations satisfy definition \ref{def:Lie}. Items \ref{repre} and \ref{B} above are just the $\z$-grading of the multiplicative structure $(\mathcal{L},\lsem\cdot,\cdot\rsem)$: $\rho(X)=\lsem X,\cdot\rsem$ is a representation twisted with the structure automorphism $\gamma$ (cf. proposition \ref{prop:superstructure-morphism}); the operation in item \ref{B} is symmetric because all elements of $\mathbf{S}$ are odd by definition.
\end{proof}

Let us denote by $\rho\colon\mathfrak{g}\to\End(\mathbf{S})$ the representation of part \ref{repre} and by $B\colon\Sym^2\mathbf{S}\to\mathfrak{g}$ the symmetric form of part \ref{B} of the above lemma. We can then write $\lsem\cdot,\cdot\rsem=[\cdot,\cdot]\oplus\rho\oplus B$.

\begin{ejemplo}
An interesting and simple example of a Lie superalgebra, which we call a \textbf{semi--direct product}, arises when a representation $\rho$ of a Lie algebra is given on a vector space $\mathbf{S}$ and $B\equiv 0$; the Lie superbracket is then defined by
\begin{equation}\label{eq:semidirect}
\lsem X\oplus s,Y\oplus t \rsem=[X,Y]\oplus (\rho(X)t-\rho(Y)s)
\end{equation}
We'll denote this Lie superalgebra as $\mathfrak{g}\oright_{\rho}\mathbf{S}$ or by $\mathfrak{g}\oright\mathbf{S}$ if the representation is clear.
\end{ejemplo}

We now give an algebraic classification of Lie superalgebras.

\begin{teo}\label{thm:lie-superalgebras}
A Lie superalgebra $\mathcal{L}=(\mathfrak{g}|\mathbf{S})$ is completely determined by the following data:
\begin{itemize}
\item A representation $\rho$ of the Lie algebra $\mathfrak{g}$ on the vector space $\mathbf{S}$;
\item A symmetric $\mathfrak{g}$-equivariant bilinear map $B\colon\Sym^2\mathbf{S}\to\mathfrak{g}$ that lies in the kernel of the composition
\[
\xymatrix{
\left( \Sym^2\mathbf{S}^*\otimes\mathfrak{g} \right)^{\mathfrak{g}}\ar[r]^-{\id\otimes\rho}&\left( \Sym^2\mathbf{S}\otimes\mathbf{S}^*\otimes\mathbf{S} \right)^{\mathfrak{g}}\ar[r]^-{f\otimes\id}&\left(\Sym^3\mathbf{S}^*\otimes\mathbf{S}\right)^{\mathfrak{g}}
}
\]
where $f(\alpha\otimes \sigma)=\alpha\cdot\sigma$.
\end{itemize}
\end{teo}
\begin{proof}
That a Lie superalgebra structure on the direct sum $\mathfrak{g}\oplus\mathbf{S}$ gives the pair $(\rho,B)$ is the content of lemma \ref{lemma:lie-superalgebras}. Conversely, given $(\rho, B)$ the construction of a Lie superalgebra structure on the space $\mathfrak{g}\oplus\mathbf{S}$ is guaranteed by the fact that $B$ is $\mathfrak{g}$-equivariant and with the definitions
\begin{eqnarray*}
\lsem X,Y\rsem	&=& [X,Y]\\
\lsem X,s\rsem	&=& \rho(X)(s)\\
\lsem s,X\rsem	&=& \rho(X)\gamma(s)\\
\lsem s,t\rsem	&=& B(s,t)
\end{eqnarray*}
where $[\cdot,\cdot]$ is the Lie bracket of $\mathfrak{g}$, $\gamma$ is the structure morphism (cf. proposition \ref{prop:superstructure-morphism}) $X,Y\in\mathfrak{g}$ and $s,t\in\mathbf{S}$. The super-Jacobi identity is guaranteed by the $\mathfrak{g}$-equivariance of $B$.
\end{proof}

\section[Twisted action of $S_k$]{Twisted action of the symmetric group on the tensor superalgebra}\label{sec:twisted}
\noindent In classical linear algebra and differential geometry one encounters the exterior and symmetric algebras of a vector space $V$, through which multilinear alternating and symmetric maps whose domain is $V$ are codified. These algebras are quotients by adequate ideals of a bigger algebra, the tensor algebra of $V$, which codifies general multilinear mappings. It turns out that the defining ideals for the symmetric and exterior algebras are invariant subspaces of a quite natural action of the symmetric group $S_k$ on the space $\bigotimes^k V$ of homogeneous tensors of rank $k$. This action is just a permutation on monomials times a sign. To wit, if $v_1\otimes\cdots\otimes v_k$ is a decomposable tensor (i.e. each $v_\mu$ is in $V$) and $\sigma$ is a permutation then
\[
\sigma\cdot v_1\otimes\cdots\otimes v_k=\epsilon_\sigma\cdot v_{{\sigma}(1)}\otimes\cdots\otimes v_{{\sigma}(k)}
\]
where $\epsilon_\sigma$ is either identically one or the signature character of $S_k$. We denote these actions by $\rho$ and $r$ respectively. That is:
\begin{subequations}\label{eq:tensor-actions}
\begin{align}
\rho(\sigma)v_1\otimes\cdots\otimes v_k	&=v_{{\sigma}(1)}\otimes\cdots\otimes v_{{\sigma}(k)}\label{eq:tensor-action-sym}\\
r(\sigma)v_1\otimes\cdots\otimes v_k		&=\sgn\sigma\cdot v_{{\sigma}(1)}\otimes\cdots\otimes v_{{\sigma}(k)}\label{eq:tensor-action-alt}
\end{align}
\end{subequations}
In this way the space of symmetric and alternating $k$-tensors are defined as
\begin{equation}\label{eq:sym-alt-quotients}
\Sym^k V:= \bigotimes\nolimits^k V\big{/}r(S_k)\quad\text{and}\quad\Lambda V:=\bigotimes\nolimits^k V\big{/}\rho(S_k).
\end{equation}
In the first quotient above all alternating sums are in the zero class whereas in the second all symmetric sums are in the zero class. Thus in the first quotient acting by a transposition leaves the class unchanged while in the second one such a linear map changes the sign of the class, as was to be expected.

In order to extend such a construction to supervector spaces we need to take into account the parity of homogeneous elements and how it behaves under permutations. For this we need the following

\begin{defn}\label{def:twisted-action}
Let $A$ be a subset of $\set{1,2,\ldots,k}$, $\sigma\in S_k$ and let $\tau$ be the permutation that orders $\setupla[r]{b}=\sigma(A)$ increasingly. The \textbf{relative signature of $\sigma$ with respect to $A$} is 
\[
\sgn^A(\sigma)=\sgn
\begin{pmatrix}
\tau(b_1)&b_2&\cdots&\tau(b_r)\\
\sigma(a_1)&\sigma(a_2)&\cdots&\sigma(a_r)
\end{pmatrix};
\]
that is, it's the signature of the permutation of $\sigma(A)$ that orders its elements increasingly.
\end{defn}

To see that the relative signature is well-defined we digress for a moment on the concept of a shuffle permutation. If $\set{1,\ldots,k}$ is written as a disjoint union of $B$ and $C$ of cardinalities $r$ and $s$ respectively then the elements of the coset set $S_k/(S_B\times S_C)$ are a system of representatives for the group 
\[
S_{r,s}:=\set{\tau\in S_k|\Big(b>b'\Rightarrow \tau(b)>\tau(b')\Big)\text{ and }\Big( c>c'\Rightarrow \tau(c)>\tau(c')\Big)}
\]
of $(r,s)-$\textbf{shuffles}; of course $b,b'$ and $c,c'$ denote elements of $B$ and $C$ respectively, and the order on $B$ and $C$ is the induced one from $\set{1,\ldots,k}$. Using the following lemma (whose proof we omit) it can be shown that this group does not depend on the choice of an order on neither $B$, $C$ or even $B\sqcup C\cong\set{1,\ldots,k}$:

\begin{lema}
Let $\sigma$ be a permutation of the set $A\cong\set{1,\ldots,k}$, and let $B$ and $C$ be subsets such that $A=B\sqcup C$; if $A$ is well ordered such that $B$ and $C$ inherit a well order then there exists a unique permutation $\tau\in S_B\times S_C$ such that $\sigma\inv{\tau}$ is monotone when restricted to $B$ and $C$.
\end{lema}

If we want analogues of identities \eqref{eq:sym-alt-quotients} to hold for supervector spaces we need to take into account the set $A$ of indices for which the corresponding elements of $V$ are odd given that whenever any odd elements are commuted a sign must appear. We begin with the tensor superalgebra:

\begin{defn}
Let $V=(V_0|V_1)$ be a supervector space. The \textbf{tensor superalgebra of $V$} is an associative unital superalgebra $\bigotimes V$ with a universal property: let $\mathcal{A}$ be an associative unital superalgebra, $f\colon V\to\mathcal{A}$ a superlinear map and $\iota\colon V\to\bigotimes V$ the inclusion morphism; then there exists a unique unital superalgebra morphism $F:\bigotimes V\to\mathcal{A}$ such that the diagram
\[
\xymatrix{%
\bigotimes V\ar[r]^{F}	&\mathcal{A}\\
V\ar[u]^{\iota}\ar[ur]_{f}
}
\]
commutes; that is $f=F\circ\iota$.
\end{defn}

As with the classical case, this algebra may be presented using a basis for $V$. To wit: let $\setupla{v}$ be such a basis then
\[
\bigotimes V=\left<v_{\mu_1}\cdots v_{\mu_s}\right>.
\]
That is, $\bigotimes V$ is the free associative unital superalgebra generated by $V$. The $\Z$-grading of $\bigotimes V$ is given by
\begin{equation}\label{eq:supertensor-decomposition}
\bigotimes\nolimits^k(V_0|V_1)= \bigoplus_{\mu_j\in\set{0,1}}V_{\mu_1}\otimes\dotsm\otimes V_{\mu_k}
\end{equation}
We use the standard notation $v_1\otimes\cdots\otimes v_k$ for an element of $\bigotimes^k(V_0|V_1)$. The structure automorphism $\Gamma$ of $\bigotimes V$ (cf. proposition \ref{prop:superstructure-morphism}) is given by the extenstion of the structure automorphism $\gamma$ of $V$. This means
\[
\Gamma(v_1\otimes\cdots\otimes v_k)=\gamma(v_1)\otimes\cdots\otimes\gamma(v_k)
\]
This allows us to write $\bigotimes\gamma$ for $\Gamma$. Thus the $\z$-grading of the supervector space $\bigotimes V$ is given as follows: in decomposition \eqref{eq:supertensor-decomposition} the summands containing an even number of $V_1$ factors are even and the other ones are odd. So let $T=v_1\otimes\cdots\otimes v_k$ be an arbitrary decomposable supertensor and set $A_T=\set{\mu\mid v_\mu\in V_1}$. We define
\begin{equation}\label{eq:odd-signature}
\sgn^-\sigma(T)=\sgn^{A_T}\sigma(T)
\end{equation}
for all permutations $\sigma$. 

In analogy with identities \eqref{eq:sym-alt-quotients} we expect $\Sym(V_0|V_1)$ and $\Lambda(V_0|V_1)$ to be supercommutative and superalternating respectively. We therefore define the following actions of $S_k$ on $\bigotimes^k(V_0|V_1)$:
\begin{subequations}\label{eq:supertensor-actions}
\begin{align}
\rho(\sigma)(v_1\otimes\cdots\otimes v_k)&={\sgn^-\sigma}\cdot v_{\sigma(1)}\otimes\cdots\otimes v_{\sigma(k)}\label{eq:supertensor-actions-sym}\\
\intertext{and}
r(\sigma)(v_1\otimes\cdots\otimes v_k)	&={\sgn^-\sigma}\sgn\sigma\cdot v_{\sigma(1)}\otimes\cdots\otimes v_{\sigma(k)}.\label{eq:supertensor-actions-alt}
\end{align}
\end{subequations}
Note that if all the $v$'s are even we get the actions defined in \eqref{eq:tensor-actions}. If we formally define the symmetric and exterior superalgebras as in \eqref{eq:sym-alt-quotients} we get the following

\begin{teo}\label{thm:supertensor-algebras}
Let $V=(V_0|V_1)$ be a supervector space. The (super)symmetric and (super)exterior superalgebras satisfy the following isomorphisms:
\[
\Sym(V_0|V_1)\cong\Sym V_0\otimes\Lambda V_1\quad\text{and}\quad\Lambda(V_0|V_1)\cong\Lambda V_0\otimes\Sym V_1
\]
\end{teo}
\begin{proof}
The result follows from the fact that the supersymmetric algebra is the quotient of $\bigotimes(V_0|V_1)$ by the action \eqref{eq:supertensor-actions-sym} and the superexterior algebra is obtained anihilating the action \eqref{eq:supertensor-actions-alt}. 
\end{proof}

\subsection{The exterior superalgebra}\label{subsec:exterior-superalgebra}
\noindent Let $(V|S)$ be a supervector space. We now study the superalgebra $\Lambda(V|S)^*$ of superalternating forms on $(V|S)$. Theorem \ref{thm:supertensor-algebras} above allows us to write $\omega\otimes p$ for a typical element of this superalgebra, where $\omega\in\Lambda V$ and $p\in\Sym S$. The wedge superproduct of two such elements $\omega\otimes p$ and $\tilde{\omega}\otimes\tilde{p}$ is defined naturally as $\omega\wedge\tilde{\omega}\otimes p\tilde{p}$. However we expect formulae analogous to \eqref{eq:wedge-product-permutation} and \eqref{eq:wedge-supercommutative} to prevail in this setting. For such equations to hold we need to take into account the parity of $\omega\wedge\tilde{\omega}$ only, because of the isomorphism given by theorem \ref{thm:supertensor-algebras}: in the exterior superalgebra all odd elements commute. We thus have the equality
\[
\parity{\omega\otimes p\wedge\tilde{\omega}\otimes\tilde{p}}=\parity{\omega\wedge\tilde{\omega}}
\]
so given bases $\setupla[m]{\xi}$ and $\setupla{\sigma}$ of $V$ and $S$ respectively the wedge superproduct of decomposable  elements is given by
\begin{equation}\label{eq:super-wedge}
\begin{split}
(\xi_{\mu_1}\wedge\dotsm\wedge \xi_{\mu_k}\otimes \sigma_{j_1}\dotsm \sigma_{j_r})	
										&\wedge(\xi_{\mu_1}\wedge\dotsm\wedge \xi_{\mu_{\tilde k}}\otimes \sigma_{j_1}\dotsm \sigma_{j_{\tilde r}})=\\
(\xi_{\mu_1}\wedge\dotsm\wedge \xi_{\nu_k}\wedge \xi_{\nu_1}\wedge\dotsm\wedge \xi_{\mu_{\tilde k}})
										&\otimes (\sigma_{j_1}\dotsm \sigma_{j_r}\cdot \sigma_{l_1}\dotsm \sigma_{l_{\tilde r}}).
\end{split}
\end{equation}
The action of the symmetric group on elements like the above is
\begin{equation}\label{eq:action-superalternating}
\tau(\xi_{\mu_1}\wedge\dotsm\wedge \xi_{\mu_k}\otimes \sigma_{j_1}\dotsm \sigma_{j_r})={\sgn\widehat{\tau}}\cdot \xi_{\mu_{\inv{\tau}(1)}}\wedge\dotsm\wedge \xi_{\mu_{\inv{\tau}(k)}}\otimes \sigma_{j_{\inv{\tau}(1)}}\dotsm \sigma_{j_{\inv{\tau}(1)}}
\end{equation}
where $\widehat{\tau}$ denotes the class of $\tau$ on the group $S_{k+r}/S_r$, since the effect of $\tau$ on the indices $\listupla[r]{j}$ carries no sign to the result.

Now let $x\in(V|S)$ be a homogeneous vector. The operator $x\ins$ should also be superderivations of $\Lambda(V|S)$ like in the purely even case. This means that equation \eqref{eq:inserting-antiderivation} should become
\begin{equation}\label{eq:ins-superderivation}
x\lrcorner(\alpha\wedge\beta)=(x\lrcorner\alpha)\wedge\beta+(-1)^{\parity{x}\parity{\alpha}}\alpha\wedge(x\lrcorner\beta).
\end{equation}
To see that this is the case we use equation \eqref{eq:super-wedge} and dual bases to compute $x\lrcorner$ in these bases. So if $x$ is purely even then it is a linear combination of $\setupla[m]{v}$ and if it is purely odd it is in the span of $\setupla{s}$, the dual bases corresponding to $\setupla[m]{\xi}$ and $\setupla{\sigma}$ respectively. If $x$ is purely even equation above is just equation \eqref{eq:inserting-antiderivation}; on the other hand, if $x$ is purely odd then the operator $x\ins$ is an odd superderivation due to theorem \ref{thm:superderivations} and equation \eqref{eq:ins-superderivation} follows. 

\chapter{Smooth supermanifolds}\label{chap:superfolds}
\noindent This chapter is devoted to developing a theory of smooth supermanifolds closer to classical differential geometry, in a sense to be specified throughout the chapter. All our manifolds are assumed to be connected, Hausdorff and smooth (of class $\C$).

\section{Supermanifolds as superalgebra bundles}
\noindent Let $(V|\mathbf{S})$ (cf. remark \ref{remark:notation}) be a supervector space of finite dimension. Whatever the term \emph{smooth superfunction} would denote, it's clear that at least the polynomials in $V$ should be considered as smooth. The algebra of super-polynomials $\Pol^\bullet(V|\mathbf{S})$ should include both even and odd polynomial mappings. Recalling that, under the choice of a basis, the algebra of polynomials $\Pol{V}$ can be identified with the symmetric algebra $\Sym V^*$, we then have the isomorphism
\begin{equation}\label{eq:superpol}
\Pol^\bullet (V|\mathbf{S})\cong \Sym^\bullet (V|\mathbf{S})^*\cong\Sym V^*\otimes \Lambda \mathbf{S}^*
\end{equation}
by theorem \ref{thm:supertensor-algebras} and this algebra should be a subalgebra of that of smooth superfunctions. The space $\C(V,\Lambda \mathbf{S}^*)$ naturally contains $\Pol^\bullet (V|\mathbf{S})$ and therefore is a good candidate for being the algebra of smooth superfunctions of the supervector space $(V|\mathbf{S})$. This is the reason why, in what follows, we work with $\mathbf{S}^*$ instead of $\mathbf{S}$. 

Because of equation \eqref{eq:superpol} we can define a vector supermanifold to be the pair $(V,V\times\Lambda \mathbf{S}^*)$ and take this as the local model for our definition of supermanifolds.

\begin{defn}\label{def:superfold}
A \textbf{smooth supermanifold of superdimension $(m|n)$} is a pair $(M|\mathcal{R}M)$ consisting of a smooth manifold $M$ of dimension $m$ and a superalgebra bundle $\mathcal{R}M$ over $M$ such that for each point $p$ in $M$ the algebra $\mathcal{R}_pM$ is a free supercommutative superalgebra of finite rank $n$. The sections $\sect{\mathcal{R}M}$ are the \textbf{superfunctions} on the supermanifold $\superfold{M}$. This is an infinite-dimensional supercommutative algebra. 
\end{defn}

A basic example is the supermanifold $(M|\Lambda T^* M)$ whose superfunctions are the exterior differential forms on $M$.

\subsection{Remarks on our definition}
A few remarks are in order. Firstly, our definition is not to be confused with a classical theorem due to Marjorie Batchelor (cf. \cite{batchelor}) which states that every supermanifold is isomorphic to the exterior bundle of a vector bundle over the underlying manifold $M$. As we saw in chapter 1, every free supercommutative superalgebra of finite rank $n$ is (non-naturally) isomorphic to the exterior algebra of some vector space of dimension $n$. The isomorphism depends on two choices: a set of odd generators $\setupla{\xi}$ of the superalgebra and a basis $\setupla{v}$ of $V$. In our definition, the fibre at each point is a free supercommutative superalgebra of rank $n$, and later on (corollary \ref{cor:batchelors}) we shall prove that the Batchelor isomorphism is equivalent to the choice of a splitting map.

Another fact to note is that virtually every construction involving vector bundles can be carried out in the smooth setting using partitions of unity; furthermore, since in the sheaf-theoretic approach to supermanifolds the sine qua non conditions are that the structure sheaf $\mathcal{O}_M$ (see definition \ref{def:ag-superfolds}) be coherent and of constant rank over the sheaf of smooth functions --which, incidentaly, is already the sheaf of sections of a (trivial) vector bundle-- the definition above is adequate for the smooth setting.

The third and final remark is that both approaches are equivalent. Theorems \ref{thm:sheaf=bundle} and \ref{thm:equivalent} prove that to every sheaf-theoretic supermanifold we can associate a smooth supermanifold in the sense of the definition above and that to every morphism in the sheaf-theoretic sense we can associate a morphism in our sense.

\subsection{Bundles associated to a smooth supermanifold}\label{subsect:moar-bundles}
\noindent Let $\superfold{M}$ be of superdimension $(m|n)$. For each point $p$ in $M$ the fibre $\mathcal{R}_p M$ is a free supercommutative algebra of rank $n$; as such it has a unique maximal ideal, denoted
\begin{equation}\label{eq:maximal-ideal-fibre}
\mathcal{R}^{\geq 1}_p M:=\set{\eta\in\mathcal{R}_p M | \varepsilon_p(\eta)=0}
\end{equation}
where $\varepsilon_p:\mathcal{R}_p M\to\R$ is the augmentation map of the algebra $\mathcal{R}_p M$. Also, this algebra is filtered
\begin{equation}
\mathcal{R}_p M=\bigcup_{k=0}^{n } \mathcal{R}^{\geq k}_p M
\end{equation}
as in \eqref{eq:filtration}.

\begin{defn}\label{def:odd-directions}
Let $\superfold{M}$ be a supermanifold of dimension $(m|n)$ and let $p\in M$. The vector space
\[
\mathbf{S}^*_p M:=\mathcal{R}^{\geq 1}_p M\big{/}\mathcal{R}^{\geq 2}_p M
\]
is the \textbf{space of odd codirections at the point $p$}. Its dual $\textbf{S}_p M$ is the space of \textbf{odd directions} at $p$.
\end{defn}

Now we proceed as differential geometers and define the bundles $\mathcal{R}^{\geq k}M$, with $k\geq 0$ an integer, the \textbf{nilpotent bundle} of $\superfold{M}$ being the case $k=1$. The sections of each of these vector bundles are nilpotent superfunctions on $\superfold{M}$

\begin{prop}
Let $\superfold{M}$ be a supermanifold. The algebra $\sect{\mathcal{R}M}$ of smooth superfuncions is a filtered algebra:
\begin{equation}\label{eq:superfunctions-filtration}
\sect{\mathcal{R}^{\geq 0}M}\supset\sect{\mathcal{R}^{\geq 1}M}\supset\cdots\sect{\mathcal{R}^{\geq n-1 }M}\supset\sect{\mathcal{R}^{\geq n}M}\supset\set{0}
\end{equation}
\end{prop}
\begin{proof}
On a trivializing neighbourhood $U\subseteq M$ the bundle $\mathcal{R}U$ is isomorphic to $U\times\Lambda \mathbf{S}^*$, so 
\[
\C(U,\Lambda \mathbf{S}^*)=\bigcup_{k\geq 0}\C(U,\Lambda^{\geq k}\mathbf{S}^*)
\]
and this decomposition extends globally by a partition of unity argument.
\end{proof}

Next we define the vector bundles
\begin{equation}\label{eq:odd-directions}
\mathbf{S}M=\bigsqcup_{p\in M} \mathbf{S}_p M
\end{equation} 
and 
\begin{equation}\label{eq:odd-codirections}
\mathbf{S}^* M = \bigsqcup_{p\in M} \mathbf{S}^*_p M
\end{equation}
of \textbf{odd directions} and \textbf{odd codirections} respectively. These will play an important role when we define the tangent superbundle of a supermanifold. These are all smooth vector bundles over $M$. Notice that for each fibre $\mathcal{R}_p M$ the augmentation map $\varepsilon_p:\mathcal{R}_p M\to\R$ is a unital homomorphism of superalgebras, so we get another such morphism $\varepsilon_{M}:\sect{\mathcal{R}M}\to\C(M)$ which we also call an augmentation map.

\begin{nota}\label{remark:odd-and-even}
An important distinction arises when considering the $\z$-grading of the algebra $\mathcal{R}_pM$: the even elements, which we denote by $\mathcal{R}_{+,p} M$, and the odd elements, denoted by $\mathcal{R}_{-,p} M$. Accordingly we get the bundles $\mathcal{R}_+ M$ and $\mathcal{R}_- M$ and their sections are called, respectively, the \textbf{even} and \textbf{odd} superfuncions of $\superfold{M}$. 
\end{nota}

\subsection{Morphisms}
\noindent Let $(M,\mathcal{O}_M)$ and $(N,\mathcal{O}_N)$ be ringed spaces. Recall that a morphism between them is defined to be a pair of maps $(\phi,\phi^\#)$, such that $\phi:M\to N$ is continuous and for each open set $U\subseteq N$ the \emph{localized map} $\phi^\#:\mathcal{O}_U\to\mathcal{O}_{\inv{\phi}(U)}$ is a morphism of rings (cf. \cite[section 2.3(b)]{ueno}). In the case of supermanifolds, the morphism $\phi^\#$ is required to be a unital morphism of supercommutative algebras. This forces $\phi^\#$ to be even.

To make a differential--geometric sense out of this definition, let us recall a well-known fact about the algebra $\C(M)$ of smooth functions of a smooth manifold $M$. For a proof we refer the reader to \cite[chapter 4]{jet}.

\begin{lema}\label{lema:smooth-algebras}
Let $\phi:M\to N$ be a smooth map. Then the map
\[
\begin{split}
\phi^*\colon\C(N)	&\to\C(M)\\
				f	&\mapsto f\circ\phi
\end{split}
\]
is a unital homomorphism of associative algebras.
\end{lema}

Our definition of a supersmooth map takes the above morphism into account. We first consider a special case: let $\mathcal{R}M=M\times\Lambda V^*$ be a trivial bundle and consider the corresponding supermanifold $(M|M\times\Lambda V^*)$. The sheaf of sections is nothing but $\C(M,\Lambda V^*)$. A smooth superfunction on the supermanifold in question is expressed as
\begin{equation}
f=f_0+\text{nilpotent part}
\end{equation}
where $f_0$ is a smooth function on $M$ and so we get the inclusion $\iota:\C(M)\to\C(M,\Lambda V^*)$. If $(N|N\times\Lambda W^*)$ is another such supermanifold then, associated to a smooth map $\phi:M\to N$, we must obtain an even morphism of supercommutative algebras $\Phi:\C(N,\Lambda W^*)\to\C(M,\Lambda V^*)$ such that the diagram
\begin{equation}\label{eq:diagrama1}
\xymatrix{
\C(N,\Lambda W^*)\ar[rr]_{\varepsilon_N} \ar[d]^{\Phi} && \C(N) \ar@/_1pc/[ll]_{\iota} \ar[d]^{\phi^*}\\
\C(M, \Lambda V^*)  \ar[rr]^{\varepsilon_M} && \C(M) \ar@/^1pc/[ll]^{\iota}
}
\end{equation}
commutes, where $\varepsilon_M$ and $\varepsilon_N$ are the augmentation maps, and $\iota_M$ and $\iota_N$ the corresponding inclusions. Now let $f$ and $g$ be smooth functions on $N$ and let $\eta$ be an element of $\C(N,\Lambda W^*)$; considering the twisted commutator
\begin{equation}\label{eq:commutator-supersmooth-functions}
[\Phi\empty_{\phi}f](\eta):=\Phi(f\eta)-(f\circ\phi)\Phi(\eta)
\end{equation}
we get
\begin{equation}\label{eq:commutator-supersmooth-functions-2}
[\Phi\empty_{\phi}f](\eta)=(\Phi(f)-f\circ\phi)\Phi(\eta)
\end{equation}
because $\Phi$ is an algebra morphism; moreover, since $\Phi$ is even it preserves the $\z$-grading and therefore the term $\Phi(f)-f\circ\phi$ is an even superfunction on $N$. Considering the diagram \eqref{eq:diagrama1} we see that $\varepsilon\circ\Phi=\phi^*$ and therefore the term $\Phi(f)-f\circ\phi=\Phi(f)-\phi^*(f)$ is nilpotent. Therefore, iterating this commutator we get zero after finitely many iterations. If $q$ is the dimension of $W$ then, putting $k=\lfloor\frac{q}{2}\rfloor$ (integer part of $\frac{q}{2}$), it is manifest that $[\cdots,[[\Phi\empty_{\phi}f_0]\empty_{\phi}f_1]\empty_{\phi}\cdots\empty_{\phi}f_k]\equiv 0$ for any $k+1$ smooth functions on $N$. 

Now, given two smooth supermanifolds $\superfold{M}$ and $\superfold{N}$ of superdimensions $(m|p)$ and $(n|q)$ respectively, the above constructions are valid on trivializing neighbourhoods of the corresponding bundles $\mathcal{R}M$ and $\mathcal{R}N$. Using a partition of unity we arrive at the same results: we have a commutative diagram
\begin{equation}\label{eq:diagrama2}
\xymatrix{
\sect{\mathcal{R}N}\ar[rr]_{\varepsilon_N} \ar[d]^{\Phi} && \C(N)  \ar[d]^{\phi^*}\\
\sect{\mathcal{R}M}\ar[rr]^{\varepsilon_M} && \C(M) 
}
\end{equation}

Note that the inclusions are not considered in this diagram. Summarizing the considerations above we have
\begin{prop}\label{prop:supermaps-are-diffops}
Let $\superfold{M}$ and $\superfold{N}$ be smooth supermanifolds and $(\phi,\Phi):\superfold{M}\to\superfold{N}$ a supersmooth map. Then the map $\Phi:\sect{\mathcal{R}N}\to\sect{\mathcal{R}M}$ is a linear differential operator along the smooth map $\phi:M\to N$ of order at most $k=\lfloor\frac{\rank \mathcal{R}N}{2}\rfloor$.
\end{prop}

We refer to appendix \ref{app:diff-ops} for the definition of a linear differential operator along a smooth map (definition \ref{def:dif-op-along}).

Another important property of supersmooth maps is the following:

\begin{prop}\label{prop:preserve-filtration}
If $(\phi|\Phi):\superfold{M}\to\superfold{N}$ is a supersmooth map then
\[
\Phi\left(\sect{\mathcal{R}^{\geq k}N}\right)\subseteq \sect{\mathcal{R}^{\geq k} M}
\]
for all non-negative integers $k$. That is, supermanifold morphisms preserve the filtration \eqref{eq:superfunctions-filtration}.
\end{prop}
\begin{proof}
Since $\Phi$ is an even superalgebra morphism it preserves the nilpotency of superfunctions, so $\Phi(\sect{\mathcal{R}^{\geq 1}N})\subseteq\sect{\mathcal{R}^{\geq 1}M}$; also, algebra morphisms preserve powers so the result follows.
\end{proof}

\begin{cor}\label{cor:preserve-filtration}
The morphism $\Phi:\sect{\mathcal{R}N}\to\sect{\mathcal{R}M}$ defines bundle morphisms 
$$\Phi^k:\mathcal{R}^{\geq k}N/\mathcal{R}^{\geq k+1}N\to\mathcal{R}^{\geq k}M/\mathcal{R}^{\geq k+1}M$$
by $\Phi^k(x+\mathcal{R}^{\geq k+1}_p M)=\Phi(x)\mod\Gamma(\mathcal{R}^{\geq k+1}M)$ for each $k\geq 0$.
\end{cor}
\begin{proof}
Let $f$ be a smooth function on $N$ and $\eta$ a section of ${\mathcal{R}^{\geq k}N}$. Because of the identity
\[
\twistcom{\Phi}{\phi}{f}(\eta)=(\Phi(f)-(f\circ\phi))\Phi(\eta)
\]
we know that $\Phi(f)-(f\circ\phi)$ is a section of $\mathcal{R}^{\geq 2}M$ and because of the inclusions of filtration \eqref{eq:superfunctions-filtration} we know that $(\Phi(f)-(f\circ\phi))\Phi(\eta)$ is a section of $\mathcal{R}^{\geq k+2}M$, which is contained in $\mathcal{R}^{\geq k+1}M$; therefore $\twistcom{\Phi}{\phi}{f}(\eta)\in\sect{\mathcal{R}^{\geq k+1}}$ and thus $\Phi^k$ is a bundle morphism over the smooth map $\phi$.
\end{proof}

It is worth noting that, although the morphism $\phi^*$ of lemma \ref{lema:smooth-algebras} is naturally attached to the smooth map $\phi$, the differential operator $\Phi$ is not uniquely determined by $\phi$. In fact there is a large set of differential operators along $\phi$ from where to choose. Nevertheless if $N$ is compact the map $\Phi$ determines a  unique supersmooth map:

\begin{prop}
Let $\superfold{M}$ and $\superfold{N}$ be supermanifolds with $N$ compact and let $\Phi:\sect{\mathcal{R}N}\to\sect{\mathcal{R}M}$ be a unital superalgebra morphism. Then there exists a unique smooth map $\phi:M\to N$ such that $\Phi$ is a differential operator along $\phi$.
\end{prop}
\begin{proof}
For each $p\in N$ we consider the evaluation maps $\ev_p:\C(N)\to\R$ then $I_p:=\ker\ev_p$ is a maximal ideal of $\C(N)$. The manifold $N$ being compact implies there is a correspondence between maximal ideals of $\C(N)$ and points of $N$. A similar results holds then for ideals of $\sect{\mathcal{R}N}$, that is: an ideal $I$ of $\sect{\mathcal{R}N}$ is maximal if and only if $I=\mathcal{R}^{\geq 1}_p N$. 

Now let $\eta\in\sect{\mathcal{R}N}$ and consider the map $\Psi_q:=\ev_q\circ\varepsilon\circ\Phi$ for a point $q\in M$. It is an algebra morphism $\Psi_q:\sect{\mathcal{R}N}\to\R$ and it is clearly surjective. Therefore (cf. \cite[theorem 2.6]{jacobson}) $\ker\Psi_p$ is a  maximal ideal of $\sect{\mathcal{R}N}$ and it therefore corresponds to a unique point $p\in N$. Setting $\phi(p)=q$ we get a well defined map $\phi:M\to N$. To see it is indeed smooth consider the composition $\varepsilon_M\circ\Phi\circ\iota_N$, where $\iota_N:\C(N)\hookrightarrow\sect{\mathcal{R}N}$ is the inclusion.  This map gives an algebra morphism $\widetilde{\phi}:\C(N)\to\C(M)$ such that its kernel when composed with $\ev_p$ is exactly $\ker\Psi_q$. lemma \ref{lema:smooth-algebras} renders $\widetilde\phi=\phi^*$ and tehrefore $\phi$ is smooth.

To see that $\Phi$ is a differential operator along $\phi$ we just invoke formula \eqref{eq:commutator-supersmooth-functions-2} and the fact that $\Phi$ is a unital superalgebra homomorphism. The result follows.
\end{proof}


\section[Proof of equivalence]{Proof of equivalence between our geometric approach and the usual approach}\label{sec:equiv}
\noindent  In \cite[section 4.2]{var} the definition of a supermanifold is given along the follwoing lines: 
 
\begin{defn}[Supermanifolds as ringed spaces]\label{def:ag-superfolds}
Let $M$ be a smooth manifold of dimension $m$. A \textbf{supermanifold of dimension $(m,n)$} is a pair $(M,\mathcal{O})$ where $\mathcal{O}$ is a sheaf of superalgebras that is locally isomorphic to $\C\otimes\Lambda \mathbf{S}^*$, for some vector space $\mathbf{S}$ of dimension $n$. 
\end{defn}
We'll refer to these objects as algebro-geometric supermanifolds.

The fact that the sheaf of algebras $\C$ (and not, for instance, the sheaf of real analytic functions) appears in this definition motivated us to consider a classical approach. Our fundamental result to justify our approach is then the following

\begin{teo}\label{thm:sheaf=bundle}
Let $(M,\mathcal{O})$ be a supermanifold of dimension $(m|n)$ in the sense of definition \ref{def:ag-superfolds} and let $\mathbf{S}$ be a vector space of dimension $n$. There exists a bundle of superalgebras $\mathcal{R}M$ such that every fibre $\mathcal{R}_pM$ is isomorphic to $\Lambda\mathbf{S}^*$ and with the further property that $\sect{\mathcal{R}M}\cong\mathcal{O}$ as sheaves.
\end{teo}
We'll postpone the proof to subsection \ref{sec:equiv}.

\begin{cor}[Batchelor's theorem, \cite{batchelor}]\label{cor:batchelors}
The structure sheaf of any supermanifold $\superfold{M}$ can be realized (in a non-cannonical way) as the sheaf of sections of a bundle of exterior algebras of finite rank.
\end{cor}
\begin{proof}
Choosing an inclusion $\iota\colon\C_M\hookrightarrow\sect{\mathcal{R}M}$ that splits the sequence
\begin{equation}\label{eq:batchelor}
\xymatrix@1{0\ar[r]&\sect{\mathcal{R}^{\geq 1}M}\ar@{^{(}->}[r]^{\iota}&\sect{\mathcal{R}M}\ar[r]^{\varepsilon}&\C(M)\ar[r]&0}
\end{equation}
yields the isomorphism using the previous theorem.
\end{proof}

\begin{nota}\label{rem:batchelor}
The argument just laid establishes the fact that Batchelor's theorem is equivalent to the splitting of \eqref{eq:batchelor}. The proof of theorem \ref{thm:sheaf=bundle} will actually show how to find such a splitting.
\end{nota}

The above results say that, at the level of objects, our category of smooth supermanifolds and the category of algebro--geometric supermanifolds are in a one--one correspondence. The following two results take care of the morphisms.

\begin{lema}
Let $M$ and $N$ be smooth manifolds and let $\C_M$ and $\C_N$ denote the sheaves of smooth functions over $M$ and $N$ respectively. Given a \emph{continuous} map $\phi:M\to N$, if $\Phi:\C_N\to\phi_*\C_M$ is a unital morphism of algebras over $\phi$ then $\phi$ is smooth and $\Phi=\phi^*$, the map of lemma \ref{lema:smooth-algebras}.
\end{lema}
\begin{proof}
Recall that if $U\subseteq N$ is an open set, the sheaf $\phi_*\C_M(U)$ is defined as $\C_M(\inv{\phi}(U))$, the restriction of $\C_M$ to the open set $\inv{\phi}(U)$ (cf \cite[section 4.1]{var}). Let $f\in\C_N(U)$ and $x\in\inv{\phi}(U)$. Then if $V\subset\R$ is any open subset containing $\phi(f(x))$ the set $W:=\inv{\phi}(\inv{f}(V))$ is relatively open on $\inv{\phi}(U)$. The restriction morphism $\operatorname{res}^{U}_{\inv{f}(V)}$ of the sheaf $\C_N$ induces the morphism
\begin{equation}
\Phi_{\inv{f}(V)}:\C_N(\inv{f}(V))\to\C_M(W)
\end{equation}
of sheaves. Consider the function $h:=f-f(\phi(x))\mathbf{1}$, where $\mathbf{1}$ is the function identically $1$ on $N$. The image of $h$ under $\Phi_{\inv{f}(V)}$ is 
\begin{equation}
g:=\Phi_{\inv{f}(V)}(h)=\Phi_{\inv{f}(V)}(f)-\Phi_{\inv{f}(V)}(f)(x)\mathbf{1}
\end{equation}
(recall that $\Phi_{\inv{f}(V)}(f)$ is by definition a smooth funcion on $W$). The function $g$ is not invertible on $W$ because $x\in W$ and therefore $g$ has a zero there. Since $U$ is arbitrary, we have $\Phi_{U}(f)(x)\in V$ for all $V$ containing $f(\phi(x))$. Therefore 
\begin{equation}\label{eq:sheaf-morphism}
\Phi_U(f)=\phi^*(f)=f\circ\phi. 
\end{equation}
Now let $\listupla{x}$ be any local coordinates on an open subset $\widetilde{U}$ of $U$. Each of these coordinates are  smooth functions on $\widetilde{U}$, and therefore $\Phi_{\widetilde{U}}(x_\mu)=x_\mu\circ\phi$ is smooth for all indices $\mu$, $1\leq\mu\leq n$. This is precisely the definition of a smooth map between smooth manifolds. Therefore $\phi$ is smooth. As a consequence of equation \eqref{eq:sheaf-morphism} we get $\Phi=\phi^*$. 
\end{proof}

\begin{teo}\label{thm:equivalent}
Let $(M, \mathcal{O}_M)$ and $(N,\mathcal{O}_N)$ be supermanifolds and let $\Phi:\mathcal{O}_N\to\phi_*\mathcal{O}_M$ be a sheaf morphism along the continuous map $\phi:M\to N$. Then $\Phi$ is a differential operator along $\phi$.
\end{teo}
\begin{proof}
Consider  inclusion $\iota_N:\C_N\to\mathcal{O}_N$ as in corollary \ref{cor:batchelors}. It is clear that $\varepsilon_M\circ\Phi\circ\iota_N:\C_N\to\C_M$ is a sheaf morphism over the continuous map $\phi$. The above lemma implies that $\phi$ is smooth and that $\varepsilon_M\circ\Phi\circ\iota_N=\phi^*$ so we recover the commutative diagram
\begin{equation*}
\xymatrix{
\sect{\mathcal{R}N}\ar[rr]_{\varepsilon_N} \ar[d]^{\Phi} && \C(N) \ar@/_1pc/[ll]_{\iota} \ar[d]^{\phi^*}\\
\sect{\mathcal{R}M}\ar[rr]^{\varepsilon_M} && \C(M) \ar@/^1pc/[ll]^{\iota}
}
\end{equation*}
similar to diagram  \eqref{eq:diagrama2}, so $\varepsilon_M\circ\Phi=\phi^*\circ\varepsilon_N$; therefore defining the commutator $[\Phi,f](\eta)=\Phi(f\eta)-(f\circ\phi)\Phi(\eta)$ for all smooth functions $f$ on $M$ and $\eta\in\sect{\mathcal{R}M}$ it's seen that $\Phi(f)-\phi^*(f)$ is nilpotent, so $\Phi$ is a differential operator along $\phi$.
\end{proof}

\subsection{Proof of theorem \ref{thm:sheaf=bundle}}\label{sec:proof}
\noindent Let $(M,\mathcal{O})$ be an algebro-geometric supermanifold. As we already said (remark \ref{rem:batchelor}), the theorem is equivalent to the splitting of the sequence
\[
\xymatrix@1{0\ar[r]&\mathcal{N}\,\ar@{^{(}->}[r]&\mathcal{O}\,\ar[r]_{\varepsilon}\ar[r]&\C\ar[r]&0}
\]
so as to give $\mathcal{O}$ the structure of a sheaf of $\C$-modules. We state this as
\begin{thm}\label{thm:jota}
There exists a unital sheaf homomorphism $j\colon\C\to\mathcal{O}$ such that $\varepsilon\circ j=\id_{\C}$.
\end{thm}

Let us define an equivalence relation for these maps:
\begin{defn}
Let $\iota,\iota'\colon\C\to\mathcal{O}$ be unital morphisms and $k\geq 0$. We say they are \textbf{equivalent up to degree $k$}, written $\iota\underset{k}{\sim}\iota'$, if $\im(\iota-\iota')\subseteq\mathcal{N}^{k+1}$.
\end{defn}

It is clear that, if the odd dimension of $(M,\mathcal{O})$ is $n$ and $k\geq\lfloor\frac{n}{2}\rfloor$, then $\iota\underset{k}{\sim}\iota'$ always holds. This observation will be the basis of the proof. We will use the following

\begin{lem}\label{lem:iota}
Let $\iota\underset{2k}{\sim}\iota'$. The map $\iota-\iota'$ is a derivation of $\C$ with values in $\mathcal{N}^{2k+2}$; furthermore, there exists a vector field $X$ with values in $\mathcal{N}^{2k}$ such that $\iota-\iota'\equiv X\mod\mathcal{N}^{2k}$.
\end{lem}
\begin{proof}
Let $f$ and $g$ be smooth functions on $M$. Using the fact that both $\iota$ and $\iota'$ are unital algebra homomorphisms it is straightforward to compute
\[
\begin{split}
(\iota-\iota')(fg)=\iota'(f)(\iota-\iota')(g)+(\iota-\iota')(f)\iota(g)
\end{split}
\]
Therefore if we define $X(f)=(\iota-\iota')(f)$ we immediately see that $X$ is the required vector field.
\end{proof}

Let us fix a locally finite open cover $\set[\alpha\in\mathrm{A}]{U_\alpha}$ of $M$ such that
\begin{enumerate}[label=\alph{enumi}.]
\item there exists an isomorphism $\tau_\alpha\colon\mathcal{O}_\alpha:=\mathcal{O}|_{U_\alpha}{\to}\C_\alpha\otimes\Lambda\mathbf{S}^*$ of superalgebras with unit, and
\item there exist $y_\alpha=(y_{\alpha}^1,\ldots,y_{\alpha}^m)$ smooth functions on $U_\alpha$ such that $(U_\alpha,y_\alpha)$ is a coordinate chart on $M$.
\end{enumerate}
Let $U_{\alpha\beta}:=U_{\alpha}\cap U_{\beta}$, and let $\mathcal{O}_{\alpha\beta}$ and $\C_{\alpha\beta}$ denote the corresponding sheaves.

By the above conditions we can fix $\xi^{\alpha}_{\mu}\in\C\otimes\Lambda_{+}^{\geq 2}\mathbf{S}^*$, $1\leq\mu\leq m$, such that 
\begin{equation}\label{eq:tau}
\tau_{\alpha}(y_{\alpha}^{\mu})=y_{\alpha}^{\mu}\otimes 1+\xi^{\alpha}_{\mu}. 
\end{equation}
Set $\tau_{\alpha\beta}=\tau_{\alpha}\circ\tau_{\beta}^{-1}$. The required morphism $j$ will be constructed locally and glued using these morphisms. Note that $\tau_{\alpha\beta}$ is \textsc{not} $\C$-linear in general. Set $\iota_{\alpha,0}(f)=f\otimes 1$; this is obviously the most simple inclusion $\C\to\C\otimes\Lambda\mathbf{S}^*$.

\begin{lema}\label{lem:iota-r}
For all $\alpha$ and all $r\geq 1$ there exist $\iota_{\alpha,r}\colon\C\to\C\otimes\Lambda\mathbf{S}^*$ such that for all $f\in\C_{\alpha}$
\begin{enumerate}[label=(\arabic{enumi})]
\item\label{cond:a} $\Big(\iota_{\alpha,r}-\iota_{\alpha,r-1}\Big)(f)\equiv 0\mod\Lambda_{+}^{2r}\mathbf{S}^*$
\item\label{cond:b} $\Big(\tau_{\alpha\beta}\circ\iota_{\alpha,r}-\iota_{\beta,r}\Big)(f)\equiv 0\mod\Lambda_{+}^{\geq 2r+2}\mathbf{S}^*$.
\end{enumerate}
\end{lema}
\begin{proof}
The proof is by induction on $r$. We will abbreviate $\Lambda_{+}^{\geq k}\mathbf{S}^*$ by $A^k$.

Recall we've set $\iota_{\alpha,0}(f)=f\otimes 1$. Formula \eqref{eq:tau} guarantees that
\[
Z=\sum\limits_{\mu=1}^{m}\xi^{\alpha}_{\mu}\frac{\partial}{\partial y_{\alpha}^\mu}
\]
is a vector field with values in $A^{2}$, so defining $\iota_{\alpha,1}=\iota_{\alpha,0}+Z$ it is readly seen that conditions \ref{cond:a} and \ref{cond:b} are satisfied.

Now for the induction step. Suppose we have defined $\iota_{\alpha,r}$ for $r\geq 1$ and that conditions \ref{cond:a} and \ref{cond:b} are met for these morphisms. Define
\begin{equation*}
X_{\alpha\beta}=\tau_{\alpha\beta}\iota_{\beta,r}+\iota_{\alpha,r}\mod A^{2r+4}
\end{equation*}
This is a vector field with values in $C_{\alpha\beta}\otimes \Lambda^{2r+2}\mathbf{S}^*$; in local coordinates its expression is
\begin{equation*}
X_{\alpha\beta}=\sum_{\mu=1}^{m}X_{\alpha\beta}^{\mu}\frac{\partial}{\partial y_\alpha^\mu}.
\end{equation*}
Then for any $\alpha,\beta,\gamma\in\mathrm{A}$ we have
\[
\begin{split}
X_{\alpha\gamma} %
&=\tau_{\alpha\beta}\circ(\tau_{\beta\gamma}\circ\iota_{\gamma,r}-\iota_{\beta,r})\\
&+(\tau_{\alpha\beta}\circ\iota_{\beta,r}-\iota_{\alpha,r})\\
&\overset{!}{=}\tau_{\alpha\beta}X_{\beta\gamma}+X_{\alpha\beta}\\
&=\tilde{\tau}_{\alpha\beta}(X_{\beta\gamma})+X_{\alpha\beta}
\end{split}
\]
where $\tilde{\tau}_{\alpha\beta}$ denotes the odd differential of $\tau_{\alpha\beta}$. This implies $\tilde{\tau}_{\alpha\beta}\in\C_{\alpha\beta}\times\Aut_{\Z}(A)$, i.e. it preserves the $\Z$-graduation; this implies $\tilde{\tau}_{\alpha\beta}$ is $\C_{\alpha\beta}$-linear. Observe also that for all $\alpha$ we have $X_{\alpha\alpha}=0$ and therefore $\tilde{\tau}_{\alpha\beta}(X_{\beta\alpha})=-X_{\alpha\beta}$.

Let $\set[\alpha\in\mathrm{A}]{\psi_\alpha}$ be a partition of unity subordinated to the cover $\set{U_\alpha}$ and define
\begin{equation}\label{eq:x-ext}
X_{\alpha}=\sum\limits_{\gamma\neq\alpha}\psi_{\gamma}\cdot X_{\alpha\gamma};
\end{equation}
it is manifest that $X_{\alpha}$ so defined is a smooth vector field on $\bigcup_{\gamma}U_{\alpha\gamma}$ and it can be smoothly extended by $0$ to the rest of $U_\alpha$. Now we compute

\begin{align*}
X_{\alpha}-\tilde{\tau}_{\alpha\beta}(X_{\beta}) %
&=\sum\limits_{\alpha\neq\gamma\neq\beta}\psi_{\gamma}\cdot X_{\gamma\alpha}-\tilde{\tau}_{\alpha\beta}(\psi_{\gamma}\cdot X_{\gamma\beta})&&\\
&=\sum_{\gamma} \psi_{\gamma}\cdot\Big(X_{\gamma\alpha}-\tilde{\tau}_{\alpha\beta}(X_{\beta\gamma})\Big)&&\text{($\tilde{\tau}_{\alpha\beta}$ is $\C_{\alpha\beta}$-linear)}\\
&=\sum_{\gamma}\psi_{\gamma}(X_{\gamma\alpha}+X_{\beta\gamma})&&(\tilde{\tau}(X_{\gamma\beta})=-X_{\beta\gamma})
\end{align*}
This computation shows that the original vector field $X_{\alpha\beta}$ can be smoothly extended to $U_\alpha\cup U_\beta$ and therefore to all of $M$. Lemma \ref{lem:iota} furnishes the unital homomorphism $\iota_{\alpha,r+1}=\iota_{\alpha,r}+X_{\alpha}$ for which condition \ref{cond:a} is readily satisfied. Condition \ref{cond:b} follows from the construction of $X_{\alpha}$ from $X_{\alpha\beta}$ in equation \eqref{eq:x-ext}: indeed, $\tau_{\alpha\beta}\circ\iota_{\alpha,r+1}+X_{\alpha}\cong\iota_{\beta,r+1}+X_{\beta}\mod A^{2r+2}$ since both sides of the congruence take values on $A^{2r+4}$ by construction.
\end{proof}

To finish the proof of theorem \ref{thm:sheaf=bundle} we construct for each $r\geq 1$ the inclusions $\iota_{\alpha,r}$; then, as observed before, if $r\geq\lfloor\frac{n}{2}\rfloor$ we get the identity $\tau_{\alpha\beta}\circ\iota_{\beta,r}=\iota_{\alpha,r}$. Then we can define $j\colon\C\to\mathcal{O}$ by $j(f)=\sum_{\alpha}\iota_{\alpha.r}(f)$; since theorem \ref{thm:jota} is equivalent to \ref{thm:sheaf=bundle} the latter theorem is completely proved.\qed

Now our claim that our approach yields the same theory as the algebro-geometric approach is completely established considering theorem \ref{thm:equivalent}.
\section{$\Z$-graded supermanifolds}
\noindent Let $\superfold{M}$ be a supermanifold. Recall that for every point $p$ of $M$ we have the bundle $\mathbf{S}_p M$ of odd directions and its dual $\mathbf{S}^*_p M$ and that, choosing a connection of algebras on $\mathcal R M$ we have an isomorphism $\mathcal R M \cong \Lambda\mathbf{S}^* M$. In this section we study supermanifolds arising from the inverse construction.

Let $\xi:E\to M$ and $\widetilde\xi:\widetilde E\to N$ be smooth vector bundles, and let $\Phi:E\to\widetilde E$ be a bundle morphism along a smooth map $\phi:M\to N$. Then the lift $\Lambda\Phi:\Lambda E\to \Lambda\widetilde E$ is a morphism of the associated supermanifolds $(M|\Lambda E)$ and $(N|\Lambda\widetilde E)$. Associated to $\Phi$ we get a map $\Phi^*:\Gamma(\widetilde E)\to\sect{E}$ as in lemaama \ref{lema:smooth-algebras}. This kind of morphisms are not only $\z$-graded but also $\Z$-graded. As a differential operator along $\phi$, the map $\Phi^*$ is as simple as it could be:

\begin{prop}\label{prop:morphism-Z-graded}
In the above setting, the map $\Phi^*$ is a differential operator of order $0$.
\end{prop}
\begin{proof}
It suffices to observe that $\Lambda\Phi$ is $\C(M)$-linear, that is
\[
\Lambda\Phi(f\omega)=(f\circ\phi)\Lambda\Phi(\omega)
\]
for all sections $\omega$ of $\Lambda(\widetilde E)$. Therefore the commutator $\twistcom{\Phi^*}{\phi}{f}$ vanishes for all smooth functions $f$ on $M$.
\end{proof}

In this setting we can make the identification $\mathbf{S}M=E$ and the supermanifold structure is thus completely determined by the vector bundle $\mathbf{S}M$. 

\begin{defn}
A \textbf{$\Z$-graded supermanifold} is a supermanifold of the form $(M|\Lambda(\mathbf{S}M))$ for some vector bundle $\mathbf{S}M$. A \textbf{$\Z$-graded smooth supermap} is the lifting to the exterior bundle of a bundle map $F:\mathbf{S}M\to\mathbf{S}N$.
\end{defn}

\noindent These supermanifolds are also referred to as \textbf{split supermanifolds}. We will denote such a supermanifold as $(M|\mathbf{S}M)$. Batchelor's theorem (corollary \ref{cor:batchelors}) is tantamount to saying that every smooth supermanifold $\superfold{M}$ is (non-naturally) isomorphic to a split supermanifold.

\section{The tangent and cotangent superbundles}
\noindent In this section we give the construction of the tangent and cotangent superbundles of a smooth supermanifold. We'll rely on the contents of appendix \ref{app:diff-ops}. We begin with a techincal albeit rather simple fact.

\begin{lema}\label{lema:two-bundle-maps}
Given a supersmooth map $(\phi|\Phi):\superfold{M}\to\superfold{N}$ there are two bundle maps $F\colon\mathbf{S}M\to\mathbf{S}N$ and $F^*\colon\mathbf{S}^*N\to\mathbf{S}^*M$ naturally associated to it.
\end{lema}
\begin{proof}
The map $F^*$ is just the bundle map $\Phi^1$ of corollary \ref{cor:preserve-filtration} and $F$ is its dual.
\end{proof}

Now let's recall that the tangent bundle of a smooth manifold $M$ has as sections the derivations of the algebra $\C(M)$. That is $\der(\C(M))=\sect{TM}$. This is a Lie algebra under the commutator of vector fields. If $\superfold{M}$ is any supermanifold, then the space of its superderivations should then be a Lie superalgebra. Because of the inclusion $\C(M)\hookrightarrow\sect{\mathcal{R}M}$ it is manifest that the Lie algebra of vector fields on $M$ must be a Lie subalgebra of the even part of this superalgebra. Now each fibre $\mathcal{R}_p M$ is a free supercommutative algebra with only odd generators, the space $\mathbf{S}^*_pM$. According to corollary \ref{cor:odd-derivations-ins}, the Lie superalgebra $\sder(\mathcal{R}_{p,+} M|\mathcal{R}_{p,-} M)$ is generated by $\mathbf{S}_p M$. Now, fixing an isomorphism $\phi:\Lambda\mathbf{S}^*_p M\to\mathcal{R}_p M$ we know (theorem \ref{thm:superderivations}) the space of superderivations of $\mathcal{R}_p M$ is isomorphic to $\mathbf{S}_p M\otimes\Lambda\mathbf{S}^*_p M$. 

\begin{defn}\label{def:pointwise-derivations}
For every point $p$ of $M$ the space 
\[
\sder_p(\mathcal{R}M)=\sder(\mathcal{R}_{p} M)
\]
 of \textbf{pointwise derivations} of $\mathcal{R}M$. The vector bundle 
\[
\sder(\mathcal{R}M):=\bigsqcup_{p\in M}\sder_p(\mathcal{R}M).
\]
is the \textbf{bundle of pointwise derivations} of $\superfold{M}$.
\end{defn}
The elemaaents of the fibres of this bundle are derivations of the corresponding fibre of $\mathcal{R}M$. A section of this bundle, then, acts only fibrewise and it's linear on each fibre and is therefore linear over the ring $\C(M)$. That is, if $\eta$ is a superfunction and $f$ is smooth on $M$ we have $D(f\eta)_p=D(f(p)\eta_p)=f(p)D(\eta_p)$ for every section $D$ of $\sder(\mathcal{R}M)$. Summarizing:

\begin{prop}\label{prop:der-order-zero}
The sections of the bundle $\sder(\mathcal{R}M)$ are differential operators of order $0$ over $\C(M)$.
\end{prop}

Clearly in the space $\sect{\sder(\mathcal{R}M)}$ we are missing a lot of superderivations of $\sect{\mathcal{R}M}$, as we expect the graded Leibniz rule to apply. That is, if $D$ is a superderivation of $\sect{\mathcal{R}M}$ we expect
\begin{equation}\label{eq:Leibniz}
D(\eta\psi)=D(\eta)\psi+(-1)^{\parity{\eta}\parity{D}}\eta D(\psi)
\end{equation}
and this identity is clearly not linear for smooth functions on $M$. Nevertheless this is the equation all superderivations must satisfy. This has an important consequence:

\begin{prop}
Superderivations of the algebra $\sect{\mathcal{R}M}$ are differential operators of order at most $1$ (cf. definition \ref{def:diff-op}).
\end{prop}
\begin{proof}
Let $D$ be a superderivation, $f$ a smooth function on $M$ and $\eta$ an arbitrary smooth superfunction. Considering the commutator $[D,f](\eta)$ (see \eqref{eq:definition-commutator}) we get
\[
[D,f](\eta)=D(f\eta)-fD(\eta)
\]
Since $D$ is a derivation, the Leibniz rule \eqref{eq:Leibniz} implies
\[
[D,f](\eta)=D(f)\eta+(-1)^{\parity{D}}fD(\eta)-fD(\eta).
\]
so if $D$ is an odd derivation the commutator above is zero for any smooth superfunction, whereas if $D$ is even we get zero above when iterating the commutator. Acording to our definition of a linear differential operator (definition \ref{def:diff-op}) the result follows.
\end{proof}

\begin{cor}\label{cor:superderivations-differential-operators}
The space of superderivations of $\sect{\mathcal{R}M}$ is the space of sections of a bundle $\Der\mathcal{R}M$ that is a sub-bundle of $\mathbf{Diff}^{\leq 1}(\mathcal{R}M,\mathcal{R}M)$ of differential operators on the vector bundle $\mathcal{R}M$.
\end{cor}

We now proceed to the construction of the bundle of superderivations. Since smooth vector fields must be even elemaaents of the Lie superalgebra $\sder(\sect{\mathcal{R}M})$ we must include $\sect{TM}$ in the even subspace of $\sder(\sect{\mathcal{R}M}$. On the other hand, the space of generators of $\sder(\sect{\mathcal{R}M})$ must also contain $\sect{\mathbf{S}M}$ (by theorem \ref{thm:superderivations}) and at the same time the pointwise derivations (definition \ref{def:pointwise-derivations}). With these considerations in mind, let $f$ be a smooth function on $M$ and $D$ a superderivation of $\sect{\mathcal{R}M}$. We know that $D(f)$ is an even superfunction and, because of the requirements imposed by the Leibniz identity \eqref{eq:Leibniz}, we know it must be of the form $X(f)+\eta(f)$, with $X$ a vector field on $M$ and $\eta$ an even nilpotent superfunction (i.e., a section of $\mathcal{R}_{+}^{\geq 2}M$). So there seems to be a decomposition
\[
\sder(\sect{\mathcal{R}M})\cong \sect{\mathcal{R}M\otimes(TM\oplus\mathbf{S}M)},
\]
where the $\mathbf{S}M$ factor accounts for the generators of the space of superderivations. This is indeed the case because on a trivializing neighbourhood $U$ of $\mathcal{R} M$ around $p$ we get the isomorphism
\[
\sect{\sder\mathcal{R}U}\cong\sect{\mathbf{S}U\otimes\Lambda\mathbf{S}^*U}
\]
when choosing an isomorphism $\phi_U:\Lambda\mathbf{S}^*U\to\mathcal{R}U$. Batchelor's theorem (corollary \ref{cor:batchelors})produces a global isomorphism
\begin{equation}
\sder(\sect{\mathcal{R}M})\cong\sect{\Lambda\mathbf{S}^*M\otimes(TM\oplus\mathbf{S}M)}.
\end{equation}
which  depends on the bundle isomorphism $\phi:\Lambda\mathbf{S}M\to\mathcal{R}M$. We have to prove that this is indeed the vector bundle we sought.

\begin{teo}\label{thm:exact-sequence-derivations}
The following sequence of vector bundles is exact
\begin{equation}\label{eq:exact-sequence-derivations}
\xymatrix{
0\ar[r]&{\sder(\mathcal{R} M)}\ar@{^{(}->}[r]^{\iota}&\Der(\mathcal{R}M)\ar@{->>}[r]^{\sigma}&{\mathcal{R}M}\otimes TM\ar[r]& 0
}
\end{equation}
Here $\sigma$ denotes the principal symbol (definition \ref{def:principal-symbol}).
\end{teo}
\begin{proof}
We've seen that the bundle $(TM\oplus\mathbf{S}M)\otimes\Lambda\mathbf{S}^*M$ actually contains operators that act on the superfunctions as derivations, and that we can choose an isomorphism $\Der(\mathcal{R}M)\cong (TM\oplus\mathbf{S}M)\otimes\mathcal{R}M$. We need to verify that \eqref{eq:exact-sequence-derivations} is actually exact.

We know (proposition \ref{prop:der-order-zero}) the sections of $\sder\mathcal{R}M$ are differential operators of order zero, so their principal symbol is zero. This proves $\im(\iota)\subseteq\ker(\sigma)$. Also, this condition is sufficient for a differential operator to be of order zero, so $\im(\iota)=\ker(\sigma)$.

In order to prove that $\sigma$ is surjective we choose a connection $\nabla$ of the bundle $\mathcal{R}M$
. Let $X\otimes r$ be a section of $TM\otimes\mathcal{R}M$ and define a map $g:\sect{TM\otimes\mathcal{R}M}\to\sect{\Der(\mathcal{R}M)}$ by $g(X\otimes r)=r\nabla_X$, which is evidently a differential operator of order $1$, so we compute its principal symbol: let $\eta$ be a superfunction and $f$ a smooth function, then
\begin{equation}
	\begin{split}
	\sigma(r\nabla_X;f)(\eta) &=[r\nabla_X,f](\eta)\\
							  &=r\nabla_X(f\eta)-fr\nabla_X(\eta)\\
							  &=rX(f)\eta +rf\nabla_X(\eta)-fr\nabla_X(\eta)\\
							  &=rX(f)\eta\hspace{5cm}\text{(because $fr=rf$)}\\
							  &=X\otimes r(f\otimes\eta)
	\end{split}
\end{equation}
which proves that $g\circ\sigma$ is the identity on $TM\otimes\mathcal{R}M$, making $\sigma$ a surjective map.
\end{proof}

\begin{cor}\label{cor:split-supertangent}
If a split supermanifold $\superfoldi{M}$ has trivial bundle of odd directions, i.e. $\mathbf{S}M\cong M\times\mathbf{S}$ for some vector space $\mathbf{S}$, then there is a natural isomorphism
\[T(M|\mathbf{S}M)\cong \Lambda\mathbf{S}^*M\otimes(TM\oplus\mathbf{S}M)
\]
\end{cor}

As in the classical setting, the supertangent bundle is algebraically related to the superalgebra of smooth superfuntions on a supermanifold.

\begin{prop}\label{prop:supermodules}
The space $\sect{\Der(\mathcal{R}M)}$ is a left supermodule over $\sect{\mathcal{R}M}$ with generators $\sect{\mathbf{S}M}$.
\end{prop}
\begin{proof}
This is because for every $p\in M$ we know the supervector space $\Der_p(\mathcal{R}M)$ is a module over $\mathcal{R}_p M$ whose space of generators is $\mathbf{S}_p M$. Since both are vector bundles over $M$ the result follows from a partition of unity argument.
\end{proof}

\begin{nota}\label{remark:tangent-left}
We've worked with the tangent superbundle as a bundle of left modules over the algebra of smooth superfunctions. This is because when we define a pairing between the spaces of supervector fields and superforms it will be convenient, for the sake of removing innecesary signs, to view the latter as a right module. Compare to theorem \ref{thm:superderivations}.
\end{nota}

The space of $\C(M)$-linear endomorphisms of $\sect{\mathcal{R}M}$ is naturally equipped with a supercommutator:
\begin{equation}\label{eq:supercom-superfolds}
\lsem A,B\rsem=A\circ B-(-1)^{\parity{A}\parity{B}}B\circ A
\end{equation}
and by restriction so is the space of sections of $\Der(\mathcal{R}M)$. The commutator above turns these spaces into (infinite-dimensional) Lie superalgebras. We state the result, whose proof is trivial:

\begin{teo}\label{teo:superder-lie}
The space $\sect{\Der(\mathcal{R}M}$ is a Lie superalgebra with the commutator defined by \eqref{eq:supercom-superfolds}.
\end{teo}

\subsection{Tangential maps}
\noindent Let $(\phi|\Phi):\superfold{M}\to\superfold{N}$ be a supersmooth map, fixed throughout this subsection. Recall from lemaama \ref{lema:two-bundle-maps} that there are two bundle maps associated to it: $F:\mathbf{S}M\to\mathbf{S}N$ and $F^*:\mathbf{S}^*M\to\mathbf{S}^*N$. Let's recall that in the classical case, given a smooth map $\phi:M\to N$ there exist two bundle maps $\phi_*:TM\to TN$ and $\phi^*:T^*N\to T^*M$ naturally associated to $\phi$. In the supermanifold case, the tangent superbundle $T\superfold{M}=\Der(\mathcal{R}M)$ is generated by $TM\oplus\mathbf{S}M$. So now we can consider the maps
\begin{align}
(\phi|\Phi)_*:=	&\phi_*\oplus F^{\empty}\colon TM\oplus\mathbf{S}M 		\to TN\oplus\mathbf{S}N\\
(\phi|\Phi)^*:=	&\phi^*\oplus F^*\colon 	T^*N\oplus\mathbf{S}^*N		\to T^*M\oplus\mathbf{S}^*M\label{eq:codifferential}
\end{align}
where $F=(\Phi^1)^*$ is the map of corollary \ref{cor:preserve-filtration} and call them, respectively, the \textbf{differential} and \textbf{codifferential} of $(\phi|\Phi)$. It is manifest they satisfy analogous properties to those of the classical differential and codifferential. A noteworthy fact is the following

\begin{prop}
Let $\mathcal{I}$ be the functor that associates to any given supermanifold $\superfold{M}$ the corresponding $\Z$-graded supermanifold $(M|\mathbf{S}M)$ and let $F\colon\mathbf{S}M\to\mathbf{S}N$ be the map of lemaama \ref{lema:two-bundle-maps}. If $(\phi|\Phi):\superfold{M}\to\superfold{N}$ is a morphism of supermanifolds then $\mathcal{I}(\phi|\Phi)=(\phi|F)$.
\end{prop}
\begin{proof}
Recall that a morphism of $\Z$-graded supermanifolds is the exterior lift $\Lambda(G)$ of a morphism $G:\mathbf{S}M\to\mathbf{S}N$ over the smooth map $\phi:M\to N$. By corollary \ref{cor:preserve-filtration} we know $\Phi^1=F$ is naturally associated to $\Phi$ when taking the quotient of $\mathcal{R}M$ by its nilpotent sub-bundle. Since $\Lambda(\Phi^1)=\Lambda(F)$ we get a morphism $(\phi|F)$ of $\Z$-graded supermanifolds naturally associated to $(\phi|\Phi)$ via the quotient map. Thus $\mathcal{I}(\phi|\Phi)=(\phi|F)$.
\end{proof}

When we analysed smooth supermaps we found that they are differential operators along smooth maps; the proof of this fact relied on the operator
\[
[\Phi\empty_{\phi}\cdot]\colon\C(N)\to\sect{\mathcal{R}^{\geq 2}M}
\]
which is a derivation along $\phi$. That is:
\[
[\Phi\empty_{\phi}fg]=[\Phi\empty_{\phi}f](g\circ\phi)+(f\circ\phi)[\Phi\empty_{\phi}g]
\]
which can be seen by evaluating the commutator on the section $\mathbf{1}$ and using the fact that $\Phi$ is a superalgebra morphism. Because of the universal property of jets (theorem \ref{teo:jet-universal}) we get the following diagram:
\[
\xymatrix{
\C(N)\ar[r]^{[\Phi,\cdot]}\ar[d]_{\jet^1}	&\sect{\mathcal{R}^{\geq 2}M}\\
\sect{\Jet^1(N)}\ar[ur]_{\jet^1([\Phi.\cdot])}
}
\]
Because of corollary \ref{cor:cotangent-1-jet} the above diagram turns into
\begin{equation}
\xymatrix{
\C(N)\ar[r]^{\twistcom{\Phi}{\phi}{\cdot}}\ar[d]_{d}	&\sect{\mathcal{R}^{\geq 2}M}\\
\sect{T^*N}\ar[ur]_{\jet^1(\twistcom{\Phi}{\phi}{\cdot})}
}
\end{equation}
Let's now compute
$$\twistcom{\Phi}{\phi}{fg}=\Phi(fg)-(fg)\circ\phi=\Phi(f)\Phi(g)-(f\circ\phi)(g\circ\phi),$$
which turns into
\[
\begin{split}
\Phi(f)\Phi(g)-(f\circ\phi)(g\circ\phi)	&=\Phi(f)\Phi(g)+\Phi(f)(g\circ\phi)-\Phi(f)(g\circ\phi)\\
+(f\circ\phi)(g\circ\phi)				&=\Phi(f)\big(\Phi(g)-g\circ\phi\big)+(\Phi(f)-f\circ\phi)(g\circ\phi)\\
										&=([\Phi,f]+f\circ\phi)([\Phi,g])+([\Phi,f])(g\circ\phi)\\
										&=[\Phi,f][\Phi,g]+[\Phi,g](f\circ\phi)+([\Phi,f])(g\circ\phi);
\end{split}
\]
Now the term $\twistcom{\Phi}{\phi}{f}\twistcom{\Phi}{\phi}{g}$ is an even section of $\mathcal{R}^{\geq 4}M$ so it makes sense to consider the class of $\twistcom{\Phi}{\phi}{fg}$ in the quotient $\sect{\mathcal{R}^{\geq 2}M/\mathcal{R}^{\geq 3}M}$. 
Therefore we get the map
\begin{equation}\label{eq:auxiliary-differential}
\begin{split}
\Phi^!:T^*N				&\to\Lambda^2\mathbf{S}^*M=\mathcal{R}^{\geq 2}M/\mathcal{R}^{\geq 3}M\\
		df_{\phi(p)}	&\mapsto \Phi(f)_p-(f\circ\phi)(p)+{\mathcal{R}^{\geq 3}_p M}
\end{split}
\end{equation}
which is a derivation along $\phi$ in the sense described above.

\begin{defn}
Let $(\phi|\Phi):\superfold{M}\to\superfold{N}$ be a supersmooth map. The \textbf{auxiliary codifferential} of the supersmooth map $(\phi|\Phi)$ is the map $\Phi^!$ of \eqref{eq:auxiliary-differential}. The \textbf{auxiliary differential} is the dual map $\Phi_!:=(\Phi^!)^*$.
\end{defn}

\begin{lema}\label{lema:aux-differential}
If the map $\Phi:\sect{\mathcal{R}N}\to\sect{\mathcal{R}M}$ is a differential  operator of order $0$ then $\Phi_!\equiv 0$.
\end{lema}
\begin{proof}
Since $\Phi^!$ is the natural extension to $T^*M$ of the derivation $\twistcom{\Phi}{\phi}{f}$ if one is zero so is the other. But $\twistcom{\Phi}{\phi}{f}=\Phi(f)-f\circ\phi=0$ means $\Phi(f)=f\circ\phi$ which is a necessary condition for $\Phi$ to be a differential operator of order zero. To wit: given an arbitrary superfunction $\eta$ on $\superfold{M}$ the vanishing of the auxiliary (co)differential implies $0=\twistcom{\Phi}{\phi}{f}(\eta)=(\Phi(f)-f\circ\phi)\Phi(\eta)$ so we get $\Phi(f\eta)=(f\circ\phi)\Phi(\eta)$.
\end{proof}

\subsection{Straightenings}
A reasonable question to ask at this point is whether or not there are any ``higher order'' auxiliary differentials associated to the supermanifold morphism $(\phi|\Phi)$. Since $\Phi$ preserves the filtration \eqref{eq:filtration} it seems feasible that there be a map
\begin{equation}
\Phi^{\text{aux}}:\sect{\mathbf{S}^*N}\to\sect{\mathcal{R}^{\geq 3}M/\mathcal{R}^{\geq 4}M}\cong\sect{\Lambda^{3}\mathbf{S}^*M}.
\end{equation}
Let us analyse this possibility. First note that \eqref{eq:auxiliary-differential} is well defined because the commutator $\twistcom{\Phi}{\phi}{f}$ has a class modulo $\sect{\mathcal{R}^{\geq 3}M}$  independent of the any isomorphism between the vector bundles $\mathcal{R}^{\geq 2}M/\mathcal{R}^{\geq 3}M$ and $\Lambda^2\mathbf{S}^*M$; this in turn is due to the fact that the map $\varepsilon$ gives a natural ``truncation'' of $\Phi(f)$ to its non-nilpotent part. When trying to construct a map $\Phi^{\text{aux}}$ as above we see that if $\sigma\in\sect{\mathbf{S}^*N}$ then $\Phi(\sigma)$ is a section of $\mathcal{R}^{\geq 1}M$. So in order to get a section of $\Lambda^3\mathbf{S}^*M$ an isomorphism is needed that chooses an adequate class for $\Phi(\sigma)$. This isomorphism in turn depends on some other choices:

\begin{prop}
If $\Psi_k:\Lambda^k\mathbf{S}^*M\to\mathcal{R}^{\geq k}M/\mathcal{R}^{\geq k+1}M$ is a vector bundle isomorphism for all non-negative integers $k$ then there exists an algebra bundle isomorphism $\Psi:\Lambda\mathbf{S}^*M\to\mathcal{R}M$ such that $\Psi_k=\Psi|_{\Lambda^k\mathbf{S}^*M}$.
\end{prop}
\begin{proof}
Choose a connection on $\mathbf{S}M$; this gives a local basis $\setupla{s}$ for this bundle, then using the dual basis $\setupla{ds}$ extend to an algebra isomorphism via $ds_{\mu_1}\wedge\cdots\wedge ds_{\mu_k}\mapsto ds_{\mu_1}\cdots ds_{\mu_k}+\sect{\mathcal{R}^{k+1}M}$.
\end{proof}

Hence, in order to consistenly define a class for $\Phi(\sigma)$ we need a bundle isomorphism $\Psi\colon\Lambda\mathbf{S}^*\to\mathcal{R}M$ and a connection $\nabla$ on $\mathbf{S}M$. These pairs are in correspondence with a very special class of maps:

\begin{teo}[Flowbox coordinates for supermanifolds]\label{thm:flowbox}
Let $\superfold{M}$ be a supermanifold; let
\begin{itemize}
\item $\Conn(\mathbf{S}^*M)$ denote the set of connections in the vector bundle of odd directions;
\item $\Inc{\superfold{M}}$ denote the set $\set{\psi\colon(TM\oplus\mathbf{S}M)_{\bullet}\hookrightarrow\Der_{\bullet}(\mathcal{R}M)}$ of $\z$-graded inclusions of the bundle of generators $TM\oplus\mathbf{S}M$ into the bundle of superderivations, and
\item $\mathcal{S}\superfold{M}=\set{\Psi\colon\Lambda\mathbf{S}^*M\stackrel{\cong}{\to}\mathcal{R}M}$ the set of unital superalgebra bundle isomorphisms. 
\end{itemize}
There is a bijection of sets $\Inc(\superfold{M})\times\Conn{\mathbf{S}M}\leftrightarrow\mathcal{S}\superfold{M}$ satisfying the following properties:
\begin{enumerate}
\item if $s,\widetilde{s}$ are sections of $\mathbf{S}M$ then $\lsem\psi(s),\psi(\widetilde{s})\rsem=0$;
\item for any $s\in\sect{\mathbf{S}M}$ and any vector field $X$ on $M$: $\lsem\psi(X),\psi(s)\rsem=\psi(\nabla_X{s})$;
\item if $s$ is a section of $\mathbf{S}M$ and $\sigma$ a section of $\Lambda\mathbf{S}^*M$ then $\Psi(s\ins \sigma)=\psi(s)(\Psi(\sigma))$.
\end{enumerate}
\end{teo}

Before proving the above theorem let us analyse the geometric meaning of the bijection. As we saw in proposition \ref{prop:morphism-Z-graded}, morphisms in the category of split supermanifolds are rather simple: they're just exterior bundle maps associated to vector bundle morphisms; this means they're differential operators of order $0$, or what is the same, module morphisms $\Phi\colon\sect{\Lambda\mathbf{S}^*M}\to\sect{\Lambda\mathbf{S}^*N}$ covering a smooth map $\phi\colon N\to M$, which in turn provides a unital algebra homomorphism $\phi^*\colon\C(M)\to\C(N)$. On the other hand, morphisms of supermanifolds in general are, comparatively, quite more complicated. So what the flowbox coordinates theorem allows us to do is ``straighten up'' an arbitrary supermanifold $\superfold{M}$ into a split one. We thus call the mappings of the set $\mathcal{S}\superfold{M}$ \textbf{straightenings} of $\superfold{M}$. 

\begin{proof}[Proof of theorem \ref{thm:flowbox}]
Recall that the bundle of derivations is locally isomorphic to the space $\C(U,\sder(\Lambda\mathbf{S}^*))$ once a trivialisation is chosen on $\mathbf{S}U$. We observe that the linear map $\psi\colon (TU\oplus \mathbf{S}U)^\bullet\to\C(U,\sder^\bullet\Lambda\mathbf{S}^*)$ satisfies the hypotheses of lemaama \ref{lemma:derivations-isomorphism}: it is injective and the image of the odd part $\mathbf{S}U$ consists of supercommuting odd derivations. Thus this lemaama gives us an isomoprhism $\Psi\colon\Lambda\mathbf{S}^*U\to\mathcal{R}U$ with the properties 1 and 3 of the statement of the theorem. For property 2 we observe that $\lsem\psi(X),\psi(s)\rsem$ is an odd supervector field and the result of the bracket operation is uniquely determined by $X$ and $s$. 

The uniqueness of $\psi$ for a given $\Psi$ is also guaranteed by lemma \ref{lemma:derivations-isomorphism}. The theorem follows from a partition of unity argument.
\end{proof}

Let a straightening $\Psi$ of $\superfold{M}$ be given. If $(\phi|\Phi)\colon\superfold{M}\to\superfold{N}$ then we have a morphism 
$(\phi|\widetilde{\Phi})\colon(M|\mathbf{S}M)\to\superfold{N}$ given by the following diagram
\begin{equation}\label{eq:diag-straightening}
\xymatrix{%
\sect{\mathcal{R}N}\ar[r]^-{\widetilde{\Phi}}\ar[d]_{\Phi}&\sect{{\Lambda\mathbf{S}^*M}}\\
\sect{\mathcal{R}M}\ar[ur]_{\sect{\inv{\Psi}}}
}
\end{equation}

\noindent When dealing with the bundle $\Lambda\mathbf{S}^*M$ it is now possible to talk about the truncation of one of its sections: consider the map $\tr^{k}\colon\Lambda\mathbf{S}^*M\to\Lambda^{\leq k}\mathbf{S}^*M$ that truncates a given form to its part of degree at most $k$.

Now let $\sigma\in\sect{\mathbf{S}^*N}$ and consider the section $\widetilde{\Phi}(\sigma)\in\sect{\Lambda_{-}{\mathbf{S}^*M}}$, i.e. $\widetilde{\Phi}(\sigma)$ is of positive degree in the bundle $\Lambda\mathbf{S}^*M$. Taking the class of $\widetilde{\Phi}(\sigma)-\tr^1 \widetilde\Phi (\sigma)$ modulo $\sect{\Lambda^{\geq 5}\mathbf{S}^*M}$ we get a well defined map $\check{\Phi}^3\colon\mathbf{S}^* N\to\Lambda^3\mathbf{S}^*M$. If we continue in this manner we get bundle maps
\[
\check{\Phi}^\bullet\colon(T^* N|\mathbf{S}^* N)_\circ\to\Lambda^\bullet_\circ\mathbf{S}^* M
\]
where $\circ$ denotes the $\z$-grading and $\bullet$ the $\Z$-grading. We then have the following

\begin{prop}\label{prop:split-iff}
Let $(\phi|\Phi)\colon\superfold{M}\to\superfold{N}$. The unital superalgebra morphism 
$$\Phi\colon\sect{\mathcal{R}M}\to\sect{\mathcal{R}N}$$
is a differential operator of order zero along the smooth map $\phi$ if and only if each of $\check{\Phi}^k$ is zero for $k\geq 2$.
\end{prop}
\begin{proof}
The ``if'' part is rather trivial: observe that if $\Phi(f\eta)=(f\circ\phi)\eta$ then for $\sigma\in\sect{\mathbf{S}^*M}$ we know $\Phi(f\sigma)=(f\circ\phi)\sigma$ is a section of $\mathbf{S}^*M$; this is equivalent to $\check{\Phi}^3\equiv 0$. Likewise, lemma \ref{lema:aux-differential} tells us that $\Phi_{!}\equiv 0$ forces $\Phi$ to be a differential operator of order $0$.
Since, under the choice of a straightening, all sections of $\mathcal{R}N$ can be written as linear combination of products of sections of $\mathbf{S}^*N$ and $\C(N)$ we get one implication by the above observations.

For the converse let $\Phi\colon\sect{\mathcal{R}M}\to\sect{\mathcal{R}N}$ and recall that given a straightening of $\mathcal{R}M$ we get a map $\widetilde{\Phi}$ associated to $\Phi$ given by diagram \eqref{eq:diag-straightening}. Suppose $\check{\Phi}^k\equiv 0$ for all $k\geq 2$. Choosing a straightening for $\mathcal{R}M$ we know $\check{\Phi}^k\equiv 0$ means that for any $\eta\in\sect{\mathcal{R}^{\geq k}N}$ the section $\check{\Phi}(\eta)=\widetilde{\Phi}(\eta)-\tr^{k-1}(\Phi(\eta))=0$ modulo $\sect{\Lambda^{k+1}\mathbf{S}^*M}$ for all $k$ so it is actually a section of $\Lambda^k\mathbf{S}^*M$. In particular $\widetilde{\Phi}$ maps sections of $\mathbf{S}^*N$ to sections of $\mathbf{S}^*M$. This means $\twistcom{\Phi}{\phi}{f}$ preserves the grading of $\Lambda\mathbf{S}^*N$ (choosing a straightening for $\mathcal{R}N$ if necessary) and is thus identically zero for all $f$; that is, $\Phi$ is a differential operator of order zero.
\end{proof}

Thus we have seen that any ``higher-order auxiliary differentials'' of $(\phi|\Phi)\colon\superfold{M}\to\superfold{N}$ are codified by the set of straightenings of $\superfold{N}$ albeit the property of being a bundle map (equivalently, a differential operator of order zero) is independent of any straightening.

\section{Superdifferential forms}
\noindent Here we use the work done in section \ref{sec:twisted} to derive some theorems about exterior superdifferential forms.

\begin{defn}
Let $\superfold{M}$ be a supermanifold. A \textbf{superdifferential form of degree $k\geq 0$} is a $k$-multilinear map
\[
\omega\colon\sect{T\superfold{M}}\overbrace{\times\dotsm\times}^{k\, \text{times}}\sect{T\superfold{M}}\to\sect{\mathcal{R}M}
\]
that is also \textit{supermultilinear} over the algebra $\sect{\mathcal{R}M}$ of smooth superfunctions and superalternating.
\end{defn}

A $k$-form is then a section of the bundle $\Lambda^k T^*\superfold{M}$ and as such it takes $k$ supervector fields as arguments. In the classical case the value of a form is defined up to a sign depending on the order in which the arguments are inserted. This is also the case for superdifferential forms but since homogeneous supervector fields can be even or odd the sign depends not only on the order in which the fields are inserted but also on the parity of the arguments. So let $\listupla[k]{X}$ be homogeneous supervector fields and $\omega$ a $k$-form; for $\sigma\in S_k$ the action on the arguments of $\omega$ is
\begin{equation}\label{eq:sigma-omega}
\sigma\cdot\omega\tupla[k]{X}=\dfrac{\sgn\sigma}{\sgn^-\sigma}\omega(X_{\inv\sigma(1)},\ldots,X_{\inv\sigma(k)})
\end{equation}
where $\sgn^-\sigma$ was defined by equation \eqref{eq:odd-signature}. 

Our definition above stresses the fact that a $k$-form is \textit{supermultilinear} over the superfunctions. What this means, for $k=1$, is that if $f$ is a homogeneous superfunction and $\omega$ a  $1$-form then 
\begin{equation}\label{eq:supermultilinear}
\omega(fX)=f\omega(X)
\end{equation}
(see remark \ref{remark:tangent-left}).

With equations \eqref{eq:sigma-omega} and \eqref{eq:supermultilinear} we can now prove the following

\begin{prop}
Let $\omega$ be a $k$-form. For all $\listupla[k]{X}$  homogeneous supervector fields and for all $f$ homogeneous superfunction we have
\begin{equation}\label{eq:k-supermultilinear}
\omega(X_1,\ldots,fX_s,\ldots,X_k)=(-1)^{\parity{f}\parity{\omega}(\parity{X_1}+\dotsm+\parity{X_{s-1}})}\omega(X_1,\ldots,X_s,\ldots,X_k)
\end{equation}
for $1\leq s\leq k$.
\end{prop}
\begin{proof}
If $s=1$ then \eqref{eq:k-supermultilinear} turns into \eqref{eq:supermultilinear} when the other $k-1$ entries are fixed; for $s>1$ conjugate by the transposition $(1s)$ and apply both \eqref{eq:sigma-omega} and \eqref{eq:supermultilinear}. The result follows from these observations.
\end{proof}

\subsection{The de Rham supercomplex}
\noindent Now we turn our attention to the exterior derivative of superforms. The equation
\begin{equation}\label{eq:exterior-derivative}
\begin{split}
d\omega\tuplao[k+1]{X}	&=\sum_{\mu=1}^{k+1}(-1)^\mu X_\mu\omega(X_0,\ldots,\widehat{X_\mu},\ldots,X_{k+1})\\
						&+\sum_{\nu<\mu}(-1)^{\mu+\nu}\omega([ X_\nu ,X_\mu],X_0,\ldots,\widehat{X_\nu},\ldots,\widehat{X_\mu},\ldots,X_{k+1})
\end{split}
\end{equation}
of classical differential geometry seems promising, but again one has to take into account the parity of the supervector fields. In order to verify that the above formula works (i.e. that it satisfies all the properties expected of the exterior derivative) we could carry out the computations with the appropriate signs; we'll prove the corresponding formula below. Instead we now focus on abstract properties of the exterior derivative and we'll use theorem \ref{thm:flowbox} to compute the de Rham cohomology of a smooth supermanifold. Many results here depend on the work made on appendix \ref{app:cartan-poincare}·

An isomorphism $\Psi\colon\Lambda\mathbf{S}^*M\to\mathcal{R}M$ furnishes in an obvious way a superdiffeomorphism $$(\id|\Gamma\Psi)\colon\superfold{M}\to\superfoldi{M};$$ if the tangent bundle is to preserve any of its nice properties of classical differential geometry then it is to be expected that this superdiffeomorphism induce an isomorphism between the respective bundles of $\superfold{M}$ and $\superfoldi{M}$, and we know (corollary \ref{cor:split-supertangent}) what the supertangent bundle of the latter is. We now prove the following
\begin{lema}\label{lema:superforms}
Let a straightening $\Psi\colon\Lambda\mathbf{S}^*\to\mathcal{R}M$ of $\superfold{M}$ be given. Then $\Psi$  induces an isomorphism of superalgebra bundles
\[
\ring{\Psi}\colon\Lambda^{\bullet}T^*\superfold{M}\stackrel{\cong}{\to}\Big(\Lambda T^*M\otimes\Sym\mathbf{S}^*M\Big)^{\bullet}\otimes\Lambda\mathbf{S}^*M
\]
where $\bullet$ denotes the $\Z$-grading.   
\end{lema}
\begin{proof}
We know from corollary \ref{cor:split-supertangent} that $T\superfoldi{M}\cong\Lambda\mathbf{S}^*M\otimes(TM\oplus\mathbf{S}M)$. Then $\Psi\colon\Lambda\mathbf{S}^*M\to\mathcal{R}M$ induces the even map
\[
\ring{\psi}\colon\Lambda\mathbf{S}^*M\otimes(TM\oplus\mathbf{S}M)_\pm{\longrightarrow}\Der_\pm\mathcal{R}M
\]
given by $\ring{\psi}(\eta\otimes(X\oplus s))=\Psi(\eta)\wedge(\nabla_X+s\ins)$ which is also an isomorphism; setting $\ring\Psi=\Lambda({\ring{\psi}})^*$ the result follows.
\end{proof}

The $\Z$ grading of the algebra $(\Lambda T^*M\otimes\Sym\mathbf{S}^*M)_p$ at a given point is then given by 
\[
\bigoplus_{a+b=k}\Lambda^a T_p^*M\otimes\Sym^b\mathbf{S}_p^*M
\]
for any non-negative integer $k$; thus, for any given $k$, the bundle of superdifferential forms of total degree $k$ is isomorphic to
\begin{equation}\label{eq:z-grading-superforms}
\bigoplus_{a+b=k}\left(\Lambda^{a}T^*M\otimes\Sym^b\mathbf{S}^*M\right)\otimes\Lambda\mathbf{S}^*M.
\end{equation}

From the flowbox theorem (theorem \ref{thm:flowbox}) we know that given a straightening $\Psi$ we automatically get a connection $\nabla$ on the bundle $\mathbf{S}M$ of odd directions. We now turn to properties of general connections on vector bundles that will be useful to us.

Let $E$ be a vector bundle over $M$ and let $D$ be a connection on $E$. Then $D$ induces an operator $d^D$ in the space  of $E$-valued $k$-forms $\sect{\Lambda^k T^*M\otimes E}$, the sections of $E$ being of course identified with the case $k=0$, such that
\[d^D\colon\sect{\Lambda^k T^*M\otimes E}\to\sect{\Lambda^{k+1} T^*M\otimes E},\] 
called the \textbf{twisted exterior derivative} of the connection $D$, defined by a formula analogous to \eqref{eq:exterior-derivative}
\begin{equation}\label{eq:d-nabla}
\begin{split}
d^D\omega\tuplao[k+1]{X}&=\sum_{\mu=1}^{k+1}(-1)^\mu D_{X_\mu}\left(\omega(X_0,\ldots,\widehat{X_\mu},\ldots,X_{k+1})\right)\\
						&+\sum_{\nu<\mu}(-1)^{\mu+\nu}\omega([ X_\nu ,X_\mu],X_0,\ldots,\widehat{X_\nu},\ldots,\widehat{X_\mu},\ldots,X_{k+1})
\end{split}
\end{equation} 

One then gets the following identities (cf. \cite[chapter 1]{besse})
\begin{subequations}\label{eq:twisted-derivative}
\begin{align}
(d^D)^2			&=R^D\label{eq:d-square}\\
d^D(R^D)			&=0\quad\text{(Bianchi identity)}\label{eq:bianchi}
\end{align}
\end{subequations}
where $R^D$ denotes the curvature of $D$. 

With equations \eqref{eq:twisted-derivative} above and the connection $\nabla$ provided by the straightening $\Psi$ (theorem \ref{thm:flowbox}) we can compute the exterior derivative of superdifferential forms. We first digress on two operators that will be useful to us.

Let $\mathbf{S}$ be a vector space of dimension $n$ and consider the action of $\mathrm{GL}(\mathbf{S})$ on the space $\Sym\mathbf{S}^*\otimes\Lambda\mathbf{S}^*$. With the choice of a basis $\setupla{s}$ for $\mathbf{S}$ the infinitesimal action of $\mathfrak{gl}(\mathbf{S})$ on $\mathcal{A}$ is by derivations and can be written down as
\[
A\star=-\left(\sum_{\mu=1}^n ds_\mu\cdot(As_\mu)\ins\right)\otimes\id-\id\otimes\left(\sum_{\mu=1}^n ds_\mu\wedge(As_\mu)\ins\right)
\]
which can be readly checked by computing a derivative (here $\setupla{ds}$ denotes the dual basis).

\begin{defn}\label{def:twisted-shift}
The \textbf{twisted shift operators} $A\smalltriangleleft$ and $A\smalltriangleright$ are defined by
\begin{align*}
A\smalltriangleright=\sum_{\mu=1}^{n}A(s_\mu)\ins\otimes ds_\mu\wedge\\
A\smalltriangleleft=\sum_{\mu=1}^{n}ds_\mu\cdot\otimes A(s_\mu)\ins
\end{align*}
for $A\in\mathfrak{gl}(\mathbf{S})$
\end{defn}

Note that these operators act on $\Sym\mathbf{S}^*\otimes\Lambda\mathbf{S}^*$ with bidegree $(-1,1)$ and $(1,-1)$ respectively. This means that 
\[
\begin{split}
A\smalltriangleright\colon\Sym^{\bullet}\mathbf{S}^*\otimes\Lambda^{\circ}\mathbf{S}^*%
&\to\Sym^{\bullet-1}\mathbf{S}^*\otimes\Lambda^{\circ+1}\mathbf{S}^*\\
A\smalltriangleleft\colon\Sym^{\bullet}\mathbf{S}^*\otimes\Lambda^{\circ}\mathbf{S}^*%
&\to\Sym^{\bullet+1}\mathbf{S}^*\otimes\Lambda^{\circ-1}\mathbf{S}^*
\end{split}
\]
They are actually a particular case of the Cartan-Poincaré operators defined on appendix \ref{app:cartan-poincare}, equations \eqref{eq:cartan-poincare-operators}; as such, the properties given by the following lemaama follow from the proof of proposition \ref{prop:cartan-poincare-operators}.

\begin{lema}\label{lemaama:twisted-shift}
If $A,B\in\mathfrak{gl}(\mathbf{S})$ then 
\begin{align*}
\{A\smalltriangleleft, B\smalltriangleleft \}&=0\\
\{A\smalltriangleright, B\smalltriangleright\}&=0\\
\{A\smalltriangleright, B\smalltriangleleft\}&=-(AB)\star\otimes\id-\id\otimes(BA)\star
\end{align*}
where $\{X,Y\}=XY+YX$ is the anticommutator.
\end{lema}

From the proof of \ref{prop:cartan-poincare-operators} we also get the following
\begin{cor}\label{cor:twisted-shift-diag}
The operators $\id\smalltriangleleft$ and $\id\smalltriangleright$ are (co)boundary operators (i.e. $(\id\smalltriangleleft)^2=(\id\smalltriangleright)^2=0$) on the bigraded complex $\Sym^\bullet\mathbf{S}^*\otimes\Lambda\mathbf{S}^*$. Furthermore, they satisfy the identity
\[
\{\id\smalltriangleright,\id\smalltriangleleft\}=(k+l)\id_{\Sym^k\mathbf{S}^*\otimes\Lambda^l\mathbf{S}^*}
\]
\end{cor}

Finally, as a consequence of lemaama \ref{lema:homology} and the above corollary we get
\begin{cor}\label{cor:twisted-shift-homology}
The cohomologies of $\id\smalltriangleleft$ and $\id\smalltriangleright$ satisfy
\[
H^{k,l}(\Sym\mathbf{S}^*\otimes\Lambda\mathbf{S}^*,\id\smalltriangleleft)=H^{k,l}(\Sym\mathbf{S}^*\otimes\Lambda\mathbf{S}^*,\id\smalltriangleright)=%
\begin{cases}
\set{0},& \text{if }(k,l)\neq (0,0)\\
\R,&		\text{if }(k,l)=0
\end{cases}
\]
\end{cor}

We now show that the analogue of formula \eqref{eq:exterior-derivative} is valid for superdifferential forms. Let's recall the axioms a linear map $d$ must satisfy in order to be the exterior (super)derivative:
\begin{enumerate}
\item If $f$ is a smooth superfunction then $df(D)=Df$ for all superderivations $D$.
\item $d^2=0$.
\item $d(\alpha\wedge\beta)=d\alpha\wedge\beta+(-1)^{\parity{\alpha}}\alpha\wedge d\beta$.
\end{enumerate}
Note that axiom 3 above forces $d$ to be even.

\begin{lema}\label{lemaa:horrible-formula}
Let $\omega$ be a $k$-superform on the supermanifold $\superfold{M}$ and suppose $\listuplao[k]{D}$ are homogeneous supervector fields. Then the exterior derivative is given by
\begin{equation}\label{eq:super-exterior-derivative}
\begin{split}
d \omega\tuplao[k]{D}%
&:=\sum_{\sigma\in S_{k+1}/S_1\times S_k}\frac{\sgn\sigma}{\sgn^-\sigma}D_{\inv{\sigma}(0)}\omega(D_{\inv{\sigma}(1)},\ldots,D_{\inv{\sigma}(k)})\\
&-\sum_{\sigma\in S_{k+1}/S_2\times S_{k-1}}\frac{\sgn\sigma}{\sgn^-\sigma} \omega(\lsem D_{\inv{\sigma}(0)},D_{\inv{\sigma}(1)}\rsem,D_{\inv{\sigma}(2)},\ldots,D_{\inv{\sigma}(k)})
\end{split}
\end{equation}
\end{lema}
\begin{proof}[Sketch of proof]
Axiom 1 for the exterior derivative is readly checked, taking into account the fact (remark \ref{remark:tangent-left}) that the space of superforms is a right module over the superalgebra $\sect{\mathcal{R}M}$. Axioms 2 and 3 follow from lengthy albeit easy computations; for instance, for $k=0$ we get
\[
\begin{split}
d^2\omega(D_0,D_1)&=
 D_0 d\omega(D_1)-(-1)^{\parity{D_0}\parity{D_1}}D_1 d\omega{D_0}-d\omega(\lsem D_0,D_1\rsem)\\
&= D_0 D_1\omega-(-1)^{\parity{D_0}\parity{D_1}}D_1 D_0\omega-\lsem D_0,D_1\rsem\omega
\end{split}
\]
which vanishes for the definition of the superbracket; a similar computation yields the Leibniz rule for the product of two superfunctions.
\end{proof}

Let's see where we are: so far we've seen that formula \eqref{eq:exterior-derivative} is valid for superdifferential forms and that given a straightening $\Psi$ of a supermanifold $\superfold{M}$ we can write down a superdifferential form in a very straightforward and simple way (lemma \ref{lema:superforms}). Also the straightening $\Psi$ gives us a connection $\nabla$ on the bundle $\mathbf{S}M$ of odd directions (corollary \ref{cor:flowbox}) and we've stated some properties of general affine connections on vector bundles (equations \eqref{eq:twisted-derivative}); then we showed some properties of certain operators (definition \ref{def:twisted-shift}, lemaama \ref{lemaama:twisted-shift}, and corollaries \ref{cor:twisted-shift-diag} and \ref{cor:twisted-shift-homology}) on the $\Z$-bigraded algebra $\Sym^\bullet\mathbf{S}^*\otimes\Lambda^\circ\mathbf{S}^*$ for $\mathbf{S}$ a finite dimensional vector space; considering the vector bundle $\mathbf{S}M$ denote the resulting algebra ``bundle''
\footnote{Technically the total space is not a bundle because the algebra $\mathcal{A}$ is infinite-dimensional; however, we'll focus on homogeneous components, which are finite dimensional.} 
by $\mathcal{A}^{\bullet,\circ}M$ and note that lemma \ref{lema:superforms} allows us to think on a superdifferential $k$-forms in two ways:

\begin{enumerate}
\item as a form in $M$ with values in the vector bundle $\mathcal{A}^{\bullet,\circ}$ (this is $\Z$-bigraded although we'll only consider the $\bullet$-grading);
\item as a ``superpolinomial'' of degree $k$ with values in $\Lambda\mathbf{S}^*M$.
\end{enumerate}
The first point of view is justified by decomposition \eqref{eq:z-grading-superforms} and so the homogeneous part of a $k$-superform is given by a differential form of degree $a$ on $M$ whose value is a polynomial of degree $b$ on $\mathbf{S}M$ such that, when evaluated, gives an alternating form on the odd codirections. The second point of view fixes $a$ and $b$ on decomposition \eqref{eq:z-grading-superforms} and returns the tensor product $\alpha\otimes\mathbf{p}$ of a differential form $\alpha$ of degree $a$ on $M$ and a symmetric form $\mathbf{p}$ of degree $b$, such that it returns a superfunction (i.e. a skew-symmetric form on $\mathbf{S}M$) when appropriately evaluated. Note that these decompositions depend in turn on the given straightening $\Psi$.

Let us now fix some notation. Given a superform $\omega$ of degree $k$ we'll work on the homogeneous components given by \eqref{eq:z-grading-superforms}; given a straightening $\Psi$ we'll denote by $\nabla^*$ the induced connection on the bundle $\mathbf{S}^*M$ of odd codirections; this connection in turns induces similar operations on the homogeneous components of $\mathcal{A}^{\bullet,\circ}M$. Decomposition \eqref{eq:z-grading-superforms} and the associativity of the tensor products and direct sums allow us to write $\omega$ and $\nabla^*$ in two different ways:
\begin{enumerate}[label=\bf{\Alph*}]
\item $\widetilde{\omega}\in\sect{\Lambda^{a}T^*M\otimes\big( \Sym^b\mathbf{S}^*\otimes\Lambda\mathbf{S}^*M \big)}$ (accordingly, write $\widetilde{\nabla}$ for the induced connection on $\mathcal{A}^{b,\circ}$);
\item $\check{\omega}\in\sect{\big( \Lambda^a T^*M\otimes\Sym^b\mathbf{S}^*M \big)\otimes\Lambda\mathbf{S}^*M}$ (accordingly, write $\check{\nabla}$ thought of as a connection on the bundle $\Sym^b\mathbf{S}^*M$).
\end{enumerate}
Note that we omit the $\empty^*$ superscript on the connection. We'll denote by $R^*,\check{R}$ and $\widetilde{R}$ the curvatures of the connections $\nabla^*,\check{\nabla}$ and $\widetilde{\nabla}$ respectively.

With all of the above results and remarks we now state and prove the main result of this section:

\begin{teo}\label{thm:super-exterior-derivative}
Let a straightening $\Psi$ of $\superfold{M}$ be given. Using the connection $\nabla$ provided by $\Psi$ and the isomorphism $\ring\Psi$ of lemma \ref{lema:superforms} the exterior derivative for superdifferential forms becomes
\begin{equation}\label{eq:super-d}
\ring\Psi\circ d\circ\inv{\ring\Psi}=d^{\widetilde\nabla}+(-1)^N \Big(\id\smalltriangleleft+\check{R}\smalltriangleright\Big)
\end{equation}
where $N$ is the operator of numbers (definition \ref{def:op-numbers}) of the algebra $\sect{\Lambda T^*M}$ of smooth differential forms on $M$ (that is, the map that assigns the integer $a$ to $a$-forms on $M$).
\end{teo}
Let us fix $a$ and $b$ in decomposition \eqref{eq:z-grading-superforms} and let $\omega$ be a $k$-superform. With the notation above, we'll write $\widetilde{\omega}\tupla[a]{X}\tupla[b]{s}$ for case \textbf{A} and $\check{\omega}(\listupla[a]{X};\listupla[b]{s})$ for case \textbf{B}; here the $X$'s denote vector fields on $M$ and the $s$'s sections of $\mathbf{S}M$. Of course, since $d\omega$ will take an input of $k+1$ superfields we decompose it in two of its homogeneous components $(d\omega)^{(a+1,b)}$ and $(d\omega)^{(a,b+1}$; the surprising fact of this result is that $d\omega$ only depends on these two homogeneous components.

\begin{proof}[Proof of theorem \ref{thm:super-exterior-derivative}]
Let $\listuplao[a]{X}$ be $a+1$ arbitrary vector fields on $M$ and $\listuplao[b]{s}$ arbitrary sections of $\mathbf{S}M$. Using the inclusion $\psi$ provided by the straightening $\Psi$ and insterting the values of $X$ and $s$ in formula \eqref{eq:super-exterior-derivative} we get

\begin{align}
&d\omega(\listuplao[a+1]{\psi{X}},\listupla[b]{\psi s})\nonumber\hfill\\
&=\sum\limits_{\mu=0}^{a}(-1)^{\mu} \psi X_\mu\left[\omega(\psi X_0,\ldots,\widehat{\psi X_\mu},\ldots,\psi X_a,\listupla[b]{\psi s})\right]\tag{a}\label{diffext-a}\\
&+\sum\limits_{\nu<\mu}(-1)^{\mu+\nu}\omega(\lsem\psi X_\mu,\psi X_\nu\rsem,\psi X_0,\ldots,\widehat{\psi X_\nu},\ldots,\widehat{\psi X_\mu},\ldots,X_a,\psi s_1,\ldots,\psi s_b)\tag{b}\label{diffext-b}\\
&+\sum\limits_{\mu,\lambda}(-1)^{\mu +a}\omega(\lsem\psi X_\mu,\psi s_\lambda\rsem,\psi X_0,\ldots,\widehat{\psi X_\mu},\ldots,X_a,\psi s_1,\ldots,\widehat{\psi s_\lambda},\ldots,\psi s_b)\tag{c}\label{diffext-c}\\
\intertext{and}
&d\omega(\listupla[a]{\psi{X}},\listuplao[b]{\psi{s}})\nonumber\hfill\\
&=\sum\limits_{\lambda}(-1)^{a} \psi s_\lambda\left[  \omega(\listupla[a]{\psi X},\psi s_0,\ldots,\widehat{\psi s_\lambda},\ldots,\psi s_b)\right]\tag{d}\label{diffext-d}
\end{align}

We now explain the signs for each term. Terms \eqref{diffext-a} and \eqref{diffext-b} are the alternating sums only over the vector field arguments so only the index $\mu$ contributes to the total sign; term \eqref{diffext-c} is the intertwining of both even and odd generators and since even and odd elemaaents anticommute in this case (because $\omega$ is superalternating) we get a contribution of $\mu$ to the sign by the even generator and by $a$ for the odd one because $s_\lambda$ has to commute with $\mu$ vector fields in order to get where it is in the term; finally, term \eqref{diffext-d} has sign $(-1)^a$ for the same reasons as the exponent $a$ appear in term \eqref{diffext-c}. On terms \eqref{diffext-a} and \eqref{diffext-b} the generators $\psi X_\mu$ and $\psi s_\lambda$ are acting as superderivations on the superfunctions $\omega(\psi X_0,\ldots,\widehat{\psi X_\mu},\ldots,\psi X_a,\listupla[b]{\psi s})$ and $\omega(\listupla[a]{\psi X},\psi s_0,\ldots,\widehat{\psi s_\lambda},\ldots,\psi s_b)$ respectively.

We now compute each term separately.

\begin{align}
\eqref{diffext-a}%
&=\sum\limits_{\mu}(-1)^{\mu}\check{\nabla}_{X_\mu}\left( \check{\omega}(X_0,\ldots,\widehat{X_\mu},\ldots,X_a) ;\listupla[b]{s})\right)\nonumber\\
&=\sum\limits_{\mu}(-1)^{\mu}\widetilde{\nabla}_{X_\mu}\left( \widetilde{\omega}(X_0,\ldots,\widehat{X_\mu},\ldots,X_a)\right)\tupla[b]{s}\tag{A}\label{AA}
\end{align}

Let's recall that the flowbox theorem (or rather, corollary \ref{cor:flowbox}) states that the connection $\nabla$ assigns the even superderivation $\nabla^*_X$ to each vector field $X$ on $M$. Now, since the curvature of $\nabla$ is defined as $R^\nabla_{X,Y}=[\nabla_X,\nabla_Y]-\nabla_{[X,Y]}$ then it follows that to the vector field $[X,Y]$ the even superderivation $R^\nabla_{X,Y}+[\nabla_X,\nabla_Y]$ corresponds. Now we compute
\begin{align}
\eqref{diffext-b}%
&=\sum\limits_{\nu<\mu}(-1)^{\mu+\nu}\check{\omega}(\lsem\check{\nabla}_{X_\nu},\check{\nabla}_{X_\mu}\rsem,X_0,\ldots,\widehat{X_\nu},\ldots,\widehat{X_\mu},\ldots,X_a\ldots;s_1,\ldots,s_b)\nonumber\\
&=\sum\limits_{\nu<\mu}(-1)^{\mu+\nu}\check{\omega}(\check{R}_{X_\nu,X_\mu},X_0,\ldots,\widehat{X_\nu},\ldots,\widehat{X_\mu},\ldots,X_a;\listupla[b]{s})\tag{b.1}\label{eq:b.1}\\
&+\sum\limits_{\nu<\mu}(-1)^{\mu+\nu}\check{\omega}(\check{\nabla}_{[X_\nu,X_\mu]},X_0,\ldots,\widehat{X_\nu},\ldots,\widehat{X_\mu},\ldots,X_a;\listupla[b]{s})\tag{b.2}\label{eq:b.2}
\end{align}
The subterms \eqref{eq:b.1} and \eqref{eq:b.2} need special considerations. First observe that the endomorphism $R_{X,Y}^*$ (and all its extensions) can be written, using a local basis $\setupla{s}$ and its dual $\setupla{ds}$, as
\begin{equation}\label{eq:r-nabla}
\begin{split}
R_{X,Y}^*%
&=\sum_{\alpha}(s_\alpha\circ R_{X,Y}^*)ds_\alpha\\
&=-\sum_{\alpha}R_{X,Y}(s_\alpha)ds_\alpha
\end{split}
\end{equation}
the last equality stemming from the fact that the dual connection (and its extensions) carry a sign times the original connection. So finally we get for the latter two subterms:

\begin{align*}
\eqref{eq:b.1}%
&=\sum\limits_{\nu,\mu,\alpha}(-1)^{\nu+\mu}\check{\omega}\left( \underline{-}\check{R}_{X_\nu,X_\mu}(s_\alpha)ds_\alpha,X_0,\ldots,\widehat{X_\nu},\ldots,\widehat{X_\mu},\ldots,X_a;\listupla[b]{s} \right)\\
&=\sum\limits_{\nu,\mu,\alpha}(-1)^{\mu+\nu\underline{+1}}\check{\omega}(R_{X_\nu,X_\mu}(s_\alpha)ds_\alpha,X_0,\ldots,\widehat{X_\nu},\ldots,\widehat{X_\mu},\ldots,X_a;\listupla[b]{s})
\end{align*}
note that we've underlined the corresponding signs; also, the factor $R_{X_\mu,X_\nu}(s_\alpha)$ is an odd generator and so it's misplaced in the above equality. Therefore a new sign appears due to the commuting of this factor with the remaining $a-1$ vector fields. Finally, the factor $ds_\alpha$ is a superfunction and so it goes out of the arguments of $\omega$ by superlinearity; since it's in the first ``slot'' it does not carry a sign; thus:
\begin{align}
\eqref{eq:b.1}%
&=\sum\limits_{\mu,\nu,\alpha}(-1)^{\mu+\nu+1+a-1}ds_\alpha\wedge\check{\omega}\left( X_0,\ldots,\widehat{X_\nu},\ldots,\widehat{X_\mu},\ldots,X_a;R_{X_\mu,X_\nu}(s_\alpha)\listupla[b]{s} \right)\nonumber\\
&=\sum\limits_{\mu,\nu,\alpha}(-1)^{\mu+\nu+a}ds_\alpha\wedge\check{\omega}\left( X_0,\ldots,\widehat{X_\nu},\ldots,\widehat{X_\mu},\ldots,X_a;R_{X_\mu,X_\nu}(s_\alpha)\listupla[b]{s} \right)\tag{B--1}\label{B-1}
\end{align}
Note the factor $(-1)^a$ and recall that $\check\omega$ is a superpolynomial with even degree $a$ or, equivalently, an $a$-form on $M$ with values on an appropriate bundle. As for the subterm \eqref{eq:b.2} it must be observed that the argument $\nabla_{[X_\nu,X_\mu]}$ corresponds under the inclusion $\psi$ to $[X_\nu,X_\mu]$, therefore
\begin{equation}
\eqref{eq:b.2}=\widetilde{\omega}\left( [X_\nu,X_\mu],X_0,\ldots,\widehat{X_\nu},\ldots,\widehat{X_\mu},\ldots,X_a \right)\tupla[b]{s}\tag{B--2}\label{B-2}
\end{equation}

Next recall that the supercommutator $\lsem \psi X,\psi s\rsem$ corresponds to $\psi\left(\nabla_X s\right)$, which yields
\begin{equation}
\eqref{diffext-c}=(-1)^{\mu}\check{\omega}(X_0,\ldots,\widehat{X_\mu},\ldots,X_a;\nabla_{X_\mu}s_\lambda s_1\cdots \widehat{s_\lambda}\cdots s_b)\tag{C}\label{CC}
\end{equation}
Notice that the factor $(-1)^a$ no longer appears since $\nabla_{X_\mu}s_\lambda$ has to return to its place with the odd generators and so it has to commute with $a$ vector fields. Also notice that we used the notation $s_1\cdots s_b$ because $\omega$ is symmetric in the last $k-a$ terms. As for the last term we use a similar expresion for the operator $s_\lambda$ to the one used for the curvature (see \eqref{eq:r-nabla}); it then becomes
\[
s_\lambda\ins=\sum\limits_{\alpha} (s_\alpha\circ s_\lambda)\ins ds_\alpha
\]
Of course only one term survives because of the identity $s_\lambda\ins ds_\alpha=\delta_{\lambda\alpha}$ but nonetheless this trick yields
\begin{align}
\eqref{diffext-d}%
&=\sum\limits_{\lambda}(-1)^a s_\lambda\ins\widetilde{\omega}\tupla[a]{X}(s_0\cdots\widehat{s_\lambda}\cdots s_b)\nonumber\\
&=\sum\limits_{\alpha,\lambda}(-1)^a s_\alpha\circ s_\lambda\ins ds_\alpha\widetilde{\omega}\tupla[a]{X}(s_0\cdots\widehat{s_\lambda}\cdots s_b)\nonumber\\
&=\sum\limits_{\alpha,\lambda}ds_\alpha\widetilde{\omega}\tupla[a]{X}(s_0\cdots\widehat{s_\lambda}\cdots s_b)s_\alpha(s_\lambda\ins) \tag{D}\label{DD}
\end{align}
which is exactly the expresion in the bases considered for the operator $s_\lambda\ins$.

Putting everything together note that terms \eqref{AA} and \eqref{B-2} yield $d^{\widetilde{\nabla}}$ (as defined in \eqref{eq:d-nabla}). Term \eqref{B-1} is equivalent to the operator $(-1)^{N}\check{R}\smalltriangleright$ whereas term \eqref{DD} is equivalent to $(-1)^{N}\id\smalltriangleleft$. So formula \eqref{eq:super-d} follows.
\end{proof}

\begin{nota}
Since any lower-degree homogeneous terms of $d\omega$ would insert either another vector field or another odd generator it follows that the only terms that appear when substituting with their images under $\psi$ will be of the form $\nabla_X$, $s\ins$ and extensions of $R^{\nabla}$. The fact that the left hand side of \eqref{eq:super-d} is independent of $\psi$ shows that this formula is also.
\end{nota}

We finalize this chapter with an important result.

\begin{teo}\label{thm:deRham}
Given the (non-canonical) inclusion of the de Rham complex $\Lambda^\bullet T^*M$ into the complex $\Lambda^\bullet T^*\superfold{M}$ the cohomology $H^*(\Lambda^\bullet T^*\superfold{M})$ equals the de Rham cohomology $H^*_{dR}(M)$. In other words, the inclusion $\Lambda^\bullet T^*M\hookrightarrow\Lambda^\bullet T^*\superfold{M}$ is a quasi-isomorphism of complexes.
\end{teo}

For the proof of this we'll need the following result, which is a consequence of \ref{lemaama:twisted-shift}
\begin{lema}
The operator $\Delta:=\{d,(-1)^{N}\id\smalltriangleright\}-\{\id\smalltriangleleft,\id\smalltriangleright\}$ is diagonalizable on $\Lambda T^*(M|\mathbf{S}M)$ with kernel $\Lambda T^*M$: $N$ denotes, again, the operator of numbers on $\Lambda T^*M$.
\end{lema}

\begin{proof}[Proof of \ref{thm:deRham}]
Given a straightening $\Psi$ of $\superfold{M}$ we immerse the complex $\Lambda^{\bullet}T^*M$ into $\Lambda^{\bullet}T^*\superfold{M}$ and denote by $d_\psi$ the induced operator $\ring{\Psi}\circ d\circ\inv{\ring{\Psi}}$. By theorem \ref{thm:super-exterior-derivative} we know how to compute the exterior derivative; since the inclusion induces the identity on the bundle $\Lambda T^*M\otimes\big( \Sym\mathbf{S}^*M\otimes\Lambda\mathbf{S}^*M\big)$ and since the diagram
\[
\xymatrix{
\sect{\Lambda T^*M}\ar@{^{(}->}[rr]^-{\iota}\ar[d]^{d} && \sect{\Lambda T^*M\otimes\Sym^{<\infty}\mathbf{S}^*M\otimes\Lambda\mathbf{S}^*M} \ar[d]^{d_\psi}\\
\sect{\Lambda T^*M}  \ar@{^{(}->}[rr]^-{\iota} && \sect{\Lambda T^*M\otimes\Sym^{<\infty}\mathbf{S}^*M\otimes\Lambda\mathbf{S}^*M}
}
\]
must commute we must get that the induced operator on superforms is $\{d,(-1)^N\id\smalltriangleright\}$, because it must increase the degree of a superform while being the identity on generators. By the precedeeing lemaama we know this operator equals $-\{\id\smalltriangleleft,\id\smalltriangleright\}$ whose kernel is $\sect{\Lambda T^*M}$; by ama \ref{lema:homology} we get the result, to wit: the cohomology of $d_\psi$ equals the cohomology of $d$. 
\end{proof}

\appendix
\chapter{Linear differential operators}\label{app:diff-ops}

\noindent This appendix is concerned with the basic facts of differential operators we use in this work. A good introduction can be found in \cite[chapter 5]{geometry-I}. We start by defining linear differential operators. Then we move on to prove that a universal object (the jet bundle) can be constructed by means of which all differential operators become bundle maps; we also show that jets are generalisations of Taylor polynomials.

\section{Basic concepts}
\begin{defn}\label{def:diff-op}
Let $E$ and $F$ be the total spaces of two smooth vector bundles over a smooth manifold $M$. A \textbf{linear differential operator of order $n\geq 0$} is an $\R$-linear map 
\[
D:\sect{E}\to\sect{F}
\]
such that the commutator 
\begin{equation}\label{eq:definition-commutator}
[D,f](\psi):=D(f\psi)-fD(\psi), 
\end{equation}
with $f\in\C(M)$ and $\psi\in\sect{E}$ is zero when iterated with $f_0,\ldots,f_n$ smooth functions on $M$. 
\end{defn}

Observe that a differential operator of order $0$ is $\C(M)$-linear and is therefore a bundle morphism.

\begin{lema}\label{lema:formula-conmutador-iterado}
Let $\setupla[k]{f}$ be smooth functions on $M$, $k>1,$ and define the set $K:=\set{1,\ldots,k}$. The iterated commutator of  $D$ and the above $k$ smooth functions is given by
\begin{equation}\label{eq:conmutador-formula}
[\ldots[[D,f_1],f_2],\ldots,f_{k}](\eta)=\sum\limits_{A\subseteq K}(-1)^{\#(A)}f_AD(f_{K-A}\eta),
\end{equation}
where $\eta$ is any section of $E$, $\#(A)$ is the cardinality of $A$ and we set
\begin{equation*}
f_A=\left\{
\begin{matrix}
\prod\limits_{a\in A}f_a,  & A\neq\emptyset\\
\empty				&   \empty \\
1,					& A=\emptyset
\end{matrix}
\right.
\end{equation*}
\end{lema}
\begin{proof}
The proof is by induction on $k$. If $k=2$ then we compute
\[
\begin{split}
[[D,f_1],f_2](\eta) &= [D,f_1](f_2\eta)-f_2[D,f_1](\eta)\\
					&= D(f_1f_2\eta)-f_1 D(f_2\eta)-f_2 D(f_1\eta)+f_1 f_2 D(\eta)\\
					&= \sum\limits_{A\subseteq\set{1,2}}(-1)^{\#(A)}f_A D(f_{\set{1,2}-A}\eta).
\end{split}
\]
Now let's supose formula \eqref{eq:conmutador-formula} holds for all integers $l<k$. Let $L:=\set{1,\ldots,k-1}$. We compute
\[
\begin{split}
[[\cdots [[D,f_1],f_2],\ldots,f_{k-1}],f_k](\eta) &=\sum\limits_{A\subseteq L}(-1)^{\#(A)}f_AD(f_{L-A}f_{k}\eta)\\
											&-\sum\limits_{A\subseteq L}(-1)^{\#(A)}f_A f_k D(f_{L-A}\eta)
\end{split}
\]
If $A\subseteq L$ is non-empty and we set $A=\setupla[r]{\mu}$ and $L-A=\setupla[s]{\nu}$, then the above formula is equivalent to
\begin{equation}
\begin{split}
[[\cdots [[D,f_1],f_2],\ldots,f_{k-1}],f_k](\eta) &=\sum\limits_{r+s=k-1}(-1)^rf_{\mu_1}\cdots f_{\mu_r}D(f_{\nu_1}\cdots f_{\nu_s}f_k\eta)\\
											&-\sum\limits_{r+s=k-1}(-1)^rf_{\mu_1}\cdots f_{\mu_r}f_kD(f_{\nu_1}\cdots f_{\nu_s}\eta)\\
											&-f_kD(f_1\cdots f_{k-1}\eta)+D(f_1\cdots f_k\eta)\\
											&=\sum\limits_{A\subseteq K}(-1)^{\#(A)}f_AD(f_{K-A}\eta),
\end{split}
\end{equation}
where $K:=\set{1,\ldots,k}$, because $-(-1)^{k-1}D(f_1\cdots f_k\eta)=(-1)^{k}D(f_1\cdots f_k\eta)$.
\end{proof}

We keep our notation $K=\set{1,\ldots,k}$ in the proof below.

\begin{prop}\label{prop:conmutador-iterado}
If $f_1,\ldots,f_k$ are smooth functions on $M$ then
\begin{equation*}
[\cdots[[D,f_1],f_2],\ldots,f_{k}]=[\ldots[[D,f_{\sigma(1)}],f_{\sigma(2)}],\ldots,f_{\sigma(k)}]
\end{equation*}
for any permutation $\sigma\in S_k$.
\end{prop}
\begin{proof}
The action of $S_k$ on $\mathcal{P}(K)$ (the power set of $K$) is transitive on subsets of a given cardinality, therefore formula \eqref{eq:conmutador-formula} is invariant under permutations of the set $\setupla[k]{f}$.
\end{proof}

\begin{nota}
The above proposition allows us to write $[D;f_1,\ldots,f_k]$ for the iterated commutator.
\end{nota}

\begin{defn}\label{def:principal-symbol}
Let $D:\sect{E}\to\sect{F}$ be a linear differential operator of order $k\geq 0$ and let a set $\setupla[k]{f}$ of smooth functions on $M$ be given. The \textbf{principal symbol} of $D$ evaluated in $f_1\cdots f_k$ is
\begin{equation}\label{eq:principal-symbol}
\sigma_D(f_1\cdots f_k):=[D;f_1,\cdots, f_k]
\end{equation}
\end{defn}

Since the commutation bracket is symmetric in the smooth functions we get the following:

\begin{prop}
The principal symbol of a linear differential operator $D$ of order $k\geq 0$ is a section of $\Sym^k T^*M\otimes E^*\otimes F$.
\end{prop}

\section{Jets of sections}
\noindent We now discuss the concept of jets. These generalise, in local coordinates at least, the Taylor polynomials of smooth transformations; their importance, nevertheless, lies in the fact that the bundle of jets is a universal construction for factoring smooth differential operators. Our setting is a smooth vector bundle $E$ over the smooth manifold $M$.

\begin{defn}
Let $k$ be a non-negative integer and $p$ a point of $M$. Two sections $\eta$ and $\bar{\eta}$ are in \textbf{contact up to order $k$ at the point $p$} if given bundle chart $(U,x)$ around $p$ then
\[
\partial^\mu_p(\eta)=\partial^\mu_p(\bar{\eta}),\quad\abs{\mu}\leq k,\quad \partial^\mu_p:=\left.\left(\frac{\partial^{\abs{\mu}}}{\partial x^\mu} \right)\right|_p
\]
We denote this relation by $\eta\sim_{k,p}\bar{\eta}$.
\end{defn}

It is clear that this relation is independent of the bundle coordinates and that it is an equivalence relation.
The equivalence class of a section $\eta$ is called the \textbf{$k$-jet at $p$} of $\eta$ and it is denoted by $\jet^k_p(\eta)$. If $\eta$ and $\bar{eta}$ have the same $k$-jet at $p$ then their difference, when written down in local coordinates, is a homogeneous polynomial of degree $k+1$ in those coordinates. So the expression of $\jet^k_p\eta$ is the Taylor polynomial of $\eta$ in coordinates around the point $p$. It is in this sense that jets are generalisations of Taylor polynomials.

We now define the \textbf{space of $k$-jets at the point $p$} as the space
\begin{equation}\label{eq:jets-at-p}
\Jet^k_p(E)=\sect{E}\big /\sim_{k,p},
\end{equation} 
and of course the \textbf{bundle of $k$-jets of $E$} as the vector bundle
\begin{equation}\label{eq:jet-bundle}
\Jet^k(E)=\bigsqcup_{p\in M}\Jet^k_p(E).
\end{equation} 
Each of these bundle carries a natural map $\ev:\Jet^k(E)\to E$ given by $\ev(\jet^k_p\eta)=\eta(p)$. We use a special notation for jets of smooth functions:
\begin{equation}\label{eq:jets-functions}
\Jet^k(M):=\Jet^k(M\times\R)
\end{equation}

We now give an invariant characterisation of jet spaces:

\begin{prop}\label{prop:ideal-jet}
Let $I_p$ be the ideal of $\C(M)$ consisting of functions that vanish at $p$. Then $\Jet^k_p(E)\cong\sect{E}\big / (I_p^{k+1}\cdot\sect{E})$.
\end{prop}

This is a consequence of the following result, which can be consulted in standard books on vector calculus:

\begin{lema}[Taylor formula with residue]
Let $h:\R^m\to\R^n$ be smooth. The Taylor polynomial of degree $k$ of $h$ around $0$ is given by
  \begin{eqnarray*}
   h(\,x^1,\ldots,x^m\,)
   &=&
   \sum_{r=0}^k\;\frac1{r!}\,\sum_{\mu_1,\ldots,\mu_r=1}^m
   x^{\mu_1}\,\ldots\,x^{\mu_r}\;\frac{\partial^rh}
   {\partial x^{\mu_1}\ldots\partial x^{\mu_r}}(\,0,\ldots,0\,)
   \\
   &+&
   \frac1{k!}\sum_{\mu_0,\ldots,\mu_k=1}^m
   x^{\mu_0}\,\ldots\,x^{\mu_k}\,\int_0^1\frac{\partial^{k+1}h}
   {\partial x^{\mu_0}\ldots\partial x^{\mu_k}}(\,tx^1,\ldots,tx^m\,)
   \,(1-t)^k\,dt
  \end{eqnarray*}
\end{lema}

\begin{proof}[Proof of proposition \ref{prop:ideal-jet}]
Let $\prodtupla[k+1]{f}\eta$ be an element of $I^{k+1}_p\cdot\sect{E}$ and choose a trivialisation of $E$ around $p$ with coordinates $\tupla[m]{x}$. If $l\leq k+1$ then the generalised Leibniz identity for the product $\prodtupla[k+1]{f}$ implies 
\[
\frac{\partial^L}{\partial x^L}(\prodtupla[k+1]{f})\eta=0
\]
for all multi-indices $L$ of lenght $l$, because at least one of the functions $\listupla[k+1]{f}$ will appear in the expansion for the derivatives evaluated at $p$, where they all vanish. So we get $\prodtupla[k+1]{f}\eta\sim_{k+1,p} 0$ and therefore $\jet^k_p(\prodtupla[k+1]{f}\eta)=0$. Let $\psi\colon\sect{E}\big / (I_p^{k+1}\cdot\sect{E})\to\Jet^k_p(E)$ be defined as $\psi(\eta+I_p^{k+1}\cdot\sect{E})=\jet^k_p(\eta)$. Then $\psi$ is well defined by the coordinate-independence of $\jet^k_p(\eta)$, it's evidently $\R$-linear and it's surjective because of the definition of $\Jet^k_p(E)$. To see that it is injective note that any section $\eta$ is, when written down in local coordinates, a smooth function between open sets of vector spaces. By the Taylor formula above we know $\psi(\eta+I_p^{k+1}\cdot\sect{E})=0$ if and only if it's Taylor polynomial vanishes up to order $k$ at $p$, and therefore $\jet^k_p(\eta+I_p^{k+1}\cdot\sect{E})=0$ if and only if $\eta\equiv 0 \mod I_p^{k+1}$; so $\psi$ is injective.
\end{proof}

\begin{cor}
The map $\jet^k:\sect{E}\to\sect{\Jet^k E}$ assigning to every section $\eta$ its $k$-jet $\jet^k(\eta)$ punctually, is a differential operator of order $k$.
\end{cor}
\begin{proof}

We must show that for given $f_0,\ldots,f_k$ smooth functions on $M$ we have
\[
[\jet^k;f_0,\ldots,f_k]\equiv 0.
\]
Formula \eqref{eq:conmutador-formula} in this case turns into
\[
[\jet^k;f_0,\ldots,f_k](\eta)=\sum_{A\subseteq K}(-1)^{\#(A)}f_A\jet^k(f_{K-A}\eta)
\]
where here $K$ stands for the set $\set{0,\ldots,k+1}$. By the previous proposition we know $\jet^k(f_{K-A}\eta)=f_{K-A}\eta+I_p^{k+1}\cdot\sect{E}$ so the summands in the formula above are equal to permutations of $f_0\dotsm f_k\eta$ times a sign; thus the formula is equivalent to
\[
[\jet^k;f_0,\dotsc,f_k]=\sum_{\sigma\in S_{k+1}}\sgn\sigma\cdot f_{\sigma(0)}\dotsm f_{\sigma(k)}\eta + I_p^{k+1}\cdot\sect{E}.
\]
Since all the $f$'s commute and each permutation appears with a plus sign the same times it appears with a minus sign the above equation vanishes identically. Thus $\jet^k$ is a differential operator of order $k$.
\end{proof}

The advantage of the bundle of jets is that it is a space that universally factorizes differential operators. To wit:

\begin{teo}\label{teo:jet-universal}
Let $D:\sect{E}\to\Gamma(\widetilde E)$ be a differential operator of order $k$. Then there is a unique bundle morphism $ \widehat{D}:\Jet^k E\to \widetilde{E}$ such that 
\[
\xymatrix{
\sect{E} \ar[r]^-{D} \ar[d]_{\jet^k}&\Gamma(\widetilde{E})\\
\sect{\Jet^k E}\ar[ur]_{\sect{\widehat{D}}}
}
\]
commutes, where $\sect{\widehat{D}}$ is the associated morphism via the functor $\Gamma$.
\end{teo}
\begin{proof}
Since $\jet^k$ is a differential operator of order $k$ we have to prove that $\jet^k_p\eta\mapsto D(\eta)_p=\ev(\jet^k_p(D\eta))$ is a bundle morphism; that is, if $f$ is a smooth function on $M$ we should get $\widehat{D}(\jet^k(f\eta))=f{D}(\eta)$. Indeed, with the latter definition of $\widehat D$:
\[
\begin{split}
\widehat{D}\left(\jet^k(f\eta)\right)_p%
&=D((f\eta))_p+I^{k+1}_p\sect{\tilde{E}}\\
&=f(p)D(\eta)_p+I^{k+1}_p\sect{\tilde{E}}\\
&=\jet^k_p(f(p)D(\eta))\\
&=f(p)\jet^k_p(D(\eta))
\end{split}
\]
and the last section satisfies $\ev(f(p)\jet^k_p(D\eta))=f(p)D(\eta)_p$ which is exactly the definition of $\widehat D$.
\end{proof}

\begin{defn}
Let $D$ be a differential operator of order $k$. The map $\widehat{D}$ of theorem \ref{teo:jet-universal} is called the \textbf{total symbol} of $D$. It's denoted by $\sigma^{\text{total}}(D)$.
\end{defn}

\subsection{Expressions in local coordinates}
\noindent Since jets are generalisations of Taylor polynomials it is natural to seek for an expression of jets that reflects this fact. To do so let $\nabla$ be a connection on the vector bundle $E$ and $D$ a torsion-free connection on the tangent bundle of $M$. 

\begin{defn}
The iterated covariant derivatives of a section $\eta$ of $E$ are given recursively by $\nabla^0 \eta=\eta$ and
\begin{equation}
\nabla^{k+1}_{X_0,\ldots ,X_k}\eta := \nabla_{X_0}(\nabla^k_{X_1,\ldots, X_k}\eta)-\sum_{\mu=1}^k \nabla^k_{X_1,\ldots,D_{X_0}X_\mu,\ldots, X_k}\eta
\end{equation}
\end{defn}

Using the iterated covariant derivatives we define an operator
\begin{equation}\label{eq:simetrization}
J^{\nabla,D,l}_{X_1\cdots X_l}(\eta)=\frac{1}{l!}\sum_{\sigma\in S_l}\nabla^{l}_{X_{\sigma(1)},\ldots,X_{\sigma(l)}}(\eta)
\end{equation}
which is symmetric on the vector field arguments. We therefore have $J^{\nabla, D, l}\in\sect{\Sym^{\leq l}{T^*M}\otimes E}$. Here $\Sym^{\leq l}T^*M$ denotes the space $\Sym^0 T^*M\oplus\Sym^1 T^*M\oplus\cdots\Sym^l T^*M$. It is clear that this is a ``polynomial'' on vector fields, since it's a symmetric form on them. 

\begin{prop}\label{prop:symmetric-jets}
The map $\jet^k \eta\mapsto J^{\nabla, D , 0}\eta+J^{\nabla, D , 1}\eta+\cdots+J^{\nabla, D , k}\eta$ is a lineal isomorphism of bundles $\Psi:\Jet^k E\to\Sym^{\leq k}(T^*M)\otimes E$.
\end{prop}
\begin{proof}
This is just the Taylor formula for the connections $\nabla$ and $D$, since any connection gives a trivialisation when properly restricted to an open subset of the base manifold. Since Taylor  polynomials are uniquely determined by both their jets and their expression in local coordinates the result follows.
\end{proof}

Recall from \eqref{eq:jets-functions} that $\Jet^k(M)$ denotes the jet bundle of smooth functions on $M$.
	
\begin{cor}\label{cor:cotangent-1-jet}
There is a natural isomorphism of bundles $\Jet^1(M)\cong\R\times T^*M$.
\end{cor}
\begin{proof}
The isomorphism is clear. The naturality comes from the fact that $M\times\R$ has a natural connection, namely the exterior derivative.
\end{proof}

In order to avoid confusion we henceforth switch to denoting a differential operator by $L$ instead of $D$, letting the latter denote a torsion-free connection on the tangent bundle of the base manifold.

\begin{defn}\label{def:total-symbol}
The \textbf{polynomial total symbol} of a differential operator $L:\sect{E}\to\Gamma(\widetilde E)$ of order $k$ is the polynomial $\sigma^{\text{total}}(L,\cdot)\in\sect{\Sym^{\leq k}T^*M\otimes E^*\otimes\widetilde{E}}$ associated to $D$ via the isomorphism of proposition \ref{prop:symmetric-jets}.
\end{defn}

That is, the total symbol is a section of $\widetilde E$ of the form
\[
\sigma^{\text{total}}(L,X_1\cdots X_k)(\eta)=\sum_{\mu=0}^k\sum_{\tau\in S_\mu} J^{\nabla, D, \mu}_{X_{\tau(1)}\cdots X_{\tau(\mu)}}(\eta)
\]

\section{Differential operators along smooth maps}
All of the above is valid for a differential operator between sections of two vector bundles over the same manifold. A more general situation is the following:

\begin{defn}\label{def:dif-op-along}
Let $\phi:M\to N$ be a smooth map and let $\pi:E\to M$ and $\widetilde{\pi}:\widetilde{E}\to N$ be vector bundles. A \textbf{linear differential operator of order $k\geq 0$ along $\phi$} is an $\R$-linear map 
\[
D:\Gamma(\widetilde{E})\to\sect{E}
\]
such that the commutator
\[
[D,f](\eta):=D(f\eta)-(f\circ\phi)D(\eta)
\]
is identically zero when iterated with $k+1$ smooth funcions on $N$.
\end{defn}

The principal symbol of $D$ is defined in the same way it was defined for one smooth manifold and the identity map:
\begin{equation}
\sigma_D(f_0,\ldots, f_{k-1})=[D;f_0,\ldots, f_{k-1}]
\end{equation}

Since jet bundles factorise differential operators, it'd be natural to suppose the commutativity of the following diagram:
\[
\xymatrix{%
\sect{\widetilde E} \ar[r]^-{D} \ar[d]_{\jet^k}&\Gamma({E})\\
\sect{\phi^*\Jet^k\widetilde{E}}\ar[ur]_{\sect{\sigma^{\text{total}}(D)}}%
}
\]
Here $\sigma^{\text{total}}(D)$ denotes the total symbol along $\phi$. Its invariant expression is
\begin{equation}
\sigma^{\text{total}}(D,f_0,\ldots, f_{k-1}):=D+\sum_{{\scriptstyle r=0}\atop{\scriptstyle 0\leq\mu_1,\ldots\mu_r\leq k-1}}^{k-1} [D;f_{\mu_1},\ldots, f_{\mu_r}]
\end{equation}
The proof of the above equality is the same, with the appropriate change in notation, as in theorem \ref{teo:jet-universal}. 



\chapter{The Cartan-Poincaré Lemma}\label{app:cartan-poincare}
\noindent In this appendix we prove the Cartan-Poincaré Lemma. We use this result in the next appendix to prove lemma \ref{lemma:derivations-isomorphism}. We first give the setting for the proof and we prove some lemmas that are of independent interest.

Let $V$ and $W$ be finite-dimensional vector spaces and let $F\colon V\to W$ and $G\colon W\to V$ be linear maps. Let $\setupla{v}$ and $\setupla[m]{w}$ be bases for $V$ and $W$ respectively, with dual bases $\setupla{{dv}}$ and $\setupla[m]{{dw}}$. When $\alpha$ is a linear functional we denote by $\alpha\ins$ the operator of evaluating the form $\alpha$, that is $\alpha\ins(x)=\alpha(x)$ for all vectors $x$; we also use the same symbol for the extension of this operator to the symmetric and exterior algebras. We will also need the following formulas relating the operators of multiplication and insertion on the symmetric and exterior algebras:

\begin{ccr}
If $U$ is a finite-dimensional vector space then the following relations hold:
\begin{equation}\label{eq:ccr}
\begin{split}
[v\cdot,\widetilde v\cdot]		&=0\\
[\alpha\ins,\widetilde\alpha\ins]&=0\\ 	
[\alpha\ins ,v\cdot]			&=\alpha(v)\cdot						
\end{split}
\end{equation}
on $\Sym U$ (cannonical commutation relations);
\begin{equation}\label{eq:car}
\begin{split}
\{v\wedge,\widetilde v\wedge \}	&=0\\
\{\alpha\ins,\widetilde \alpha\ins \}&=0\\
\{v\wedge,\alpha\ins \}			&=\alpha(v)\wedge						
\end{split}
\end{equation}

on $\Lambda U$ (cannonical anticommutation relations), where $v\cdot$ and $v\wedge$ denote the multiplications on $\Sym U$ and $\Lambda U$ respectively, greek letters denote elements of $U^*$, and $\{X,Y\}=XY+YX$ is the anticommutator.

\end{ccr}

\begin{prop}
The operators 
\[
\begin{split}
\der_{GF}:=\sum_{\mu=1}^n GF(v_\mu)\cdot\circ dv_\mu\ins\\
\der_{FG}:=\sum_{\mu=1}^m FG(w_\mu)\wedge\circ dw_\mu\ins
\end{split}
\]
are derivations in $\Sym V$ and $\Lambda W$ respectively.
\end{prop}
\begin{proof}
Since the verifications of the Leibniz rule are very similar computations, we'll carry out the one for $\der_{FG}$. Let's first recall that $FG(w_\mu)$ is an odd element of $\Lambda W$ and that $dw_\mu\ins$ is an odd derivation. With these facts in mind we compute:

\[
\begin{split}
\der_{FG}(\alpha\wedge\beta)%
&=\sum_{\mu=1}^{m} FG(w_\mu)\wedge\circ dw_\mu\ins(\alpha\wedge\beta)\\
&=\sum_{\mu=1}^m   FG(w_\mu)\wedge((dw_\mu\ins\alpha\wedge\beta-(-1)^{\parity{\alpha}})\alpha\wedge(dw_\mu\ins\beta))\\
&=\sum_{\mu=1}^m FG(w_\mu)\wedge(dv_\mu\ins\alpha)\wedge\beta+(-1)^{2\parity{\alpha}}\alpha\wedge FG(w_\mu)\wedge(dw_\mu\ins\beta)\\
&=\der_{FG}(\alpha)\wedge\beta+\alpha\wedge\der_{FG}(\beta).\qedhere
\end{split}
\]
\end{proof}

Define the spaces
\begin{equation}\label{eq:subspaces}
A^{\bullet,\circ}:=\Sym^\bullet V\otimes \Lambda^\circ W
\end{equation}
which comprise the bigraded algebra $\Sym V\otimes\Lambda W$. We now consider the operators
\begin{subequations}\label{eq:cartan-poincare-operators}
\begin{align}
d_F :=\sum_{\mu=1}^n dv_\mu\ins\otimes F(v_\mu)\wedge:A^{\bullet,\circ}\to A^{\bullet-1,\circ+1}\\
d^*_G :=\sum_{\mu=1}^m G(w_\mu)\cdot\otimes dw_\mu\ins:A^{\bullet,\circ}\to A^{\bullet+1,\circ-1}
\end{align}
\end{subequations}
which we call the \textbf{Cartan-Poincaré operators}.

\begin{prop}\label{prop:cartan-poincare-operators}
The Cartan-Poincaré operators are boundary operators, that is $d_F^2=(d^*_G)^2=0$; furthermore, they satisfy $\{d_F,d^*_G\}=\der_{GF}\otimes\id_{\Lambda W}+\id_{\Sym V}\otimes\der_{FG}$.
\end{prop}
\begin{proof}
The computation makes use  of the cannonical commutation and anticommutation relations (abbreviated CCR and CAR respectively). We do it explicitly for $d_F$:
\[
\begin{split}
d^2_F%
&=\sum_{\mu,\nu} dv_\mu\ins\circ dv_\nu\ins\otimes F(v_\mu)\wedge\circ F(v_\nu)\wedge\\
&=\frac{1}{2}\sum_{\mu,\nu}\big( dv_\mu\ins\circ dv_\nu\ins\otimes F(v_\mu)\wedge F(v_\nu)\wedge+dv_\mu\ins\circ dv_\nu\ins\otimes F(v_\nu)\wedge F(v_\mu)\wedge \big)\\
&=\frac{1}{2}\sum_{\mu,\nu}dv_\mu\ins\circ dv_\nu\ins\otimes\{F(v_\mu)\wedge F(v_\nu)\wedge \}\\
&=0 \qquad\qquad\text{(CCR)}
\end{split}
\]
which proves the first claim.

As for the second claim, we compute:
\[
\begin{split}
\{d_F,d^*_G\}%
&=\sum_{\mu,\nu} \Big(dv_\mu\ins\otimes F(v_\mu)\wedge\circ G(w_\nu)\cdot\otimes dw_\nu\ins+G(w_\nu)\cdot\otimes dw_\nu\ins\circ dv_\mu\ins\otimes F(v_\mu)\wedge\Big)\\
&=\sum_{\mu,\nu} dv_\mu\ins\circ G(w_\nu)\cdot\otimes dw_\nu\circ F(v_\mu)\wedge + G(w_\nu)\cdot\circ dv_\mu\ins\otimes F(v_\mu)\wedge\circ dw_\nu\ins\\
&-G(w_\nu)\cdot\circ dv_\mu\ins\otimes F(v_\mu)\circ dw_\nu\ins+G(w_\nu)\cdot\circ dv_\mu\ins\otimes F(v_\mu)\wedge\circ dw_\mu\ins\\
&=\sum_{\mu,\nu} [dv_\mu\ins,G(w_\nu)\cdot]\otimes F(v_\mu)\wedge\circ dw_\nu\ins + G(w_\nu)\cdot\circ dv_\mu\ins\otimes\{ F(v_\mu)\wedge,dw_\nu\ins \}\\
&= \sum_{\mu,\nu} dv_\mu(G(w_\nu))\id\otimes F(v_\mu)\wedge\circ dw_\nu\ins + G(w_\nu)\cdot\circ dv_\mu\ins\otimes dw_\nu(F(v_\mu))\\
&=\sum_\nu \id\otimes FG(w_\nu)\wedge\circ dw_\nu\ins + \sum_\mu GF(v_\mu)\cdot dv_\mu\ins\otimes\id\qedhere
\end{split}
\]
\end{proof}

\begin{lema}\label{lema:homology}
Let $F\colon V\to W$ be a linear map, $C$ a subspace of $V$ such that $V=\ker F\oplus C$ and $G\colon W\to V$ a linear map such that $GF|_{C}=\id$. Let $Z$ be a subspace of $W$ such that $W=Z\oplus\im F$ If $\Delta=\{d_F,d_G^*\}$ then 
\begin{equation*}
H^{\bullet,\circ}(d_F)=H^{\bullet,\circ}(d_F|_{\ker\Delta})
\end{equation*}
\end{lema}

\begin{proof}
Due to the sum decompositions $V=\ker F\oplus C$ and $W=\im F\oplus Z$ we know $V/C$ and $Z$ are systems of representatives for $\ker F$ and $\coker F$ respectively; that is, every element in those spaces can be written as $v+C$ (or $w+Z$).

Now, the operator $\Delta:=\{ d_F, d^*_G\}$ preserves the subspaces
\begin{equation}\label{eq:eigensubspaces}
U^{\bullet,\circ}_{k,l}:=(\Sym^k(C)\otimes\Sym^{\bullet-k} \ker F)\otimes(\Lambda^l(\im F)\otimes\Lambda^{\circ - l}Z)
\end{equation}

\begin{quotation}
\begin{claim}
$\Delta$ is diagonalisable with eigenvalues $k+l$ and corresponding proper subspaces $U^{\bullet,\circ}_{k,l}$. 
\end{claim}

\noindent In order to prove this claim recall that proposition \ref{prop:cartan-poincare-operators} implies that 
$$\Delta=\id\otimes\der_{FG}+\der_{GF}\otimes\id$$
so $\Delta$ is a derivation of the algebra $\Sym V\otimes\Lambda W$; restricted to the subspace $U_{k,l}^{\bullet,\circ}$ the operator $\Delta$ acts as the extension as a derivation of the map $GF$ on the symmetric factor and as the corresponding extension of $FG$ on the exterior factor (due to lemmma \ref{lemma:derivations}; a similar result of course holds for the symmetric algebra). Now since any element of $\Sym^k(\ker F)$ is a linear combination of monomials whose factors are anihilated by $F$ we have that the derivation extension $\der_{GF}$ anihilates this space; on the other hand, $\der_{GF}$ is the derivation extension of the identity map on $C$ (because of the very definition of this subspace of $V$) so it acts as the operator $s\id$ on the subspace $\Sym^s(C)$ (this is a result analogous to the corollary to lemma \ref{lemma:derivations} on the symmetric algebra). Since
\[
\Sym^k(V)=\bigoplus_{r+s=k}\Sym^r(\ker F)\otimes\Sym^s(C)
\]
we obtain $\der_{GF}|_{\Sym^k V}=k\id$. A similar analysis yields the corresponding result for $\Lambda^l (\im F)\otimes\Lambda^{\circ-l} Z$. The claim follows.\hfill $\blacksquare$
\end{quotation}

We therefore have a decomposition of $d^{k,l}_F\colon A^{k,l}\to A^{k-1,k+l}$ in eigensubcomplexes of
\[
\Delta:A^{k,l}\to A^{k,l}.
\]
Now $[\Delta,d_F]=0$. Let $\eta$ be an eigenvector for $\Delta$ with eigenvalue $\lambda\neq 0$ such that $d_F(\eta)=0$, then:
\begin{equation}\label{eq:delta-exactos}
\eta=\frac{1}{\lambda}\Delta(\eta)=\frac{1}{\lambda}d_F d^*_G \eta  + d_G^* d_F \eta= d_F\left( \frac{1}{\lambda}d^*_G \eta \right)
\end{equation}
because of the definition of $\Delta$. This implies that the $d_F$-complex is exact on the eigensubcomplexes of $\Delta$ except on $\ker\Delta$. As a consequence, whatever cohomolgy there is, is bound to be on this latter subspace.
\end{proof}

Of course a similar result holds when given $G\colon W\to V$ for the operator $d^*_G$ and the proof is the same, mutatis mutandis.

We now state and prove the main result of this appendix:

\begin{cartan-poincare}\label{lem:cartan-poincare}
The homology and cohomology groups of the Cartan-Poincaré operators satisfy
\[
\begin{split}
H^{\bullet,\circ}(d_F)\cong \Sym^\bullet(\ker F)\otimes\Lambda^\circ(\coker F)\\
H_{\bullet,\circ}(d^*_G)\cong \Sym^\bullet(\coker G)\otimes\Lambda^\circ(\ker G)
\end{split}
\]
\end{cartan-poincare}

\begin{proof} 
Let $F\colon V\to W$ be a linear map and let $G$, $Z$ and $C$ be as in the hypotheses of lemma \ref{lema:homology}. Then equation \eqref{eq:delta-exactos} implies
$H^{\bullet,\circ}(d_F)=\Sym^\bullet \ker F\otimes\Lambda^\circ Z$. 

Now, $Z$ is a system of representatives for $\coker F$ and any other system of representatives produces a cohomologous class, so the map
\[
\begin{split}
\Sym^\bullet \ker F\otimes\Lambda^\circ Z 			& \to H^{\bullet, \circ}(d_F)\\
P\otimes (z_1+\im F)\wedge\cdots\wedge(z_k+\im F)	& \mapsto [P\otimes z_1\wedge\cdots\wedge z_k]
\end{split}
\]
is well defined and a linear isomorphism.

The result for a given $G\colon W\to V$ is analogous.
\end{proof}


\chapter{The composition algebra}\label{app:composition}
\noindent In this appendix we use the Cartan-Poincaré lemma to prove lemma \ref{lemma:derivations-isomorphism}. We do so by associating to an injective linear map $f\colon V\to\mathbf{S}$ an algebra very similar to the ones in which the Cartan-Poincaré operators act.

\begin{defn}
Let $\mathbf{S}$ be a finite-dimensional vector space. The \textbf{composition algebra} of $\mathbf{S}$ is
\begin{equation}\label{eq:composition-algebra}
\Lambda\mathbf{S}^*\otimes\mathbf{S}
\end{equation}
with product defined by $(\omega\otimes s)\cdot(\widetilde\omega\otimes\widetilde s)=\omega\wedge(s\ins \widetilde\omega)\otimes\widetilde s$
\end{defn}

It is quite evident that this product is not associative. Nevertheless it has a very interesting property, as a consequence of theorem \ref{thm:superderivations}:

\begin{lema}\label{lemma:bracket-composition}
Define, for elements of the composition algebra $\sigma\otimes s$ and $\hat{\sigma}\otimes\hat{s}$ the operation
\begin{equation}\label{eq:composition-bracket}
[\sigma\otimes s,\hat{\sigma}\otimes\hat{s}]=(\sigma\otimes s)\cdot(\hat{\sigma}\otimes\hat{s})+(-1)^{(\parity{\sigma}+1)(\parity{\hat{\sigma}}+1)}(\hat{\sigma}\otimes\hat{s})\cdot({\sigma}\otimes{s})
\end{equation}
Then, under the isomorphism $\Psi\colon\Lambda_{\bullet}\mathbf{S}^*\otimes\mathbf{S} \to\sder_{-\bullet}\Lambda\mathbf{S}^*$ given by $\Psi(\sigma\otimes s)(\omega)=\sigma\wedge(s\ins\omega)$ the Lie superbracket on $\sder\Lambda\mathbf{S}^*$ corresponds to the bracket operation above.
\end{lema}
\begin{proof}
Let $\omega$ be an element of $\Lambda\mathbf{S}^*$. We compute
\[
\begin{split}
[\sigma\otimes s,\hat{\sigma}\otimes\hat{s}]\omega%
		&=\sigma\wedge\big(s\ins(\hat{\sigma}\wedge\hat{s}\ins\omega)\big)+(-1)^{(\parity{\sigma}+1)(\parity{\hat{\sigma}}+1)}\hat{\sigma}\wedge\big(\hat{s}\ins \sigma\wedge s\ins\omega \big)\\
		&=\sigma\wedge(s\ins\hat{\sigma})\wedge(\hat{s}\ins\omega)+(-1)^{\parity{\hat{\sigma}}} \sigma\wedge\hat{\sigma}\wedge s\ins\hat{s}\ins\omega\\
		&-(-1)^{(\parity{\hat{\sigma}}+1)(\parity{\sigma}+1)}\hat\sigma\wedge(\hat{s}\ins\sigma)\wedge s\ins\omega +(-1)^{\parity{\sigma}\parity{\hat{\sigma}}+\parity{\hat{\sigma}}}\hat\sigma\wedge\sigma\wedge\hat{s}\ins s\ins\omega\\
		&=\big( \sigma\wedge(s\ins\hat{\sigma})\otimes\hat{s}-(-1)^{(\parity{\sigma}+1)(\parity{\hat{\sigma}}+1)} \hat{\sigma}\wedge(\hat{s}\ins\sigma)\otimes s \big)\omega
\end{split}
\]
which is exactly the formula for the action of elements $\sigma\otimes s$ as derivations of $\Lambda\mathbf{S}^*$.
\end{proof}

Note that we can write $(\sigma\otimes s)\cdot(\hat{\sigma}\otimes\hat{s})=\big((\sigma\otimes s)\star \hat{\sigma}\big)\otimes\hat{s}$, where $\star$ denotes the action of $\Lambda\mathbf{S}^*\otimes\mathbf{S}$ as derivations of the exterior algebra.


Let now $G\colon\mathbf{S}^*\to\Lambda\mathbf{S}^*$ be a linear  map; it can obviously be interpreted as an element of the composition algebra. Then, given a derivation $D$ of the exterior algebra we can form the product
\[
 D\cdot G\in\Lambda\mathbf{S}^*\otimes\mathbf{S}
\]
in the following way: if $\setupla{s}$ is a basis of $\mathbf{S}$ and $\setupla{ds}$ is its dual basis, then
\begin{equation}\label{eq:product-d-g}
D\cdot G=\sum D(G_{ds_k})\otimes s_k=D\cdot\left( \sum G_{ds_k}\otimes s_k \right)
\end{equation}
where $\sum G_{ds_k}\otimes s_k$ is the element in the composition algebra corresponding to $G$.

Now let $D\colon V\to\sder_-\Lambda\mathbf{S}^*$ be a linear map. In order to use the multiplication of the composition algebra for such a map, we consider it as an element of the algebra $\Sym V^*\otimes\Lambda\mathbf{S}^*\otimes\mathbf{S}$ of polynomials in $V$ with values in the composition algebra. The multiplication of this latter algebra is given by
\[
(p\otimes\sigma\otimes s)\cdot(q\otimes\hat{\sigma}\otimes\hat{s}) =pq\otimes\sigma\wedge(s\ins\hat{\sigma})\otimes\hat{s}
\]
so given a map $D$ as above, the product $D\cdot D$ can be thought of as a polynomial with values in the even derivations of $\Lambda\mathbf{S}^*$.

\begin{lema}
Suppose $D:V\to\sder_{-}\Lambda\mathbf{S}^*$ is a linear map. Then $[ D_v,D_{\widetilde v}]=0$ if and only if $D_v\cdot D_{\widetilde v}=0$ in the algebra $\Sym V^*\otimes\Lambda\mathbf{S}^*\otimes\mathbf{S}$ of polynomials in $V$ with values in the composition algebra.
\end{lema}
\begin{proof}
This is just a consequence of polarization:
\[
\begin{split}
(D\cdot D)_{v+\widetilde v}+(D\cdot D)_{v-\widetilde v}%
&=\frac{1}{4} \left([D_v,D_v]+[D_{\widetilde v},D_{v}]+[D_v,D_{\widetilde v}]+[D_{\widetilde v},D_{\widetilde v}]\right.\\
&-\left.[D_v,D_{v}]+[D_{\widetilde v},D_{v}][+D_v,D_{\widetilde v}]+[D_{\widetilde v},D_{\widetilde v}]\right)\\
&=[D_v,D_{\widetilde v}]
\end{split}
\]
\end{proof}

Finally we recall (corollary \ref{cor:odd-derivations-ins}) that $\sder\Lambda\mathbf{S}^*$ is a $\Lambda\mathbf{S}^*$-module freely generated by $\mathbf{S}$. So to the projection map $\pr\colon\Lambda\mathbf{S}^*\to\mathbf{S}^*$ corresponds a map
\begin{equation}\label{eq:projection-derivations}
\pr\colon\sder\Lambda\mathbf{S}^*\to\mathbf{S}	
\end{equation}
from the supermodule of superderivations to its space of generators.

We now restate and prove 

\theoremstyle{plain}
\newtheorem*{otro}{Lemma \ref{lemma:derivations-isomorphism}}

\begin{otro}
Let $\mathcal{A}$ be a free supercommutative finite-dimensional superalgebra and denote by $\mathbf{S}^*$ its space of generators. Let $D:V\to\der_{-}\mathcal{A}$ be a linear map such that the composition
\[
\xymatrix@1{
f:V\ar[r]^{D}&\sder_{-}\mathcal{A}\ar[r]^-{\pr}&\mathbf{S}&
}
\]
is injective and such that if $v,\widetilde{v}$ are in $V$ then $[D_v,D_{\widetilde v}]=0$. Then there exists an isomorphism $G:\Lambda\mathbf{S}^*\to\mathcal{A}$ of $\z$-graded algebras with unit such that $G$ induces the identity 
\[
\overline{G}:\mathbf{S}^*\to\mathcal{A}^{\geq 1}\big{/}\mathcal{A}^{\geq 2}=:\mathbf{S}^*
\]
and such that, for all $v\in V$ and all $\sigma\in\Lambda\mathbf{S}^*$
\begin{equation}\label{eq:d-o-g}
D_v(G\sigma)=G(f(v)\ins \sigma).
\end{equation}
Furthermore, up to the ideal generated by $\Lambda^3\ker(f^*)$ in $\Lambda^3\mathbf{S}^*$ the isomorphism $G$ is unique in the sense that if $G'$ is any other such isomorphism then 
\begin{equation}\label{eq:isomorfismo-unico-kernel}
\inv{G}\circ G':\Lambda\mathbf{S}^*\to\Lambda\mathbf{S}^*:\sigma\mapsto\sigma+\langle \Lambda^3\ker(f^*) \rangle
\end{equation}
\end{otro}
\begin{proof}
For simplicity we suppose that $\mathcal{A}$ is already an exterior algebra. Let's recall (theorem \ref{thm:superderivations}) that $\sder_{-}\Lambda\mathbf{S}^*\cong\Lambda_{+}\mathbf{S}^*\otimes\mathbf{S}$, so let $D\colon V\to\Lambda_+\mathbf{S}^*\otimes\mathbf{S}$ be a linear map of the kind considered in the statement of the lemma. Then $D\cdot D=0$ in the $\Z$-bigraded algebra $\mathcal{B}^{\bullet,\circ}:=\Sym^{\bullet} V^*\otimes\Lambda^{\circ}\mathbf{S}^*\otimes\mathbf{S}$. We shall prove that there exists a map
\begin{equation}
G\colon\mathbf{S}^*\to\Lambda_-\mathbf{S}^*
\end{equation}
such that $\pr^1\circ G=\id_{\mathbf{S}^*}$ and
\begin{equation}
D\cdot G=\pr\circ D,
\end{equation}
where $\pr_1:\Lambda_-\mathbf{S}^*\to\mathbf{S}^*$ and $\pr\colon\Lambda_-\mathbf{S}^*\otimes\mathbf{S}\to\mathbf{S}$ are the natural projections.

First of all, let's observe that left multiplication by an element $X$ of the algebra $\mathcal{B}$ is a boundary operator:
\[
X\cdot\colon\mathcal{B}^{\bullet,\circ}\to\mathcal{B}^{\bullet+1,\circ-1}.
\]
because for all $X$ we get $X^3=0$. Now as a linear map $f\colon V\to\Lambda\mathbf{S}^*\otimes\mathbf{S}$ it has the form $f=\alpha\otimes 1\otimes s$, where $\alpha\in V^*$ and $s\in\mathbf{S}$. So multipling by $f$ in the algebra $\mathcal B$ is tantamount to multipling by $\alpha\otimes 1\otimes s$.

In order to prove the existence of $G\colon\mathbf{S}^*\to\Lambda\mathbf{S}^*$ with the desired properties we make an ansatz for the maps $D$ and $G$ as

\begin{equation}\label{eq:ansatz}
\begin{split}
D&=D_0+D_1+D_2+\cdots\\
G&=G_0+G_1+G_3+\cdots
\end{split}
\end{equation}
where $G_0=\id_{\mathbf{S}^*}$ and $D_0=f$; also, each $G_\mu$ is an element of $\Lambda^{2\mu+1}\mathbf{S}^*\otimes\mathbf{S}$ and each $D_\mu$ of $V^*\otimes\Lambda^{2\mu}\mathbf{S}^*\otimes\mathbf{S}$. Since this will guarantee that $D_0\cdot G_0=f$ we need to show that $G$ can be chosen to satisfy
\begin{equation}\label{eq:ansatz2}
D_0\cdot G_\mu + D_1\cdot G_{\mu -1}+\cdots + D_\mu\cdot G_0=0.
\end{equation}
For this we'll make use of the Cartan--Poincaré lemma.

\begin{quotation}
\begin{claim}
The operator $D_0\cdot$ equals $d^*_{f^*}\otimes\id_{\mathbf{S}}$. 
\end{claim}
With the proposed decomposition we get $d^*_{f^*}=\sum f^*(ds_\mu)\otimes s_\mu\ins\otimes s_\mu$ for a basis $\setupla{s}$ of $\mathbf{S}$. Also, since $D_0=f$ we know $D_0=\sum f^*(ds_\mu)\otimes 1\otimes s_\mu$ for the same choice of basis. Now
\[
\begin{split}
D_0\cdot(p\otimes\alpha\otimes s)%
&=\sum f^*(ds_\mu)p\otimes s_\mu\ins\alpha\otimes s\\
&=\left(\sum f^*(ds_\mu)\otimes 1\otimes s_\mu	\right)\cdot p\otimes\alpha\otimes s\\
&=d^*_{f^*}\otimes\id_{\mathbf S}(p\otimes\alpha\otimes s)
\end{split}
\]
so the claim follows.\hfill$\blacksquare$
\end{quotation}
 
\noindent The Cartan--Poincaré lemma now implies
\[
H_{\bullet,\circ}(d_{f^*}^*)=%
\begin{cases}
0,&\bullet > 0\\
\Lambda^\circ\ker f^*\otimes\mathbf{S}, &\bullet=0
\end{cases}
\]
because $\ker f=0$. Since $D_0\cdot=f\cdot$ the fact that all homology groups vanish for polynomials of positive degree is equivalent to
\[
D_0\cdot X=Y\quad\text{implies}\quad D_0\cdot Y=0
\]
Our ansatz now requires
\[
D_0\cdot G_\mu=X\Leftrightarrow D_0\cdot X=0
\]
and furthermore the solution  $X$ is unique up to the kernel of the Cartan--Poincaré operator $D_0\cdot\colon\Lambda^{2\mu}\mathbf{S}^*\otimes\mathbf{S}\to V^*\otimes\Lambda^{2\mu-1}\mathbf{S}^*\otimes\mathbf{S}$, which is $\Lambda^{2\mu+1}(\ker f^*)\otimes\mathbf{S}$ because of the Cartan--Poincaré lemma. This fact avows for the second claim of the lemma.

Let $\mu\geq 1$. The lemma will be proved if we can show that, for chosen $G_1,\cdots,G_{\mu-1}$ satisfying
\begin{equation}\label{eq:ansatz-induction}
\begin{split}
D_0\cdot G_1+D_1\cdot\id_{\mathbf{S}}		&=0\\
D_0\cdot G_2+D_1\cdot G_1+D_2\cdot G_0		&=0\\
\empty										&\vdots\\
D_0\cdot G_{\mu-1}+D_1\cdot G_{\mu -2}+\cdots +D_{\mu -1}\cdot G_0&=0
\end{split}
\end{equation}
the choice of $G_\mu$ can be made  such that $D_0\cdot G_\mu=0$. Given the decomposition proposed in \eqref{eq:ansatz} for $G$ we have
\[
D_0\cdot G_\mu =-(D_\mu\cdot G_0+D_{\mu -1}\cdot G_1+\cdots D_1\cdot G_{\mu -1})=X
\]
because of equations \eqref{eq:ansatz2}, so we must have $D_0\cdot X=0$. To prove this last equation we observe that the sum
\[
\sum_{1\leq\alpha +\beta\leq\mu}D_\alpha\cdot(D_\beta\cdot G_{\mu-\alpha-\beta})
\]
contains all terms present in equations \eqref{eq:ansatz-induction} which we know to be zero and also contains the sum development for $-X$. Now
\begin{equation}\label{eq:finally}
\begin{split}
\sum_{1\leq\alpha +\beta\leq\mu}D_\alpha\cdot(D_\beta\cdot G_{\mu-\alpha-\beta})%
&=\frac{1}{2}\sum_{1\leq\alpha +\beta\leq\mu}D_\beta\cdot(D_\alpha\cdot G_{\mu-\alpha-\beta})+D_\alpha\cdot(D_\beta\cdot G_{\mu-\alpha-\beta})\\
&=\frac{1}{2}\sum_{1\leq\alpha+\beta\leq\mu}[D_\alpha,D_\beta]\cdot G_{\mu-\alpha-\beta}.
\end{split}
\end{equation}
We now see that
\[
\sum_{1\leq\alpha+\beta\leq\mu}[D_\alpha,D_\beta]=\sum(D_\alpha\cdot D_\beta+D_\beta\cdot D_\alpha)=0
\]
because all $D_\lambda\in V^*\otimes\Lambda^{2\lambda}\mathbf{S}^*\otimes\mathbf{S}$; also $D_\mu\cdot(D_0\cdot\id_{\mathbf{S}})=0$ is trivially true because $f\in V^*\otimes\R\otimes\mathbf{S}$ gives zero when the operator $D\cdot$ is applied to it. 

We now use the general formula
\[
A\cdot(B\cdot X)+(-1)^{\parity{A}\parity{B}}B\cdot(A\cdot X)=([A,B]\otimes\id_{\mathbf{S}})X
\]
to see that the development in \eqref{eq:finally} is zero. So now we've proved that our ansatz for $G$ yields the result.
\end{proof}

\renewcommand{\bibname}{References}
\bibliographystyle{amsalpha}

\bibliography{tesis}{}

\addcontentsline{toc}{chapter}{References}

\end{document}